\documentclass[10pt,reqno]{amsart}
\usepackage{amssymb,mathrsfs,color,mathtools,empheq, verbatim, epstopdf}
\usepackage{pinlabel}
\mathtoolsset{showonlyrefs}
\usepackage{hyperref} 

\usepackage{enumerate}
\usepackage{tensor}

\usepackage{graphicx}
\usepackage{xcolor} 
\usepackage{tensor}
\usepackage{slashed}
\usepackage{cite}

\usepackage[shortlabels]{enumitem}

\usepackage{geometry}

\def\bS {\mathbb{S}}

\newcommand{\tx}[1]{\mathrm{#1}}

\newcommand{\bs}[1]{\boldsymbol{#1}}
\newcommand{\vd}{\mathrm{d}}
\newcommand{\udr}{\,r\vd r}

\newcommand{\uln}[1]{{\underline{ #1 }}}

\definecolor{deepgreen}{cmyk}{1,0,1,0.5}

\newcommand{\A}{\mathcal{A}}
\newcommand{\B}{\mathcal{B}}
\newcommand{\E}{\mathcal{E}}

\newcommand{\G}{\mathcal{G}}
\newcommand{\HH}{\mathcal{H}}

\newcommand{\LL}{\mathcal{L}}

\newcommand{\Es}{\mathscr{E}}

\newcommand{\Hs}{\mathscr{H}}
\newcommand{\Vs}{\mathscr{V}}
\newcommand{\Ks}{\mathcal{K}}

\newcommand{\N}{\mathbb{N}}
\newcommand{\R}{\mathbb{R}}
\newcommand{\Sp}{\mathbb{S}}
\newcommand{\Z}{\mathbb{Z}}

\newcommand{\al}{\alpha}

\newcommand{\ga}{\gamma}
\newcommand{\de}{\delta}

\newcommand{\om}{\omega}
\newcommand{\la}{\lambda}
\newcommand{\lam}{\lambda} 
\newcommand{\te}{\theta}

\newcommand{\s}{\sigma}

\newcommand{\si}{\varsigma}

\newcommand{\De}{\Delta}

\newcommand{\La}{\Lambda}
\newcommand{\Lam}{\Lambda}

\newcommand{\p}{\partial}
\newcommand{\na}{\nabla}

\makeatletter

\newcommand{\Rmnum}[1]{\expandafter\@slowromancap\romannumeral #1@}
\makeatother

\newcommand{\ti}{\widetilde}
\newcommand{\ba}{\overline}

\newcommand{\U}{\underline}

\newcommand{\ang}[1]{\left\langle{#1}\right\rangle}
\newcommand{\abs}[1]{\left\lvert{#1}\right\rvert}

\newcommand{\ant}[1]{\begin{align*}\begin{split} #1 \end{split}\end{align*}}
\newcommand{\EQ}[1]{\begin{equation}\begin{split} #1 \end{split}\end{equation}}
\newcommand{\pmat}[1]{\begin{pmatrix} #1 \end{pmatrix}}

\newcommand{\Del}[1]{}

\numberwithin{equation}{section}

\newtheorem{thm}{Theorem}[section]
\newtheorem{cor}[thm]{Corollary}
\newtheorem{lem}[thm]{Lemma}
\newtheorem{prop}[thm]{Proposition}

\theoremstyle{remark}
\newtheorem{claim}[thm]{Claim}
\newtheorem{rem}[thm]{Remark}
\newtheorem{defn}[thm]{Definition}

\newcommand{\mand}{{\ \ \text{and} \ \  }}

\newcommand{\mif}{{\ \ \text{if} \ \ }}

\newcommand{\mas}{{\ \ \text{as} \ \ }}

\newcommand{\uD}{\operatorname{D}}

\newcommand{\rdr}{\, r\mathrm{d}r}

\newcommand{\rest}{\!\!\restriction}

\definecolor{green}{rgb}{0,0.8,0} 

\newcommand{\ud}{\mathrm{d}}

\newcommand{\eps}{\epsilon}

\newcommand{\bfd}{{\bf d}}

\newcommand{\calA}{\mathcal A}
\newcommand{\calB}{\mathcal B}

\newcommand{\calL}{\mathcal L}

\newcommand{\app}{\operatorname{app}}

\newcommand{\ula}{\underline{\lambda}}

\begin{document}

\title[Asymptotic expansion of two-bubble wave maps]{An asymptotic expansion of two-bubble wave maps}
\author{Jacek Jendrej}
\author{Andrew Lawrie}

\begin{abstract}
This is the first part of a two-paper series that establishes the uniqueness and regularity of a threshold energy wave map that does not scatter in both time directions. 

Consider the $\Sp^2$-valued equivariant energy critical wave maps equation on $\R^{1+2}$, with equivariance class $k \ge 4$.  It is known that every topologically trivial  wave map with energy less than twice that of the unique $k$-equivariant harmonic map $\bs Q_k$ scatters in both time directions. 
We study maps with precisely the threshold energy $\E =  2 \E(\bs Q_k)$.   

In this paper, we give a refined construction of a wave map with threshold energy that converges to a superposition of two harmonic maps (bubbles),  asymptotically decoupling in scale. We show that this two-bubble solution possesses $H^2$ regularity. We give a precise dynamical description of the modulation parameters as well as an expansion of the map into profiles. In the next paper in the series, we show that this solution is unique (up to the natural invariances of the equation) relying crucially on the detailed properties of the solution constructed here. Combined with our earlier work~\cite{JL1}, we can now give an exact description of every threshold wave map.

\end{abstract}

\thanks{J.Jendrej was supported by  ANR-18-CE40-0028 project ESSED.  A. Lawrie was supported by NSF grant DMS-1700127 and a Sloan Research Fellowship}

\maketitle

\section{Introduction}

This paper concerns wave maps  from the Minkowski space $\R^{1+2}_{t, x}$ into the two-sphere $\bS^2$, with k-equivariant symmetry. These are formal critical points of the Lagrangian action, 
\EQ{
\calA( \Psi)  = \frac{1}{2} \int_{\R^{1+2}_{t, x}} \Big( {-} \abs{\p_t \Psi(t, x)}^2 + \abs{\na \Psi(t, x)}^2 \Big) \, \ud x  \ud t, 
}
restricted to the class of maps  $\Psi: \R^{1+2}_{t, x} \to \Sp^2 \subset \R^3$ that take the form, 
\ant{
\Psi(t, r, \theta) = (u(t, r), k \te ) \hookrightarrow ( \sin u (t, r)\cos k\theta , \sin u(t, r) \sin  k \theta, \cos u(t, r)) \in \Sp^2 \subset \R^3,
}
for some fixed $k \in \N$. Here $u$ is the colatitude measured from the north pole of the sphere and the metric on $\Sp^2$ is given by $ds^2 = d u^2+ \sin^2 u\,  d \om^2$.  
We note that $(r, \te)$ are polar coordinates on $\R^2$, and $u(t, r)$ is  radially symmetric.

Wave maps are known  as nonlinear $\s$-models in high energy physics literature, see for example,~\cite{MS, GeGr17}. They satisfy a canonical example of a geometric wave equation -- it  simultaneously generalizes the free scalar wave equation to manifold valued maps and the classical harmonic maps equation to Lorentzian domains. The $2d$ case considered here is of particular interest, as the static solutions given by finite energy harmonic maps are amongst the simplest examples of topological solitons;  other examples include kinks in scalar field equations,  vortices in Ginzburg-Landau equations, magnetic monopoles, Skyrmions, and Yang-Mills instantons; see~\cite{MS}. 
Wave maps under $k$-equivariant symmetry possess intriguing features from the point of view of nonlinear dynamics, for example, bubbling harmonic maps, multi-soliton solutions, etc.,  in the relatively simple setting of a geometrically natural scalar semilinear wave equation. 
%
For a more thorough presentation of the physical or geometric content of wave maps, see e.g.,~\cite{MS, ShSt00, GeGr17}.


The Cauchy problem for $k$-equivariant wave maps is given by
\EQ{ \label{eq:wmk}
\p_{t}^2 u -   \p_{r}^2 u  - \frac{1}{r}  \p_r u + k^2  \frac{\sin 2 u}{2r^2} &= 0, \\ 
(u(t_0), \partial_t u(t_0)) &= (u_0, \dot u_0), \quad t_0 \in \R.
}
We will often use the notation, 
\EQ{ \label{eq:f-def} 
f( u):= k^2 \frac{\sin (2 u)}{2}.
}
The conserved energy is 
\EQ{ \label{eq:energy} 
\E( \bs u(t)) := 2 \pi \int_0^\infty \frac{1}{2}  \left((\p_t u)^2  + (\p_r u)^2 + k^2 \frac{\sin^2 u}{r^2} \right) \, r \, \ud r,
}
where we have used bold font to denote the vector 
$
\bs u (t) := (u(t), \p_t u(t)).
$
We will write vectors with two components as  $\bs v = (v, \dot v)$, noting that the notation $\dot v$ will not, in general, refer to a time derivative of $v$ but rather just to the second component of $\bs v$. 
We remark that~\eqref{eq:wmk} and the conserved energy~\eqref{eq:energy}  are invariant under the scaling
\EQ{ \label{eq:uscale} 
\bs u(t,  \cdot) \mapsto \bs u(t/ \la,  \cdot)_{\la}  = (u(t/ \la,  \cdot/ \la), \la^{-1} \p_t u( t/ \la, \cdot/ \la)), \qquad \lam >0.
}
which makes this problem energy critical.

 It follows from~\eqref{eq:energy} that any regular $k$-equivariant initial data $\bs u_0$  of finite energy  must satisfy $\lim_{r \to 0}u_0( r) = m\pi$ and  $\lim_{r \to \infty}u_0(r)=n\pi$  for some $m,n \in \Z$. Since the smooth wave map flow depends continuously on the initial data these integers are fixed over any time interval $t\in I$ on which the solution is defined. This splits the energy space into disjoint classes indexed by the pair $(m, n)$ and it is natural to consider the Cauchy  problem~\eqref{eq:wmk} within a fixed class. These classes are related to the topological degree of the full map $\Psi(t): \R^2 \to  \Sp^2$. In particular, $k$-equivariant wave maps with $(m, n) = (0, 0)$ correspond to topologically trivial maps $\Psi$, whereas those with $(m, n) = (0, 1)$  are degree $= k$ maps. 
 
 The unique (up to scaling) $k$-equivariant harmonic map is given explicitly by 
\EQ{ \label{eq:Q-def} 
Q(r) := 2 \arctan (r^k),
}
 and we write, $\bs Q := (Q, 0)$. We note that $Q(r)$ has degree $=k$ and  it is a standard fact that $\bs Q$  minimizes the energy amongst all degree $k$ maps (see, e.g.,~\cite{JL1}) and in particular amongst $k$-equivariant maps with $(m, n) = (0, 1)$. It is not hard to show that
$
 \E( \bs Q ) = 4  \pi k .
 $

 In this paper we  consider topologically trivial $k$-equivariant wave maps, i.e., those with data $\bs u_0$ that satisfies $\lim_{r \to 0} u_0(r)  =\lim_{r \to \infty} u_0(r) = 0$. The natural function space in which to consider such solutions in the energy space, which comes with the norm, 
 \EQ{
  \| \bs u_0 \|_{\HH}^2 := \|u_0 \|_{H}^2 + \| \dot u_0 \|_{L^2}^2 := \int_0^\infty  \Big((\p_r u_0(r))^2 +  k^2 \frac{u_0(r)^2}{r^2}  \Big) \, r \, \ud r +  \int_0^\infty \dot u_0(r)^2 \, r \, \ud r .
 }
 Denoting by $\LL_0 := -\De + k^2 r^{-2}$ we remark that the $H$ norm of a smooth function $u_0$ can also expressed as 
$
  \| u_0  \|_H^2 = \ang{ \LL_0 u_0 \mid u_0},
 $
 where $\ang{f \mid g} := (2\pi)^{-1} \ang{ f \mid g}_{L^2( \R^2)}$ is the $L^2$ inner product. We use $\LL_0$ to define spaces of higher regularity, and we let $\HH^2$ denote the norm 
\EQ{
\| \bs u_0 \|_{\HH^2}^2:= \| u_0 \|_{H^2}^2 +   \| \dot u_0 \|_H^2  := \ang{ \LL_0 u_0 \mid \LL_0 u_0}  + \ang{\LL_0 \dot u_0 \mid \dot u_0}.
}
We also require the following weighted norm, 
\EQ{
\| \bs u_0 \|_{\bs \La^{-1}  \HH} := \| (r \p_r u_0, ( r \p_r + 1) \dot u_0) \|_{\HH} .
}
While $\bs Q \not \in \HH$, this solution to~\eqref{eq:wmk} still plays a significant role in the dynamics of solutions in $\HH$;  for example,  superpositions of two bubbles, i.e.,  $Q(r/ \lam) - Q(r/ \mu)$ for $ \lam \neq \mu$, are elements of $H$.

 \subsection{Sub-threshold theorems and bubbling}
 The regularity theory for energy critical wave maps has been extensively studied; \cite{CTZcpam, CTZduke, STZ92, STZ94, KlaMac93, KlaMac94, KlaMac95, KlaMac97, KlaSel97, KlaSel02, Tat98, Tao1, Tao2, Tat01, Kri04}. Recently, the focus has been on the nonlinear dynamics of solutions with large energy. A remarkable~\emph{sub-threshold theorem} was established in~\cite{ST1, ST2, KS, Tao37}: every wave map with energy less than that of the first nontrivial harmonic map is globally regular on $\R^{1+2}$ and scatters to a constant map. The role of the minimal harmonic map in the formulation of the sub-threshold theoem was first clarified by fundamental work of Struwe~\cite{Struwe}, who showed that the smooth equivariant wave map flow can only develop a singularity by concentrating energy at the tip of a light cone via the bubbling off of at least one non-trivial finite energy harmonic map.  Bubbling wave maps were first constructed in a series of influential works by  Krieger, Schlag, Tataru~\cite{KST}, Rodnianski, Sterbenz~\cite{RS}, and Rapha\"el, Rodnianski~\cite{RR}, with the latter work yielding a stable blow-up regime;  see also the recent work~\cite{KrMiao-Duke}  for stability properties of the solutions from~\cite{KST}, as well as~\cite{JLR1} for a classification of blowup solutions with a given radiation profile, and~\cite{Pil-19} for a construction of a new class of singular solutions that blow up in infinite time. In particular, all of these works demonstrate that blow up by bubbling can occur for maps with energy slightly above the ground-state harmonic map, which shows the sharpness of the sub-threshold theorem. 
 
The sub-threshold theorem can be refined by taking into account the topological degree of the map. Only topologically trivial maps can scatter to a constant map and it was shown in~\cite{CKLS1, LO1} that the correct threshold that ensures scattering is $\E < 2\E( \bs Q)$ (rather than $\E(\bs Q)$). The reasoning behind the number $2 \E(\bs Q)$ is as follows. The topological degree counts (with orientation) the number of times a map `wraps around' $\Sp^2$. If a harmonic map of degree $k$ bubbles off from a wave map $\Psi(t)$,  then, in order for $\Psi(t)$  to be degree zero,  it must also `unwrap' $k$ times away from the bubble. The minimum energy required for each wrapping  is $4 \pi k = \E(\bs Q)$. Thus the energy required for a degree zero map to form a bubble is  $\E \ge  8 \pi k   =  2\E(\bs Q)$. 


 
\subsection{Main result: the existence and regularity of  two-bubble wave maps}  \label{s:refined} 
 
We consider topologically trivial $k$-equivariant maps with precisely the threshold energy $\E =  2 \E(\bs Q)$. Building on the work~\cite{JJ-AJM} of the first author and our work~\cite{JL1}, we can now give an exact description of every such map.  Together with the companion work~\cite{JL2-uniqueness}, we show that for equivariance classes $k \ge 4$,  there is a \emph{unique} (up to the natural invariances up the equation) threshold wave map that does not scatter in both time directions. 
 
 Let $\bs u(t) : [T_0, \infty) \to \HH$ be a solution to~\eqref{eq:wmk} with $\E(\bs u) = 2 \E(\bs Q)$. We say $\bs u(t)$ is a \emph{two-bubble in
 forward time} if there exist $\iota \in \{+1, -1\}$
 and continuous functions $\lambda(t), \mu(t) > 0$ such that
 \EQ{ \label{eq:2-bub-def} 
\lim_{t \to \infty} \| ( u(t) - \iota( Q_{\la(t)} -  Q_{\mu(t)}), \p_t u(t))\|_{\HH} = 0, \quad \lambda(t) \ll \mu(t)\text{ as }t \to \infty.
}
The notion of a \emph{two-bubble in the backward time direction} is defined similarly. Here $Q_{\nu}$ denotes the scaling 
$
Q_{\nu}(r) := Q( r/ \nu). 
$
 In~\cite{JJ-AJM} the first author constructed a two-bubble in forward time.
   Later, in~\cite{JL1} we showed  that the solution constructed in~\cite{JJ-AJM} must also be global and scattering in backwards time.  In this paper we give a refined construction of a two-bubble in forward time, deducing additional regularity properties of the solution, namely that it lies in the space $ \HH \cap \HH^2 \cap \bs \Lam^{-1} \HH$,   along with a dynamical description of the modulation parameters, and an expansion of the solution into profiles. This refined construction will serve as a key ingredient in the companion paper~\cite{JL2-uniqueness} where we show that there is, in fact, only one two bubble in forward time.

We begin by introducing some notation needed to state Theorem~\ref{t:refined} below.  We define, 
\EQ{ \label{eq:q-rho-ga} 
\rho_k &:= \Big( \frac{8k}{\pi}  \sin( \pi/k) \Big)^{\frac{1}{2}} , \quad 
 \gamma_k := \frac{k}{2} \rho_k^2, \quad q_k:=  \Big( \frac{k-2}{2} \rho_k\Big)^{-\frac{2}{k-2}}
 }
 We remark that $\rho_k^2 = 16 k \| \Lam Q \|_{L^2}^{-2}$.  
 Given a radial function $w: \R^2  \to \R$ we denote the $H$ and $L^2$  re-scalings as follows 
\EQ{ \label{eq:scale} 
w_\la(r) := w(r/ \la), \quad
w_{\ula}(r)  := \frac{1}{\la} w (r/ \la)
}
The corresponding infinitesimal generators  are given by 
\EQ{ \label{eq:LaLa0} 
&\La w:= -\frac{\partial}{\partial \lambda}\bigg|_{\lambda = 1} w_\la = r \p_r w  \quad (H  \,  \textrm{scaling}) \\
& \La_0 w:= -\frac{\partial}{\partial \lambda}\bigg|_{\lambda = 1} w_{\ula} = (1 + r \p_r ) w  \quad (L^2  \textrm{scaling})
}
Next, we define $C^{\infty}(0, \infty)$ functions $A, B, \ti B$ as the unique solutions to the equations, 
 \EQ{ 
 \LL A &= - \La_0 \La Q,  \quad 
 0=  \ang{A \mid \La Q} \\ 
 \LL B &= \gamma
 _k  \La Q - 4 r^{k-2} [\La Q]^2,  \quad 
0 =  \ang{B \mid \La Q},  \\
  \LL \ti B &= - \gamma_k \La Q + 4r^{-k-2} [\La Q]^2,  \quad 
  0 =  \ang{ \ti B \mid \La Q} 
  } 
  where here $\LL := -\De + r^{-2} f'(Q)$ is the operator obtained via linearization about $Q$. 
  These are constructed in Lemma~\ref{l:ABB} below, and here we note that 
  \EQ{
  A(r), B(r) &= O( r^k) \mas r \to 0, \quad \ti B(r) = O( r^k \abs{\log r}) \mas r \to 0 \\
  A(r), B(r), \ti B(r) &= O( r^{-k +2}) \mas r \to \infty
  }
  Next, given a time interval $J \subset \R$ and  a quadruplet of $C^1$ functions $(\mu(t), \lam(t), a(t), b(t))$ on $J$ we define  a refined  $2$-bubble ansatz, $$\bs \Phi(\mu(t), \lam(t), a(t), b(t), r) = (\Phi(\mu(t), \lam(t), a(t), b(t), r), \dot \Phi(\mu(t), \lam(t), a(t), b(t), r))$$ by 
  \EQ{ \label{eq:Phi-def1} 
 \Phi(\mu, \lam, a, b)&:= (Q_{\la} + b^2 A_\la + \nu^k B_\la ) - ( Q_\mu + a^2 A_\mu + \nu^k \ti B_\mu) \\
\dot \Phi(\mu, \lam, a, b)&:= b \La Q_{\U \la}  + b^3 \La A_{\U \la}  - 2 \gamma_k b \nu^k A_{\U \la}  + b \nu^k \La B_{\U \la} - k b \nu^k B_{\U \la} - k a \nu^{k+1} B_{\U \lam}  \\
& \quad  +  a \La Q_{\U \mu}+ a^3 \La A_{\U \mu}  + 2 \ti \ga_k a \nu^k A_{ \U\mu} + a \nu^k \La \ti B_{\U \mu} + k b \nu^{k-1} \ti B_{ \U\mu}  + k a \nu^k \ti B_{\U \mu} 
  }
  where we have introduced the notation, 
$
  \nu:= \lam/ \mu.
  $
  To ensure that $\dot \Phi \in L^2$, we now restrict to the setting $k \ge 4$. See Remark~\ref{r:k=3} below for a discussion of the cases $k =2, 3$. We establish the following theorem.

\begin{thm}[Existence and regularity of a two-bubble wave map] \label{t:refined} 
Fix any equivariance class $k \ge 4$. There exists a global-in-time solution $\bs u_c(t) \in \HH$ to~\eqref{eq:wmk} that is a \emph{two-bubble in forward time} with the following additional properties: 
\begin{itemize} 
\item The solution $\bs u_c(t)$ lies in the space $\HH \cap \HH^2 \cap \bs \Lam^{-1} \HH$, and scatters freely in negative time. 
\item There exists $T_0>0$, a quadruplet of  $C^1( [T_0, \infty))$ functions $(\mu_c(t), \lam_c(t), a_c(t), b_c(t))$, and $\bs w_c(t)  \in \HH \cap \HH^2 \cap \bs \Lam^{-1} \HH$ so that on the time interval $[T_0, \infty)$ the solution $\bs u_c(t)$ admits a decomposition, 
\EQ{
\bs u_c(t)  = \bs \Phi(\mu_c(t), \lam_c(t), a_c(t), b_c(t)) + \bs w_c(t) 
}
where $\bs \Phi$ is defined in~\eqref{eq:Phi-def1} and the functions $(\mu_c(t), \lam_c(t), a_c(t), b_c(t))$ satisfy, 
\EQ{ \label{eq:mod-est} 
\lam_c(t) &= q_k t^{-\frac{2}{k-2}} (1 + O( t^{-\frac{4}{k-2} + \eps}) ) \mas t \to \infty, \\
\mu_c(t) &= 1 -\frac{k}{2(k+2)} q_k^2 t^{-\frac{4}{k-2}}  + O(t^{-\frac{6}{k-2}+ \eps}) \mas t \to \infty , \\
b_c(t) & = q_k  \frac{2}{k-2} t^{- \frac{k}{k-2}}( 1+  O( t^{-\frac{4}{k-2} + \eps}) ) , \mas t \to \infty ,\\  
a_c(t) & =  \frac{2k}{(k-2)(k+2)} q_k^2 t^{-\frac{k+2}{k-2}} ( 1 + O( t^{-\frac{4}{k-2} + \eps}) )  \mas t \to \infty ,
}
where $\eps>0$ is any fixed small constant. We also have, 
\EQ{ \label{eq:mod'} 
\abs{ \lam_c'(t) + b_c(t)}  &\lesssim  t^{-\frac{2}{k-2}(2k-1)}   \mas t \to \infty,  \\
 \abs{\mu_c'(t)  - a_c(t)} & \lesssim  t^{-\frac{2}{k-2}(2k-1)}   \mas t \to \infty .
}
Finally,  $\bs w_c(t)$ satisfies, 
\EQ{ \label{eq:w-est} 
 \| \bs w_c(t) \|_{\HH}^2 & \lesssim \lam_c(t)^{3k-2}, \\
 \| \bs w_c(t) \|_{\HH^2}^2 & \lesssim \lam_c(t)^{3k-4}, \\
 \| \bs \Lam \bs w_c(t)  \|_{\HH}^2 & \lesssim \lam_c(t)^{2k-2} , 
}
uniformly in $t \ge T_0$. 
\end{itemize} 
\end{thm} 

 \begin{rem} 
We highlight that $\bs u_c(t)$ lies in $ \HH^2$, i.e., it is smoother than a general finite energy solution.  
In the companion paper~\cite{JL2-uniqueness} we will exploit the fact that the trajectory of $\bs u_c(t)$ naturally yields a two-dimensional forward-invariant submanifold of $\HH$ for~\eqref{eq:wmk}.  We expect that this manifold is smooth.  In fact, the method used to obtain $\HH^2$ regularity can likely be iterated to obtain higher regularity of $\bs  u_c(t)$, but we do not pursue this here. 

\end{rem} 

\begin{rem}
In the companion paper~\cite{JL2-uniqueness} we will show that if $\bs u(t)$ is \emph{any other finite energy} 2-bubble in forward time, then there exist $(t_0, \mu_0) \in \R \times (0,\infty)$ so that 
\EQ{
\bs u(t) =  \bs u_{c, t_0, \mu_0, \pm}(t) := \pm \Big(u_c( t- t_0, r/ \mu_0), \frac{1}{\mu_0} \p_t u_c(t- t_0, r/ \mu_0) \Big), 
  }
  i.e., $\bs u_c(t)$ is \emph{unique} up to sign, time translation, and scale; see~\cite[Theorem 1.1]{JL2-uniqueness}. Together with the main result in our earlier work~\cite{JL1} this completes an exact classification of every wave map with threshold energy: any such map either scatters in both time directions, is equal to $\bs u_{c, t_0, \mu_0, \pm}(t)$, or is equal to the solution obtained by time reversal of this map; see~\cite[Theorem 1.4]{JL2-uniqueness}.   We remark that the fact that $\bs u_c(t)$  scatters in backwards time  follows from the main theorem in~\cite{JL1}.
\end{rem} 

\begin{rem}[Comparing with the construction in~\cite{JJ-AJM}]  
 After we prove uniqueness in~\cite{JL2-uniqueness} it is clear that the solution $\bs u_c(t)$ constructed here is the same as the one found in the first author's work~\cite{JJ-AJM}. However, the information about $\bs u_c(t)$ from~\cite{JJ-AJM} is not sufficient to prove uniqueness via the technique introduced in~\cite{JL2-uniqueness}, which relies heavily on all of the refined properties of $\bs u_c(t)$, the modulation parameters, $(\mu_c(t), \lam_c(t), a_c(t), b_c(t))$, and the ansatz error $\bs w_c(t)$ stated in Theorem~\ref{t:refined}. The construction from~\cite{JJ-AJM} does not yield these refined properties. 
\end{rem} 

\begin{rem} \label{r:k=3}
The theorem can be proved in the case $k=3$ by a nearly identical argument after introducing suitable cutoffs in the definition of $\dot \Phi$. We chose not to include this here to keep the  exposition as simple as possible.   
The analogous theorem is expected to hold also in the case $k=2$, however more care is required,  see~\cite{JL1}. 
 
 While the main purpose of the refined construction in Theorem~\ref{t:refined} is the proof of uniqueness of $\bs u_c(t)$ in~\cite{JL2-uniqueness} we note  that  Theorem~\ref{t:refined}, and in particular the techniques introduced here to prove it, are rather general and should be of independent interest/use. 
\end{rem} 

\begin{rem}[Strong vs. weak soliton interactions] 
One can  compare/contrast  Theorem~\ref{t:refined} with other multi-soliton results such as the landmark works of Merle~\cite{Merle90}, Martel~\cite{Martel05}, and Martel, Merle~\cite{MM06}, which  constructed $N$-soliton solutions to g-KdV and NLS with distinct, nontrivial velocities;   see also~\cite{MaMeTs, CMM11, CM}.  
 We refer to the multi-solitons in all of those works as \emph{weakly interacting} since the leading order dynamics are given/determined by the internal motion of each individual soliton.   We emphasize here a distinction with Theorem~\ref{t:refined}: the bubbles in  $\bs u_c(t)$ are  \emph{strongly interacting} in the sense that the dynamics are driven by nonlinear interactions between the two bubbles, whereas in the weakly interacting regime the soliton interactions are negligible to main order.  
 There have been several recent $2$-soliton constructions in the strongly interacting regime; see e.g., ~\cite{MR18, JJ-APDE, moi-gkdv, Ngu-17, JM-19, JKL1}. To distinguish from these, here we emphasize the $\HH^2$ regularity of $\bs u_c(t)$, and the fine description of the asymptotics, which are novel. 
\end{rem}

\subsection{An outline of the proof} 



As is typical in constructions of solutions with nontrivial dynamics, the first order of business is to obtain a suitable ansatz $\bs \Phi(\mu, \la, a, b)$, which is defined in~\eqref{eq:Phi-def1}. This is carried out  in Sections~\ref{s:formal},~\ref{s:modulation}. Next, we perform a modulation analysis for solutions to~\eqref{eq:wmk} near the ansatz, i.e., given such a solution $\bs u(t)$ on a time interval $J$ we find a unique set of modulations parameters $(\mu(t), \lam(t), a(t), b(t))$ and a unique $\bs w(t) \in \HH \cap \HH^2 \cap \bs \Lam^{-1} \HH$  (see Lemma~\ref{l:mod2}) so that 
\EQ{
\bs u(t) &= \bs \Phi( \mu(t),\lam(t), a(t), b(t)) + \bs w(t) \\ 
0 &=\ang{ w(t) \mid \Lam Q_{\U{\mu(t)}}} = \ang{ w(t) \mid  \Lam Q_{\U{\lam(t)}}} = \ang{ \dot w(t) \mid \Lam Q_{\U{\mu(t)}}} = \ang{ \dot w(t)  \mid \Lam Q_{\U{\lam(t)}}} 
}
Differentiation of the orthogonality conditions above leads to dynamical estimates on the modulation parameters, (see Lemma~\ref{l:modc2}), which only become useful if we can also obtain a priori control of the size of $\|\bs w(t) \|_{\HH}$.  For now one can think of $(\mu, \lam, a, b)$ as approximately solving the formal system, 
\EQ{ \label{eq:mod-formal-intro} 
 \mu' = a, \quad \lam' = - b, \quad a'  =  - \gamma_k \lam^{k} \mu^{-k -1}, \quad b' =  -\gamma_k \lam^{k-1} \mu^{-k}
}
which leads to~\eqref{eq:mod-est}; see Lemma~\ref{l:modc2}. 

The key step is to establish energy-type and higher regularity estimates on $\bs w(t)$, which will then be used to close the estimates obtained in the modulation analysis. We devise modified linear energy functionals, 
\EQ{
\HH_1(t)&:= \frac{1}{2} \ang{ \dot w \mid \dot w} + \frac{1}{2} \ang{ \LL_{\Phi} w \mid w} - b(t) \ang{\calA_0(\lam) w \mid  \dot w} \\
\HH_2(t)& := \frac{1}{2} \ang{  \LL_\Phi \dot w \mid \dot w} + \frac{1}{2} \ang{ \LL_{\Phi}^2 w \mid w} - b(t)  \ang{ \A_0(\la) w \mid \LL_\la \dot w}  + 2  b(t)\ang{\A_0(\la) \dot w \mid \LL_\la w} \\
}
where above $\LL_{\Phi} := -\De  + r^{-2} f'( \Phi)$ is the linearized operator about the ansatz $\Phi(\mu,\lam, a, b)$ and $\LL_{\lam}:= -\De + r^{-2} f'(Q_\lam)$ denotes the linearized operator about $Q_\lam$. Of note above are the virial corrections $\calA_0(\lam)$ which can be thought of roughly as the rescaled $L^2$-scaling operator $\frac{1}{\lam} \Lam_0$, localized to scale $\lam$; see Lemma~\ref{l:opA}. Note that the terms involving this virial correction are small compared to the energy-type terms as long as $\abs{b(t)} \ll 1$, hence they are perturbative with respect to the size of the energy functionals. However, these corrections are crucial towards proving monotonicity formulae for $\HH_1, \HH_2$ as they are designed to (mostly) cancel terms of critical size but indeterminate sign that arise when time derivatives hit the potential. To clarify this, consider a toy model system of the form, 
\EQ{
\p_t  w &= \dot w \\
\p_t \dot w &= \LL_{\lam (t)} w + f(t) , \quad \ang{ w \mid  \Lam Q_{ \U \lam}} = 0 
}
where the orthogonality condition is included to ensure coercivity of the linear energy functional, and $\LL_{\lam(t)} w := -\De w + r^{-2} f'(Q_{\lam(t)}) w$ is the operator obtained by linearization about $Q_{\lam(t)}$.  A direct computation shows that 
\EQ{
\frac{\ud}{\ud t} \Big( \frac{1}{2} \ang{ \dot w \mid \dot w} + \frac{1}{2} \ang{ \LL_{\lam} w \mid w} \Big)  = \ang{ \dot w \mid f} + \frac{1}{2} \ang{ [ \p_t, \LL_\la] w \mid w} 
}
and thus the last term above obstructs an estimate on the $\HH$ norm of $\bs w(t)$ in terms of  the forcing $f(t)$. Indeed, 
\EQ{
[\p_t, \LL_\lam] w =  -\frac{\lam'}{\lam} r^{-2} f''(Q_\lam) \Lam Q_{\lam} w
}
which, since $\abs{\lam'(t)/ \lam(t)} =O( t^{-1})$ (for $\lam(t)$ that  is to leading order a power of $t$, as in our case), means that $ \abs{\ang{ [ \p_t, \LL_\la] w \mid w} } = O( t^{-1} \| \bs w \|_{\HH}^2)$, and hence cannot be absorbed on the left after integration. This is where the virial correction comes into play. First, note that the above can be rewritten as, 
\EQ{
[\p_t, \LL_\lam] w =  -\frac{\lam'}{\lam} r^{-2} f''(Q_\lam) \Lam Q_{\lam} w = -\frac{\lam'}{\lam} r^{-1}  \p_r ( f'(Q_\lam) - k^2) w
}
and an integration by parts reveals that, 
\EQ{ \label{eq:comm-intro} 
 \frac{1}{2} \ang{ [ \p_t, \LL_\la] w \mid w}  =  - \frac{\lam'}{\lam} \frac{1}{2} \ang{ r^{-1} \p_r ( f'(Q_\lam) - k^2) \mid w^2} = \frac{\lam'}{\lam} \ang{ r^{-2}( f'(Q_\lam) - k^2) w \mid \Lam w} 
}
Note the presence of the virial operator $\Lam$ above. The goal is to add a perturbative correction to the energy functional that will cancel this term, up coercive or lower order terms.   Consider the derivative of a correction of the form, 
\EQ{
\frac{\ud}{\ud t} \Big(\frac{\lam'}{\lam} \ang{ ``\Lam_0" w \mid \dot w} \Big) &=  \frac{\lam'' \lam - (\lam')^2}{\lam^2} \ang{ ``\Lam_0" w \mid \dot w} + \frac{\lam'}{\lam} \ang{ ``\Lam_0" \p_t w \mid \dot w} +  \frac{\lam'}{\lam} \ang{ ``\Lam_0" w \mid \p_t \dot w}  \\
& =    \frac{o(1)}{t} \| \bs w \|_{\HH}^2 +  O(t^{-1}) \|  \bs w \|_{\HH} \| f \|_{L^2} + \frac{\lam'}{\lam} \ang{ ``\Lam_0" \dot w \mid \dot w}  -  \frac{\lam'}{\lam} \ang{ ``\Lam_0" w \mid \LL_\lam w} 
}
where above $``\Lam_0"$ refers to a localized (to scale $\lam$) version of $\Lam_0$, to ensure this operator is bounded from $L^2 \to \HH$. Note that $\Lam_0$ is  anti-symmetric (see~\eqref{eq:ibpLa}) so the third term on the right above is $=0$. Finally, the last term produces the desired cancellation, up to a coercive term. To see this, we write, 
\EQ{
 -\frac{\lam'}{\lam} \ang{ ``\Lam_0" w \mid \LL_\lam w} &= -\frac{\lam'}{\lam} \ang{ ``\Lam_0" w \mid \LL_0 w}   -\frac{\lam'}{\lam} \ang{ ``\Lam_0" w \mid r^{-2} ( f'(Q_\lam) - k^2)  w}\\
 & =  -\frac{1}{2} \frac{\lam'}{\lam} \ang{ [\LL_0, ``\Lam_0"] w \mid w}  -\frac{\lam'}{\lam} \ang{  w \mid r^{-2} ( f'(Q_\lam) - k^2)  w} \\
 & \quad   -\frac{\lam'}{\lam} \ang{ ``\Lam" w \mid r^{-2} ( f'(Q_\lam) - k^2)  w}
}
where we have noted the identity $\Lam_0 w = \Lam w + w$ above. We note that the commutator identity $[\LL_0, \La]w = 2 \LL_0 w$ (see~\eqref{eq:comm-LL0}) leads to a localized Pohozaev-type coercivity estimate for the first two terms on the right above and these can be dismissed; see~\eqref{eq:A-pohozaev} from Lemma~\ref{l:opA}. Finally, the last term above precisely cancels the last term on the right of~\eqref{eq:comm-intro}!  We remark that this technique is inspired by a related ``Morawetz"-corrected energy functional introduced by Rapha\"el and Szeftel in~\cite{RaSz11}. 


We follow this rough outline to prove a priori energy estimates for $\HH_1(t), \HH_2(t)$ in terms of the modulation parameters, as well as a less refined weighted energy estimate (in $\bs \Lam^{-1} \HH$, but this will not require a virial correction). Then, given a sequence $t_n \to \infty$, $t_1 \gg 1$, we use a standard bootstrap argument (see e.g., ~\cite{Martel05, JJ-AJM}) to close uniform estimates on the time interval $[t_1, \infty)$ for a sequence of solutions with  initial data of the form $\bs u_0(t_n)  = \bs \Phi( \mu_n, \lam_n, a_n, b_n)$. Here,  the initial values of the parameters are chosen via the formal system~\eqref{eq:mod-formal-intro}. The desired solution $\bs u_c(t)$ is then obtained by passing to a suitable limit. 

\begin{rem}
In the context of singularity development for~\eqref{eq:wmk}, Rodnianski and Sterbenz~\cite{RS} and Rapha\"el and Rodnianski~\cite{RR} also make use of $H^2$-type estimates in their schemes (see the Morawetz type estimates in~\cite{RS} and energy-type estimates in~\cite{RR}). However, in both~\cite{RS, RR} the analysis relies on the remarkable Bogomol'nyi structure of  $\LL$. Indeed the linearized operator $\LL$ admits a factorization $\LL = \calB^* \B$ into two first order operators. While $\LL$ has nontrivial kernel (spanned by $\Lam Q$) one can show via direct computation that the operator $\ti \LL := \B \B^*$ is \emph{unconditionally coercive}, which suggests the strategy of applying $\B$ to the equation for the error $\varepsilon$ after extracting the ansatz ($Q_b$ in the Merle-Rapha\"el parlance) and working with the resulting differentiated equation. 

In contrast, we emphasize that the technique used here to prove energy-type and $H^2$-type estimates does not make use of the Bogomol'nyi structure, i.e.,  we work directly with $\LL$ rather than $\ti \LL$. This technique should be useful in other settings. 
\end{rem}

\section{Preliminaries} 
For radial functions $u, v$ on $L^2(\R^2)$, we write $u = u(r), v = v(r)$ and we use the notation, 
\EQ{
\ang{ u \mid v} := \frac{1}{2 \pi} \ang{ u\mid v}_{L^2( \R^2)} = \int_0^\infty u(r)  \ba{v(r)} \,r \, \ud r .
}
Let $\LL_0$ denote the operator 
\EQ{ \label{eq:LL0-def} 
\LL_0 w := -\De w + \frac{k^2}{r^2} w .
}
We define the function space $H$ as the completion of $C^\infty_0((0, \infty))$ functions $w$ under the norm 
\EQ{
\| w \|_{H}^2 := \ang{ \LL_0 w \mid w}  =  \int_0^\infty \Big(( \p_r w(r))^2  + k^2 \frac{w(r)^2}{r^2}  \Big) \,r \ud r
}
For the vector pair $\bs w =(w , \dot w)$ we define the norm $\HH$ by 
\EQ{
\| \bs w \|_{\HH}^2 := \| w \|_H^2 + \| \dot w \|_{L^2}^2
}
Next, we define the space $H^2$ via the norm, 
\EQ{
 \| w \|_{H^2}^2:= \ang{ \LL_0 w  \mid \LL_0 w} = \int_0^\infty\Big( ( \p_r^2 w(r))^2 + (2k^2 +1)\frac{( \p_r w(r))^2}{r^2}  + (k^4 - 4k^2) \frac{w(r)^2}{r^4} \Big) \, r \, \ud r 
}
And for the pair $ \bs w = (w,  \dot w)$ we define $\HH^2$ by 
\EQ{
\| \bs w \|_{\HH^2}^2:= \| w \|_{H^2}^2 +   \| \dot w \|_H^2 
}
We also require the following weighted norm, 
\EQ{
\| \bs w \|_{\bs \Lam^{-1} \HH} := \| (\La w, \La_0 \dot w) \|_{\HH} 
}
where $\Lam, \Lam_0$ are defined in~\eqref{eq:LaLa0}. 
It is a standard fact that the regularity of a solution $\bs u(t)$ to~\eqref{eq:wmk} in the space $\HH \cap \HH^2 \cap \bs \Lam^{-1} \HH$ is propagated by the flow. 

The  infinitesimal generators $\Lam, \Lam_0$ defined in~\eqref{eq:LaLa0} satisfy the integration by parts identities,
\EQ{ \label{eq:ibpLa} 
\ang{ \Lam f \mid g} &= - \ang{ f \mid \La g} - 2 \ang{ f \mid g} , \quad 
\ang{ \Lam_0 f \mid g} = -\ang{ f \mid \Lam_0 g} 
}

The operator $\LL_U$ obtained by linearization of~\eqref{eq:wmk} about the first component of finite energy map $\bs U = (U,  \dot U)$ plays an important role in the analysis. Given $g \in H$ we have, 
\EQ{  \label{eq:LU-def} 
\LL_U g := - \De g + k^2 \frac{\cos 2U}{r^2}  g
}
In fact, given any $\bs g = (g,\dot g) \in \HH$ we have 
\EQ{
\ang{ \uD^2 \E(\bs U) \bs g \mid \bs g } = \ang{ \LL_U g \mid g}_{L^2} + \ang{ \dot g \mid \dot g}_{L^2} =  \int_0^\infty \Big(\dot g^2 + (\p_r g)^2 + k^2 \frac{\cos 2U}{r^2} g^2 \, \Big) r \ud r
}

We record  the commutator identities, 
\EQ{ \label{eq:comm-LL0} 
[\LL_0, \La] w =   2  \LL_0 w  
}
\EQ{
[\LL_U, \La_0] h = [\LL_U, \La]h = 2 \LL_Uh + \frac{2 k^2 \sin 2 U  \La U}{r^2} h 
} 
For a time dependent function $h$ we also compute, 
\EQ{ \label{eq:comm-ptL} 
[\p_t, \LL_U] h = -\frac{2k^2\sin 2 U  \p_t U}{r^2}  h 
}

The most important instance of the operator $\LL_{U}$ is given by linearizing \eqref{eq:wmk} about $ U = Q_\la$. In this case we use the short-hand notation, 
\EQ{ \label{eq:LL-def} 
\LL_\la:= \LL_{Q_{\lam}} = (-\De + \frac{k^2}{r^2}) + \frac{1}{r^2} ( f'(Q) - k^2) 
}
We write $\LL := \LL_1$. 
We often use the notation $\LL= \LL_0 + P$, where $\LL_0$ is as in~\eqref{eq:LL0-def} and 
\EQ{ \label{eq:P-def} 
P(r)&:= \frac{1}{r^2}( f'(Q) - k^2)  = -\frac{2k^2\sin^2 Q}{r^2}  = -4k^2\frac{ r^{2k-2}}{ (1+ r^{2k})^2} 
}
We recall that 
\EQ{
\Lam Q(r) = k \sin Q = \frac{2k r^k}{1 + r^{2k}}
}
 is a zero energy eigenfunction for  $\LL$, that is,  
\EQ{
\LL \La Q = 0, \mand \La Q  \in L^2_{\textrm{rad}}(\R^2).
}
for all $k \ge 2$. When $k=1$, $\LL \La Q = 0$ holds but  $\La Q \not \in L^2$ due to slow decay as $r \to \infty$  and $0$ is referred to as a threshold resonance.
In fact, $\La Q$ spans the kernel of $\LL$; see~\cite{JL1} for more. 

We will also consider the operator $\LL^2$ which is given by the following formula, 
\EQ{ \label{eq:LL2-def} 
\LL^2 w := \LL_0^2 w + 2 P \LL_0 w  - 2 \p_r P \p_r w + (P^2 - \De P) w.
}
To ease notation we define 
\EQ{ \label{eq:K-def} 
\Ks w := 2 P \LL_0 w  - 2 \p_r P \p_r w + (P^2 - \De P) w
}
and write $\LL^2  = \LL_0^2 + \Ks$. 

We make note of the following functional inequalities. 
\begin{lem}  \label{l:wHH2} 
Suppose that $w \in H$. Then, 
\begin{align} 
 \| w \|_{L^{\infty}} &\lesssim \| w \|_H  \label{eq:winfty} 
 \end{align} 
 If $w \in H \cap H^2$, then, 
 \begin{align} 
 \| r^{-1} w  \|_{L^\infty}  + \| r^{-2} w \|_{L^2} &\lesssim  \| \p_r w \|_H \label{eq:w/rinfty} 
\end{align} 
as well as, 
\EQ{ \label{eq:w3}
\| r^{-2} w^3 \|_{L^2} \lesssim \| w \|_{H}^2 \| \p_r w \|_{H}  \\
\| r^{-3} w^3 \|_{L^2} \lesssim  \| w \|_{H} \| \p_r w \|_{H}^2
}
\end{lem} 
\begin{proof}
See~\cite[Section 2.1]{JL1} for the the proof of~\eqref{eq:winfty}. Note that $w \in H$ implies that $w(r) \to 0$ as $r \to 0$. Thus, for any $r_0 >0$.  
\EQ{
w(r_0) =  \int_0^{r_0} w_r(r)  \ud r \le  \left( \int_0^{r_0} w_r^2(r) \, \frac{\ud r}{r} \right)^{\frac{1}{2}}\left( \int_0^{r_0} r\,   \ud r\right)^{\frac{1}{2}} \lesssim r_0 \| \p_r w \|_H
}
Therefore, 
$
\| r^{-1}  w \|_{L^\infty} \lesssim \|  \p_r w \|_H. 
$
On the other hand we have 
\EQ{
\int_{r_0}^\infty  w^2(r) r^{-3} \, \ud r  &= \frac{1}{2} r_0^{-2} w(r_0)^2  + \int_{r_0}^\infty  w_r(r) w(r)  r^{-2} \, \ud r  \\
& \le \frac{1}{2} r_0^{-2} w(r_0)^2 + 4 \int_{r_0}^\infty  w_r^2(r)  \, \frac{\ud r }{r}  + \frac{1}{2} \int_{r_0}^\infty  w^2(r) r^{-3} \, \ud r
}
Absorbing the last term on the right into the left-hand side yields, 
\EQ{
\int_{r_0}^\infty  w^2(r) r^{-3} \, \ud r  & \lesssim  \| r^{-1} w \|_{L^\infty}^2 +  \|  \p_r w \|_{H}^2
}
Taking the supremum in $r_0$ completes the proof of~\eqref{eq:w/rinfty}. The inequalities in~\eqref{eq:w3} are straightforward consequences of~\eqref{eq:winfty} and~\eqref{eq:w/rinfty}. 
\end{proof}

\section{Setting up the proof of Theorem~\ref{t:refined}: modulation analysis }   \label{s:reg}

\subsection{A formal computation}  \label{s:formal} 
We begin with a formal computation designed to guess the structure of the profiles $A, B, \ti B$ and the dynamical parameters $(\mu_c, \lam_c, a_c, b_c)$ in the statement of Theorem~\ref{t:refined}.  

From the classification result in~\cite{JL1}, every pure two-bubble solution that concentrates at $t= \infty$ takes the form (after rescaling $\mu_0$ in \cite[Theorem 1.6]{JL1} to $ \mu_0 =1$), 
\EQ{ \label{eq:0order} 
\bs u(t) = - \bs Q_{\mu(t)} + \bs Q_{\la(t)} + o_{\HH}(1) \mas t \to \infty
}
with $\lam(t)  \to 0 , \mu(t) \to 1$ as $t \to \infty$.  The goal of this section is to guess the structure of the error $o_{\HH}(1)$ above. We do this formally by searching for a solution that takes the form 
\EQ{
\bs u(t)  &= - \bs Q_{\mu(t)} + a(t) \bs{\ti U}_{\mu(t)}^{(1)} +   a(t)^2\bs{\ti U}^{(2)}_{\mu(t)} + a(t)^3\bs{\ti U}^{(3)}_{\mu(t)} + \dots \\
&\quad  + \bs Q_{\la (t)}  + b(t)\bs U^{(1)}_{\la(t)}+ b(t)^2\bs U^{(2)}_{\la(t)} + b(t)^3\bs U^{(3)}_{\la(t)} + \dots 
}
with $\la(t), a(t), b(t) \to 0$ and $\mu(t)\to 1$ as $t \to \infty$. In the arguments below we view the correction terms on the first line as being most relevant in the regions $r \gg \sqrt{\lam(t) \mu(t)}$ and the correction in the second line as relevant in the region $r \ll \sqrt{\lam(t) \mu(t)}$. 

From~\eqref{eq:0order} 
we see that 
 to leading order we have 
\EQ{
\p_t u(t) \approx \mu'(t) \Lam Q_{\U{\mu(t)}} - \la'(t) \La Q_{\U{\la(t)}}.
}
This leads us to set 
\EQ{
a(t) = \mu'(t), \quad b(t) = - \la'(t)
}
 and $\bs U^{(1)} = \bs{\ti U}^{(1)}= (0, \La Q)$ so that 
\EQ{
\p_t u(t) \approx a(t) \La Q_{\U{\mu(t)}}+  b(t) \La Q_{\U{\la(t)}}
}
To find $\bs U^{(2)}$ and $\bs{\ti U}^{(2)}$, we differentiate the above again in time, obtaining 
\EQ{ \label{eq:pt2u} 
 \p_{t}^2 u(t) 
 &  \approx b'(t) \La Q_{\U{\la(t)}} + \frac{b(t)^2}{ \la(t)}[ \La_0 \La Q]_{\U{\la(t)}} +  a' (t)\La Q_{\U \mu(t)} -\frac{a(t)^2}{\mu(t)} [\La_0 \La Q]_{\U \mu(t)} 
}
One the other hand, we need 
$
 \p_{t}^2 u = \De u  - \frac{1}{r^2} f(u). 
$
We find that, 
\begin{align} 
\De u &- \frac{1}{r^2} f(u) \approx -\LL_\la b^2 U_\la^{(2)}  + \LL_\mu a^2  \ti U_\mu^{(2)}   - 4 \mu^{-k}   r^{k-2} (\La Q_{\la})^2   - 4 \la^k r^{-k-2} (\La Q_{\mu})^2  \\
&   -\frac{1}{r^2}\Big( f( Q_{\la} - Q_\mu) - f(Q_\la) + f(Q_\mu) - 4  \mu^{-k} r^{k-2} (\La Q_{\la})^2- 4 \frac{\la^k}{r^{k+2}} (\La Q_{\mu})^2 \Big) \\
&  - \frac{1}{r^2} \Big( f( Q_{\la} - Q_\mu + b^2 U_\la^{(2)} - a^2 \ti U_\mu^{(2)}) - f(Q_\la - Q_\mu) - f'(Q_\la - Q_\mu) ( b^2 U_\la^{(2)} - a^2 \ti U_{\mu}^{(2)}) \Big) \\
&  -\frac{1}{r^2} \Big( f'( Q_\la - Q_\mu) - f'(Q_\la)\Big) b^2 U_\la^{(2)}  +\frac{1}{r^2} \Big ( f'(Q_\la - Q_\mu) - f'(Q_\mu) \Big) a^2 \ti U_{\mu}^{(2)} \\
& \approx -\LL_\la b^2 U_\la^{(2)}  + \LL_\mu a^2  \ti U_\mu^{(2)}   - 4 \mu^{-k}   r^{k-2} (\La Q_{\la})^2   - 4 \la^k r^{-k-2} (\La Q_{\mu})^2 \label{eq:rhs1} 
\end{align} 
where we have used the identity and approximation, 
\EQ{
-\frac{1}{r^2}(f(Q_{\la} - Q_\mu) - f(Q_\la) + f(Q_\mu)) 
 & = -\frac{\sin 2Q_\mu}{r^2} [\La Q_\la]^2 + \frac{\sin 2Q_\la}{r^2} [\La Q_\mu]^2 \\
 & \approx - 4  \mu^{-k} r^{k-2} (\La Q_{\la})^2 - \frac{4 \la^k}{r^k} [\La Q_{\mu}]^2 
}
Combining~\eqref{eq:pt2u} with~\eqref{eq:rhs1} we arrive at the requirement that, 
\EQ{
b' \La Q_{\U{\la}} &+ \frac{b^2}{ \la}[ \La_0 \La Q]_{\U{\la}} +  a' \La Q_{\U \mu} -\frac{a^2}{\mu} [\La_0 \La Q]_{\U \mu}  \\
& = -\LL_\la b^2 U_\la^{(2)}  + \LL_\mu a^2  \ti U_\mu^{(2)}   - 4 \mu^{-k}   r^{k-2} (\La Q_{\la})^2   - 4 \la^k r^{-k-2} (\La Q_{\mu})^2
}
which we impose by solving the two equations, 
\EQ{ \label{eq:ab-eq} 
b^2 \LL_\la U_\la^{(2)} = - \frac{b^2}{ \la}[ \La_0 \La Q]_{\U{\la}} - b' \La Q_{\U{\la}} - 4 \mu^{-k}   r^{k-2} (\La Q_{\la})^2 \\
a^2\LL_\mu   \ti U_\mu^{(2)} =  -\frac{a^2}{\mu} [\La_0 \La Q]_{\U \mu} +  a' \La Q_{\U \mu} + 4 \la^k r^{-k-2} (\La Q_{\mu})^2
} 
Consider the first equation above. After rescaling we arrive at 
\EQ{ \label{eq:Teq} 
 \LL U^{(2)}  
 & = - \La_0 \La Q - \frac{\la}{b^2} \left( b' \La Q + 4 {\la}^{k-1}\mu^{-k} [\La Q]^2 r^{k-2} \right) 
}
Solving this equation for $U^{(2)} \in L^2$ requires that the right-hand-side be orthogonal to the kernel of $\LL$, which is given by $\La Q$. Since 
\EQ{
\ang{- \La_0 \La Q \mid \La Q} = 0
}
we also need that 
\EQ{
\ang{b' \La Q + 4 {\la}^{k-1} \mu^{-k} [\La Q]^2 r^{k-2} \mid \La Q} = 0 
}
This leads to the equation  
\EQ{
b' = - {\la}^{k-1}\mu^{-k} \frac{4}{\| \La Q \|_{L^2}^2} \int_0^\infty [\La Q]^3 r^{k-1}   \, \ud r
}
Computing 
\EQ{
\int_0^\infty [\La Q]^3 r^{k-1}   \, \ud r = 8k^3 \int_0^\infty  \frac{r^k}{(r^{-k} + r^k)^3} \,  \frac{\ud r}{r} =  2 k^2
}
we arrive at the formal equation for $b'$, 
\EQ{
b' = -\gamma_k \lam^{k-1} \mu^{-k} 
}
where $\gamma_k$ is as in~\eqref{eq:q-rho-ga}.  The same argument for the second equation in~\eqref{eq:ab-eq} yields the following equation for $a'$, 
\EQ{
a' &= -\frac{1}{\mu} \frac{\la^k}{\mu^k}  \frac{4}{ \| \La Q \|_{L^2}^2} \int_0^\infty  (\La Q)^3 r^{-k-1} \, \ud r  \\
}
Computing, 
\EQ{
 \int_0^\infty  (\La Q)^3 r^{-k-1} \, \ud r = 8k^3 \int_0^\infty \frac{r^{-k}}{(r^{-k} + r^k)^3} \, \frac{\ud r}{r} = 2k^2
}
we arrive at the equation, 
\EQ{
a' = - \gamma_k \lam^k \mu^{-k-1} 
}
Putting this together we record the formal system for the parameters $(\mu, \lam, a, b)$, 
\EQ{ \label{eq:form-law} 
\lam'(t) &= -b(t) \\
\mu'(t) &= a(t)  \\ 
a'(t) &= -\gamma_k   \lam(t)^k  \mu(t)^{-k-1}  \\
b'(t) & = -\gamma_k \lam(t)^{k-1}  \mu(t)^{-k} 
}
\begin{rem} 
Note that~\eqref{eq:Teq} only determines $U^{(2)}, \ti U^{(2)}$ up to adding a multiples of $\La Q$. We will fix a particular choice in the course of the arguments below. 
\end{rem} 

In the proof of Theorem~\ref{t:refined} we will also need the next profiles $\bs U^{(3)}$ and $\bs{\ti  U}^{(3)}$. But these can be obtained directly from $U^{(2)}$ and $ \ti U^{(2)}$ as will be evident in the proof. We refer the reader back to the precise ansatz~\eqref{eq:Phi-def1}.

\subsection{Refinement of the error and modulation} \label{s:modulation} 
In this section we develop the precise ansatz to be used in the construction. 

In~\cite{JL1} we considered any solution $\bs u(t)$ with ~$\E = 2 \E( \bs Q)$ such that at some time $t_0>0$,  $\bs u(t_0)$ is close to the two-bubble manifold, i.e., 
\EQ{
 \inf_{\la, \mu>0} \left(\| \bs{u}(t_0) - ( \bs Q_{\la} - \bs Q_\mu) \|_{\HH}^2 +  (\la/ \mu)^k \right)  \ll 1
}
Then, we showed that one can uniquely define modulation parameters $\ti \la(t),  \ti \mu(t)$  and the error $\bs g(t)$ in a neighborhood of $t_0$, so that 
\EQ{
 \bs g(t) :=  \bs u(t) - ( \bs Q_{\ti \la(t)} - \bs Q_{\ti \mu(t)}) , \,\quad  \ang{\La Q_{\U{\ti \la(t)}} \mid g(t)} = \ang{\La Q_{\U{\ti \mu(t)}} \mid g(t)} = 0
}
However, the ansatz $\bs Q_{\ti \lam} - \bs Q_{\ti \mu} $ will not be sufficient for our purposes here and will need to be further refined. Motivated by the formal computations in the previous section, 
we define a refined distance to the two-bubble manifold taking into account the tangent space. 
Given a solution $\bs u(t)$ on a time interval $J$, define at each fixed time $t_0$ the refined distance,  
\begin{multline} \label{eq:d-def}  
\bfd_+( \bs u(t_0)) :=  \inf_{\la, \mu>0,  a, b \in \R} \bigg(\| (u(t_0) - ( Q_{\la} -  Q_\mu)\|_{H} +  (\lam/ \mu)^{-\frac{k}{2}} \|  \p_t u(t_0) - b \Lam Q_{\U \lam} - a \Lam Q_{\U \mu}) \|_{L^2}  \\
+ \Big( \frac{\la}{\mu}\Big)^k + a^2( 1+ \abs{a}(\lam/ \mu)^{-\frac{k}{2}})  + b^2( 1+ \abs{b}(\lam/ \mu)^{-\frac{k}{2}} )  \bigg)
\end{multline}

We first extract the next order profiles before settling on a choice of $\lam, \mu, a, b$ via suitable orthogonality conditions. Before defining them in Lemma~\ref{l:mod2}, we think informally of $\la(t), \mu(t)$ as perturbations of $\ti \lam(t), \ti \mu(t)$ above and $a, b$ as perturbations of $\lam', \mu'$.  We introduce notation.


\begin{defn}  \label{d:nu} 
Given functions  $\lambda(t), \mu(t)$ we define $\nu(t)$ to be their ratio. 
\EQ{ \label{eq:nudef} 
 \nu(t) := \frac{\la(t)}{\mu(t)}, \quad \nu'(t)= \frac{\la'(t)}{\mu(t)} -  \frac{\mu'(t)}{\mu(t)}  \nu(t)
   }
\end{defn} 



We now build a refined ansatz using $\mu(t), \la(t), a(t), b(t)$, which will be specified below. To this end, we set 
\EQ{
T(t, r) :=  A(r) + \frac{\nu^{k}(t)}{b^2(t)} B(r) , \quad 
\ti T(t, r) := A(r) +   \frac{\nu^k(t)}{ a^2(t)} \ti B(r) 
}
 where $A, B, \ti B$ solve 
\EQ{ 
 \LL A &= - \La_0 \La Q, \quad 
 0=  \ang{A \mid \La Q},    \label{eq:AQ} 
 }
 \EQ{
 \LL B &= \gamma
 _k  \La Q - 4 r^{k-2} [\La Q]^2,  \quad 
0 =  \ang{B \mid \La Q},  \label{eq:BQ} 
}
and 
\EQ{
  \LL \ti B &= -\ti \gamma_k \La Q + 4r^{-k-2} [\La Q]^2, \quad 
  0 =  \ang{ \ti B \mid \La Q}. \label{eq:tiBQ}
  } 
  Then, 
\EQ{ \label{eq:LLb2T} 
\LL_{\la} b^2 T_\la &= - \frac{b^2}{\la} \La_0 \La Q_{\U \la} + \gamma_k \frac{\nu^k}{\la} \La Q_{\U \la}  - 4  \frac{(r/ \mu)^k (\La Q_{\la})^2}{r^2} \\
 \LL_\mu a^2 \ti T_{\mu} &= - \frac{a^2}{\mu} \La_0  \La Q_{\U \mu}   - \ti \gamma_k \frac{\nu^k}{\mu}  \La Q_{\U \mu}  + 4 \frac{(r/ \la)^{-k} (\La Q_{\mu})^2}{r^2} 
}

The existence of such $A, B, \ti B$ is made precise in the following lemma. 
\begin{lem}  \label{l:ABB} Let $k \ge 4$. 
There exist $C^\infty(0, \infty)$ functions $A$ satisfying~\eqref{eq:AQ},  $B$ satisfying~\eqref{eq:BQ} and $\ti B$ satisfying~\eqref{eq:tiBQ}. Moreover $A, B, \ti B$ satisfy the estimates 
\EQ{ \label{eq:Aest} 
A(r) &= O( r^{k}), \quad \p_r A(r) = O(r^{k-1}), \quad \p_r^2 A(r) = O(r^{k-2})   \mas r \to 0  \\
A(r) &= O(r^{-k+2}), \quad \p_r A(r) = O(r^{-k+1}), \quad \p_r^2 A(r) = O(r^{-k})   \mas r \to \infty
}
\EQ{ \label{eq:Best} 
B(r) &= O( r^{k}), \quad \p_r B(r) = O(r^{k-1}), \quad \p_r^2 B(r) = O(r^{k-2})   \mas r \to 0  \\
B(r) &= O(r^{-k+2}), \quad \p_r B(r) = O(r^{-k+1}), \quad \p_r^2 B(r) = O(r^{-k})   \mas r \to \infty
}
\EQ{ \label{eq:tiBest} 
\ti B(r) &= O( r^{k}\abs{\log r}), \, \,  \p_r \ti B(r) = O(r^{k-1}\abs{\log r}), \, \,  \p_r^2  \ti B(r) = O(r^{k-2}\abs{\log r})   \mas r \to 0  \\
 \ti B(r) &= O(r^{-k+2}), \quad \p_r  \ti B(r) = O(r^{-k+1}), \quad \p_r^2 \ti  B(r) = O(r^{-k})   \mas r \to \infty
}
Moreover we have $A, \La A, \La_0 \La A, B,  \La B, \La_0 \La B \in L^2(\R^2)$ and $\ti B, \La \ti B, \La_0 \La \ti B  \in L^2(\R^2)$. 
\end{lem} 
\begin{rem}  \label{r:k4} 
The assumption that $k \ge 4$ ensures that $A, B, \ti B \in L^2(\R^2)$, which simplifies the analysis in the rest of this section. 
\end{rem} 

\subsubsection{Basic estimates involving $\Lam Q, A, B, \ti B$}
We record a collection of estimates that are direct consequences of the previous lemma and the definitions of $\La Q, \La_0 \La Q$. This technical subsection can be skipped on first reading and referred back to as needed. 
\begin{lem}\label{l:ABBQ}  Let $A, B, \ti B$ be as in Lemma~\ref{l:ABB} and let $\nu= \lam/ \mu \ll 1$. Then, 
\EQ{ \label{eq:ABBQ} 
\abs{ \ang{\La Q_{{\U \la}} \mid \La Q_{{\U \mu}}}}  + \abs{ \ang{ \La_0 \La Q_{\U \la}  \mid \La Q_{\U \mu} }} + \abs{ \ang{ \La Q_{\U \la} \mid \La_0 \La Q_{\U \mu}}} &\lesssim \nu^{k-1}  \\
\abs{ \ang{ \La Q_{{\U \la}}  \mid A_{{\U \mu}}}}  + \abs{\ang{ \La Q_{\U \la}  \mid  \La A_{\U \mu}}} + \abs{ \ang{\La_0 \La Q_{\U \la} \mid A_{\U \mu}}} + \abs{ \ang{\La_0 \La Q_{\U \la} \mid  \La A_{\U \mu}}} &\lesssim   \nu^{k-1}  \\
\abs{ \ang{ \La Q_{{\U \mu}}  \mid A_{{\U \la}}}}  + \abs{\ang{ \La Q_{\U \mu}  \mid  \La A_{\U \la}}} + \abs{ \ang{\La_0 \La Q_{\U \mu} \mid A_{\U \la}}} + \abs{ \ang{\La_0 \La Q_{\U \mu} \mid  \La A_{\U \la}}} &\lesssim  \nu^{k-3} \\
\abs{ \ang{ \La Q_{{\U \mu}}  \mid B_{{\U \la}}}}  + \abs{\ang{ \La Q_{\U \mu}  \mid  \La B_{\U \la}}} + \abs{ \ang{\La_0 \La Q_{\U \mu} \mid B_{\U \la}}} + \abs{ \ang{\La_0 \La Q_{\U \mu} \mid  \La B_{\U \la}}} &\lesssim \nu^{k-3} \\
\abs{ \ang{ \La Q_{{\U \la}}  \mid \ti B_{{\U \mu}}}}  + \abs{\ang{ \La Q_{\U \la}  \mid  \La \ti B_{\U \mu}}} + \abs{ \ang{\La_0 \La Q_{\U \la} \mid \ti B_{\U \mu}}} + \abs{ \ang{\La_0 \La Q_{\U \la} \mid  \La \ti B_{\U \mu}}} &\lesssim  \nu^{k-1} \\
}
Analogous estimates hold for $\Lam_0 \Lam A, \Lam_0 \Lam B, \Lam_0 \Lam \ti B$ paired with $\Lam Q_\nu$ and $\Lam Q_{\nu^{-1}}$. 
\end{lem} 

We also require the following estimates for nonlinear interactions between $\La Q, A, B, \ti B$. 

\begin{lem}  \label{l:ABBQnl} 
Let $A, B, \ti B$ be as in Lemma~\ref{l:ABB}, and let $\nu= \lam/ \mu \ll 1$. Then, 
\EQ{ \label{eq:ABBQpair} 
\abs{ \ang{\La Q_\mu \mid \frac{1}{r^2} A_\la^2}}  + \abs{ \ang{ \La Q_{\mu} \mid \frac{1}{r^2} B_\la^2}} &\lesssim   \nu^{k}  \\
  \abs{ \ang{\La Q_\la \mid \frac{1}{r^2} A_\mu^2}} +  \abs{ \ang{ \La Q_{\la} \mid \frac{1}{r^2} \ti  B_\mu^2}}  & \lesssim   \nu^{k}  \\
 \abs{ \ang{ [\La Q_{\mu}]^3 \mid \frac{1}{r^2} A_\la}} +   \abs{ \ang{ [\La Q_{\mu}]^3 \mid \frac{1}{r^2} B_\la}}  &\lesssim   \nu^{k-2} \\ 
 \abs{ \ang{ [\La Q_\mu]^2 \mid \frac{1}{r^2} \La Q_\la A_\la}} + \abs{ \ang{ [\La Q_\mu]^2 \mid \frac{1}{r^2} \La Q_\la B_\la}} & \lesssim \nu^{2k-2}  \\
 \abs{\ang{ \La Q_\mu A_\mu \mid \frac{1}{r^2}  [\La Q_\la]^2}} & \lesssim \nu^{2k}  \\
\abs{\ang{ \La Q_\mu  \ti B_\mu \mid \frac{1}{r^2}  [\La Q_\la]^2}}& \lesssim \nu^{2k-o(1)}  \\
 \abs{ \ang{ [\La Q_\mu]^2 A_\mu \mid \frac{1}{r^2}  \La Q_\la} } +  \abs{ \ang{ [\La Q_\mu]^2 \ti B_\mu \mid \frac{1}{r^2}  \La Q_\la} } &\lesssim \nu^k  \\
 \abs{\ang{ [\La Q_\la]^2 A_\la \mid \frac{1}{r^2} \La Q_\mu}}  + \abs{\ang{ [\La Q_\la]^2 B_\la \mid \frac{1}{r^2} \La Q_\mu}} &\lesssim  \nu^k  \\
 \abs{ \ang{ [\La Q_\la]^3 \mid \frac{1}{r^2} A_\mu}} & \lesssim \nu^k  \\
 \abs{ \ang{ [\La Q_\la]^3 \mid \frac{1}{r^2} \ti B_\mu}}  &\lesssim \nu^{k-o(1)}   
}
where the $o(1)$ above can be replaced with any small constant.
\end{lem} 

\begin{lem} \label{l:ABBQL2}  
Let $A, B, \ti B$ be as in Lemma~\ref{l:ABB}, and let $\nu= \lam/ \mu \ll 1$. Then, 
\EQ{ \label{eq:ABBQL2} 
\| r^{-2} [\La Q_\mu]^2 A_\la \|_{L^2}+  \| r^{-2} [\La Q_\mu]^2 B_\la \|_{L^2}   & \lesssim \frac{\nu^{k-1}}{\la}\\
 \| r^{-2} [\La Q_{\la}]^2 A_\mu \|_{L^2} &\lesssim \frac{\nu^{k}}{\la}  \\
   \|r^{-2} \La Q_\la \La Q_\mu A_\mu \|_{L^2}  + \| r^{-2} \La Q_\la \La Q_\mu A_\la \|_{L^2} + \| r^{-2} \La Q_\la \La Q_\mu B_\la \|_{L^2}   & \lesssim \frac{\nu^k}{\la}  \\ 
\| r^{-2} [\La Q_{\la}]^2 \ti B_\mu \|_{L^2} + \| r^{-2} \La Q_\la \La Q_\mu \ti B_\mu \|_{L^2} & \lesssim \frac{\nu^{k-o(1)}}{\la}  \\
}
where the $o(1)$ above can be replaced with any small constant.
\end{lem} 

\begin{lem}  \label{l:ABBQL21} 
Let $A, B, \ti B$ be as in Lemma~\ref{l:ABB}, and let $\nu= \lam/ \mu \ll 1$. Then, 
\EQ{ \label{eq:ABBQL21} 
 \| r^{-1} [\La Q_\la]^2 \La Q_\mu  \|_{L^2} +
    \| r^{-1} \La Q_\la [\La Q_\mu]^2 \|_{L^2}  +
    \| r^{-1} [\La Q_\la]^2 A_\mu \|_{L^2}  &\lesssim \nu^k  \\
  \|r^{-1}[ \La Q_\la]^2 \ti B_\mu \|_{L^2} &\lesssim  \nu^{k-o(1)}  \\
   \| r^{-1} [\La Q_\mu]^2 A_\la \|_{L^2}    +
 \| r^{-1} [\La Q_\mu]^2B_\la \|_{L^2} &\lesssim \nu^{k-2}  \\
  \|r^{-1} \La Q_\la [A_\mu]^2 \|_{L^2}+ 
   \| r^{-1} \La Q_\la [\ti B_\mu]^2 \|_{L^2} & \lesssim  \nu^k  \\
 \| r^{-1} \La Q_\mu [A_{\la}]^2 \|_{L^2}+ 
 \| r^{-1} \La Q_\mu[ B_\la]^2 \|_{L^2}  &\lesssim  \nu^k 
}
where the $o(1)$ above can be replaced with any small constant. 
\end{lem} 
\begin{proof}[Proof of Lemma~\ref{l:ABB}]
The existence of $A, B$ and the estimates~\eqref{eq:Aest} and~\eqref{eq:Best} is proved in~\cite[Lemma 5.1]{JJ-AJM} using the variation of constants formula. We proceed in a near identical fashion to find $\ti B$ satisfying~\eqref{eq:tiBest}, which is slightly more involved. 
Denote 
\EQ{ \label{eq:Fdef} 
F(r) = -\ti \gamma_k \La Q + 4r^{-k-2} [\La Q]^2
}
and note that $\ang{ \La Q \mid F} = 0$. Recall that $\LL$ factorizes as $
 \LL = \B^* \B
$
where 
\EQ{
\B = - \p_r + \frac{k}{r} \cos Q , \quad \B^* = \p_r + \frac{1}{r} + \frac{k}{r} \cos Q
}
To solve~\eqref{eq:tiBQ} we use the variation of constants formula twice. First, define 
\EQ{
G(r) = \frac{1}{r \La Q(r)} \int_0^r F(\rho) \La Q(\rho) \rho \,  \ud \rho = - \frac{1}{r \La Q(r)} \int_r^\infty F(\rho) \La Q(\rho) \rho \,  \ud \rho
}
where the second equality we use that $\ang{ \La Q \mid F} = 0$. Note that $G$ solves 
\EQ{
\calB^* G = F
}
Now we write $G$ as a sum two ways using the above and the definition of $F$ in~\eqref{eq:Fdef}. First define 
\EQ{
G_1(r) &:= - \ti \gamma_k \frac{1}{r \La Q(r)} \int_0^r  \La Q(\rho)^2 \rho \,  \ud \rho \\
G_2(r) &:=  4\frac{1}{r \La Q(r)} \int_0^r F(\rho) \La Q(\rho)^3 \rho^{-k-1} \,  \ud \rho \\
G_3(r) &:= \ti \gamma_k \frac{1}{r \La Q(r)} \int_r^\infty  \La Q(\rho)^2 \rho \,  \ud \rho \\
G_4(r) &:= -4\frac{1}{r \La Q(r)} \int_r^\infty F(\rho) \La Q(\rho)^3 \rho^{-k-1} \,  \ud \rho
}
so that $G = G_1 + G_2 = G_3 + G_4$ and thus letting $\chi_{ \le 1}(r)$ be a smooth cutoff with $\chi_{\le 1}(r) = 1$ if $r \le 1$ and $\chi_{\le 1}(r) = 0$ if $r \ge 2$ we have 
\EQ{
G(r) = \chi_{\le1}(r)  G_1(r) + \chi_{\le 1}(r) G_2(r) + (1-\chi_{\le 1}(r)) G_3(r) + (1- \chi_{\le 1}(r)) G_4(r) 
}
From their definitions we have 
\EQ{ \label{eq:Gest} 
G_1(r)  &= O( r^{k+1}) \mas r \to 0 , \quad 
G_2(r) = O(r^{k-1}) \mas r \to 0 \\
G_3(r) &= O(r^{-k+1}) \mas r \to \infty , \quad 
G_4(r) = O(r^{-3k -1}) \mas r \to \infty
}
Now using the variation of constants formula a second time we define 
\EQ{
 \ti B_1(r) &:=   \La Q(r) \int_0^r  \frac{\chi_{ \le1}(\rho)  G_1(\rho)}{\La Q( \rho)}  \, \ud \rho -  \La Q(r) \int_r^\infty  \frac{\chi_{ \le1}(\rho)  G_2(\rho)}{\La Q( \rho)}  \, \ud \rho  \\
 & \quad + \La Q(r) \int_0^r  \frac{(1-\chi_{\le 1}(\rho)) G_3(r)}{\La Q( \rho)}  \, \ud \rho -  \La Q(r) \int_r^\infty  \frac{(1- \chi_{\le 1}(\rho)) G_4(\rho)}{\La Q( \rho)}  \, \ud \rho
}
From the above and~\eqref{eq:Gest} we deduce the left-most bounds in~\eqref{eq:tiBest}. The equations $\B^* G = F$, $\calB B_1 = G$ and  $\B^* \B \ti B_1 = F$ can then be used to deduce the bounds in~\eqref{eq:tiBest} on the derivatives of $\ti B_1$. Finally to ensure the orthogonality in the second line of~\eqref{eq:tiBQ} we define 
\EQ{
\ti B(r) = \ti B_1(r)  - \frac{ \ang{\ti B_1 \mid \La Q}}{ \| \La Q \|_{L^2}^2} \La Q
}
Noting that $\ti B$ inherits the estimates~\eqref{eq:tiBest} from $\ti B_1$ and $\La Q$ completes the proof. 
\end{proof} 

\begin{proof}[Proofs of Lemma~\ref{l:ABBQ}, Lemma~\ref{l:ABBQnl}, Lemma~\ref{l:ABBQL2}, and Lemma~\ref{l:ABBQL21}] 
These follow from  straightforward computations. 
\end{proof} 


\subsubsection{The refined ansatz and modulation}


Given parameters $\la, \mu >0$ with $\nu := \lam/ \mu$ and $a, b \in \R$, define 
\EQ{
\bs \Phi  = \bs \Phi( \mu, \la, a, b) = (  \Phi( \mu, \lam, a, b), \dot \Phi( \mu, \lam, a, b))
}
via 
\EQ{ \label{eq:Phidef} 
\Phi &:= (Q_{\la} + b^2 A_\la + \nu^k B_\la ) - ( Q_\mu + a^2 A_\mu + \nu^k \ti B_\mu) \\
\dot \Phi &:= b \La Q_{\U \la}  + b^3 \La A_{\U \la}  - 2 \gamma_k b \nu^k A_{\U \la}  + b \nu^k \La B_{\U \la} - k b \nu^k B_{\U \la} - k a \nu^{k+1} B_{\U \lam}  \\
& \quad +  a \La Q_{\U \mu}+ a^3 \La A_{\U \mu}  + 2  \ga_k a \nu^k A_{ \U\mu} + a \nu^k \La \ti B_{\U \mu} + k b \nu^{k-1} \ti B_{ \U\mu}  + k a \nu^k \ti B_{\U \mu} 
}

Next, given a solution $\bs u(t)$ on a time interval $J$ on which $\bfd_+( \bs u(t)) \ll 1$ we fix parameters $(\mu(t), \lam(t), a(t), b(t))$ by modulation about $\bs \Phi$, i.e., by  imposing suitable orthogonality conditions. 


\begin{lem}[Modulation Lemma] \label{l:mod2} There exists $\eta_0 >0$ with the following property. Let $J  \subset \R$ be a time interval and let $\bs u(t) \in \HH$ be a solution to~\eqref{eq:wmk} on $J$. Assume that 
\EQ{ \label{eq:d+small} 
\bfd_+( \bs u(t)) \le \eta_0 \quad \forall \, t \in J
}
 Then, there exists unique $C^1(J)$ functions $\mu(t), \lam(t), a(t), b(t)$ so that, defining $\bs w(t)$ by 
\EQ{
 \bs w(t) := \bs u(t) - \bs \Phi(\mu(t), \la(t), a(t), b(t)) 
}
we have, for each $t \in J$, 
\begin{align}
 \ang{ \La Q_{\U{\la(t)}} \mid w} = 0  \label{eq:law} \\
 \ang{\La Q_{\U{\mu(t)}} \mid w} = 0  \label{eq:muw}  \\
 \ang{ \La Q_{\U{\la(t)} } \mid  \dot w} = 0    \label{eq:lawdot} \\
 \ang{\La Q_{\U{\mu(t)}} \mid \dot w} = 0  \label{eq:muwdot} 
\end{align} 
and so that 
\begin{multline}  \label{eq:w-d-bound} 
\|  w(t)\|_{H} + \nu(t)^{-\frac{k}{2}} \| \dot w(t) \|_{L^2}  +\nu(t)^{k}  \\
+ \abs{a(t)}^2( 1+ \abs{a(t)} \nu(t)^{-\frac{k}{2}}) + \abs{b(t)}^2(1+ \abs{b(t)} \nu(t)^{-\frac{k}{2}} )\lesssim \bfd_+( \bs u(t)) \lesssim  \eta_0.
\end{multline} 
\end{lem} 


\begin{rem} \label{rem:IFT}  We will use the following, less standard, version of the implicit function theorem in the proof of Lemma~\ref{l:mod2}. 

 \emph{
Let $X, Y, Z$ be Banach spaces and let $(x_0, y_0) \in X \times Y$,  and $\de_1, \de_2>0$. Consider a mapping  $G: B(x_0, \de_1) \times B(y_0, \de_2) \to Z$,  continuous in $x$ and $C^1$ in $y$. Suppose that $G(x_0, y_0) = 0$ and $(D_y G)(x_0, y_0)=: L_0$ has bounded inverse $L_0^{-1}$. Suppose in addition that 
\EQ{ \label{eq:IFT} 
&\| L_0 - D_y G(x, y) \|_{\LL(Y, Z)} \le \frac{1}{3 \| L_0^{-1} \|_{\LL(Z, Y)}}  \\ 
& \| G(x, y_0) \|_Z \le \frac{\de_2}{ 3 \| L_0^{-1} \|_{\LL(Z, Y)}} 
}
for all $\| x - x_0 \|_{X} \le \de_1$ and $\| y - y_0 \|_Y \le \de_2$. 
Then, there exists a continuous function $\si: B(x_0, \de_1)  \to B(y_0, \de_2)$ such that for all $x \in B(x_0, \de_1)$, $y = \si(x)$ is the unique solution of $G(x, \si(x)) = 0$ in $B(y_0, \de_2)$. 
}

The above is proved in the same fashion as the usual implicit function theorem, see, e.g.,~\cite[Section 2.2]{ChowHale}. The key point is that the bounds~\eqref{eq:IFT} give uniform control on the size of the open set on which the Banach contraction mapping theorem can be applied. 
\end{rem}

\begin{proof}[Proof of Lemma~\ref{l:mod2}] 
We sketch the proof. We suppress the dependence of all functions on $t \in J$ below as no constants appearing in the proof will depend on $t \in J$. Using the assumption~\eqref{eq:d+small} we can find $\la_0, \mu_0>0$ and $a_0, b_0 \in \R$  so that defining 
\EQ{
g_0 &:=  u - ( Q_{\la_0} - Q_{\mu_0})  \\
\dot g_0&:= u_t - (b_0 \Lam Q_{\U \lam_0} + a_0 \Lam Q_{\U \mu_0}) 
}
we have 
\EQ{ \label{eq:l0m0-bound} 
\| g_0 \|_{H} + \nu_0^{-\frac{k}{2}} \| \dot g_0 \|_{L^2}  
+ \nu_0^k+ a_0^2\big( 1+ \abs{a_0} \nu_0^{-\frac{k}{2}}\big)   + b_0^2\big( 1+ \abs{b_0}\nu_0^{-\frac{k}{2}}\big)   < 2 \bfd_+( \bs u) <  2 \eta_0
}
where we have set $\nu_0 := \la_0/ \mu_0$ as usual. 
Set, 
\EQ{ \label{eq:w0-def} 
w_0 &:= u - \Phi( \mu_0, \lam_0, a_0, b_0)  \\
\dot w_0 &:=  u_t  -   \dot \Phi( \mu_0, \lam_0, a_0, b_0)
}
and note that by~\eqref{eq:l0m0-bound} and the definition of $\bs \Phi$  
 we have 
\EQ{
\| (w_0,  \nu_0^{-\frac{k}{2}}\dot w_0) \|_{\HH} <  C_0 \bfd_+( \bs u) =: C_0\eta_1.
}
for some uniform constant $C_0$. 
Define functions $\bs F = (F, \dot F)$ by, 
\EQ{
F(h, \dot h,  \mu, \la, a, b):= h + \Phi( \mu_0, \lam_0, a_0, b_0) - \Phi( \mu, \lam, a, b) \\
\dot F(h, \dot h,  \mu, \la, a, b):= \dot h + \dot\Phi( \mu_0, \lam_0, a_0, b_0) - \dot \Phi(\mu, \lam,  a, b)
}
Note that $\bs F( 0, 0, \mu_0, \la_0, a_0, b_0) = (0, 0)$.
Next, define a function $G$,
\EQ{
G( h, \dot h, \mu, \la, a, b) &=  \Big( \frac{1}{\mu} \ang{ \La Q_{\U \mu}  \mid F(\bs h, \mu, \lam, a, b)}, \, \frac{1}{\lam} \ang{ \Lam Q_{\U \lam} \mid F( \bs h, \mu, \lam, a, b)} ,  \\
&\quad \quad  \,  \nu_0^{-\frac{k}{2}}\ang{ \Lam Q_{\U \mu} \mid  \dot F( \bs h, \mu, \lam, a ,b )}, \,\nu_0^{-\frac{k}{2}} \ang{ \Lam Q_{\U \lam} \mid  \dot F(  \bs h , \mu, \lam, a, b)} \Big) 
}
noting that $G(0, 0,  \mu_0, \la_0, a_0, b_0) = (0, 0, 0, 0)$. 
 We now introduce new variables 
\begin{align} 
\ell&:= \log \lam, \quad m := \log \mu \\
\ti a& := \nu_0^{-\frac{k}{2}} a , \quad \ti b :=  \nu_0^{-\frac{k}{2}} b, \\
\ti h&:= h, \quad  \dot {\ti  h}  :=  \nu_0^{-\frac{k}{2}} \dot h, \quad \bs{ \ti h} := ( \ti h, \dot{\ti h})
\end{align} 
Set 
\EQ{
\ti G(  \bs {\ti h}, m, \ell, \ti a,  \ti b) = G( \bs h, \mu, \lam, a, b), \quad  \bs{\ti F}( \bs{\ti  h}, m, \ell, \ti a,  \ti b) = \bs F( \bs h, \mu, \lam, a, b) 
}
Note that in the new variables, $\lam \p_\lam = \p_\ell$,  $\mu \p_\mu = \p_m$, $\nu_0^{\frac{k}{2}} \p_b = \p_{\ti b}$, and $\nu_0^{\frac{k}{2}} \p_a = \p_{\ti a}$.  A straightforward computation using~\eqref{eq:l0m0-bound}  gives, 
\EQ{
D_{m, \ell, \ti a, \ti b} G \rest_{ ( 0, 0, m_0, \ell_0, \ti a_0,  \ti b_0)}  = \pmat{ \kappa_k - o_{\eta_1}( 1) & o_{\eta_1}( 1) &o_{\eta_1}( 1) &o_{\eta_1}( 1) \\ o_{\eta_1}( 1) & \kappa_k - o_{\eta_1}( 1) & o_{\eta_1}( 1) & o_{\eta_1}( 1) \\ o_{\eta_1}( 1)& o_{\eta_1}( 1)&  \kappa_k- o_{\eta_1}( 1)&o_{\eta_1}( 1) \\   o_{\eta_1}( 1)&  o_{\eta_1}( 1)& o_{\eta_1}( 1)& \kappa_k- o_{\eta_1}( 1)} 
}
where $o_{\eta_1}(1) \to 0$  as $\eta_1 \to 0$. The matrix is diagonally dominant for $\eta_1$ small enough and hence invertible. We now apply the quantitative version of the implicit function theorem as Remark~\ref{rem:IFT}. First we need to verify the bounds~\eqref{eq:IFT}. The second estimate in~\eqref{eq:IFT} is clear since 
\EQ{
\Big| \ti G( \ti h, \dot{ \ti h}, m_0, \ell_0, &\ti a_0, \ti b_0) \Big| = \bigg| \Big( \frac{1}{\mu} \ang{ \La Q_{\U \mu}  \mid h }, \, \frac{1}{\lam} \ang{ \Lam Q_{\U \lam} \mid h} ,  \nu_0^{-\frac{k}{2}}\ang{ \Lam Q_{\U \mu} \mid  \dot h }, \,\nu_0^{-\frac{k}{2}} \ang{ \Lam Q_{\U \lam} \mid  \dot h} \Big)  \bigg| \\
& = \bigg| \Big( \frac{1}{\mu} \ang{ \La Q_{\U \mu}  \mid \ti  h }, \, \frac{1}{\lam} \ang{ \Lam Q_{\U \lam} \mid \ti  h} ,  \nu_0^{-\frac{k}{2}}\ang{ \Lam Q_{\U \mu} \mid \nu_0^{\frac{k}{2}} \dot{\ti h} }, \,\nu_0^{\frac{k}{2}} \ang{ \Lam Q_{\U \lam} \mid \nu_0^{-\frac{k}{2}} \dot{\ti h}} \Big)   \bigg|  \\
& \lesssim \| \ti h \|_{H} + \| \dot{ \ti h} \|_{L^2} 
}
For the first estimate in~\eqref{eq:IFT}, we define $L_0:= D_{m, \ell, \ti a, \ti b} G \rest_{ ( 0, 0, m_0, \ell_0, \ti a_0,  \ti b_0)}$.  A straightforward computation reveals that 
\EQ{
\abs{L_{0, ij}  - ( D_{m, \ell, \ti a, \ti b} G )_{ij} } = o_{\eta_1}(1) 
}
Applying Remark~\ref{rem:IFT} we  obtain a mapping $\varsigma: B_{\HH }(0, C_0 \eta_1) \to B_{\R^4}((m_0, \ell_0, \ti a_0,  \ti b_0),  C \eta_1)$ so that for all $( \ti h,  \dot{\ti h} , m_1, \ell_1,  \ti a_1, \ti b_1) \in B_{\HH}( 0, C_0 \eta_1)  \times B_{ \R^4}( (m_0, \ell_0, \ti a_0, \ti b_0), C \eta_1)$ we have 
\EQ{
G(  \ti h, \dot{\ti h}, m_1, \ell_1, \ti a_1, \ti b_1) = (0,0, 0, 0)  \Longleftrightarrow (m_1, \ell_1, \ti a_1, \ti b_1) = \varsigma(\ti h, \dot{\ti  h}) 
}
We conclude by setting 
\EQ{
&(m, \ell, \ti a,  \ti b)  := \varsigma( \ti w_0,  \dot {\ti w_0})  , \quad  \bs{ w}:= \ti{\bs F} (\bs{\ti w_0},  \varsigma( \ti w_0,  \dot {\ti w_0}) )
}
where $\ti {\bs w_0 }:= ( w_0, \nu_0^{-\frac{k}{2}} \dot w_0)$ with 
 $\bs w_0$ is as in~\eqref{eq:w0-def}. Undoing the change of variables, we then set 
\EQ{
 \la &= e^\ell, \quad \mu = e^m, \quad  a := \nu_0^{\frac{k}{2}} \ti a, \quad b := \nu_0^{\frac{k}{2}} \ti b
}
Note that by construction we have 
\EQ{ \label{eq:mod-bound} 
\abs{ \lam/ \lam_0 - 1} + \abs{ \mu/ \mu_0 - 1} + \nu_0^{-\frac{k}{2}} \abs{ a- a_0} + \nu_0^{-\frac{k}{2}} \abs{ b- b_0} \lesssim \eta_1
}
from which one can deduce that 
\EQ{
 \| w \|_H + \nu^{-\frac{k}{2}} \| \dot w \|_{L^2}  \lesssim \eta_1
}
and thus the bound~\eqref{eq:w-d-bound}. 
For the standard fact that $(\mu, \la, a, b)$ can be taken to be $C^1(J)$ functions, see~\cite[Remark 3.13]{JL1}. The coercivity estimate~\eqref{eq:coerce1} is standard given the existing literature, and one may follow the arguments in~\cite[Lemma 5.4]{JJ-AJM}; see Lemma~\ref{l:loc-coerce} for the main ingredient in the proof. 
\end{proof}


\subsubsection{Equations satisfied by $\bs w, \bs \Phi$} In this subsection we record the dynamical equation satisfied by $\bs w(t)$ and related identities satisfied by $\bs \Phi$. 

From~\eqref{eq:wmk} and the definition of $\bs w(t)$ in Lemma~\ref{l:mod2} we arrive at the equation satisfied by $\bs w(t)$. 
\EQ{ \label{eq:weq} 
\p_t w &= \dot w + (\dot \Phi - \p_t \Phi)   \\
\p_t \dot w &= \De w + \Big( -  \p_t \dot \Phi  + \De \Phi  - \frac{1}{r^2} f( \Phi) \Big)    - \frac{1}{r^2}  \Big ( f( \Phi +w) - f( \Phi) \Big) 
}
We also record several identities satisfied by  $( \Phi, \dot \Phi)$ related to the right-hand side of~\eqref{eq:weq}.  First,  
 using that $\mu \nu' =  \la' - \nu \mu'$, we have 
\EQ{  \label{eq:dotPhi-ptPhi}
\dot \Phi - \p_t \Phi 
&=  (b+ \la') \La Q_{\U \la}  + b^2( b + \la') \La A_{\U \la}  - 2 b \la ( b' + \gamma_k \frac{\nu^{k}}{\la}) A_{\U \la}  \\
 & \quad +  \nu^k( b + \la') \La B_{\U \la}  - k \nu^k( b + \la')  B_{ \U\la} - k \nu^{k+1}( a - \mu')  B_{ \U\la} \\
 & \quad +(a - \mu') \La Q_{\U \mu}  + a^2(  a - \mu') \La A_{\U \mu} + 2a \mu ( a' +   \ga_k \frac{\nu^k}{\mu}) A_{\U\mu}  \\
& \quad + \nu^k(   a- \mu') \La \ti B_{\U \mu}+ k \nu^{k-1} (  b + \la') \ti B_{\U \mu} + k \nu^{k} (  a- \mu') \ti B_{\U \mu}   
}
Next, note that using~\eqref{eq:LLb2T} we have 
\EQ{ \label{eq:eqPhi} 
-\De \Phi  + \frac{1}{r^2} f( \Phi) &=  \gamma_k \frac{\nu^k}{\lam} \Lam Q_{\U \lam} - \frac{b^2}{\lam} \Lam_0 \Lam Q_{\U \lam}  + \gamma_k  \frac{\nu^k}{\mu}  \Lam Q_{\U \mu} + \frac{a^2}{\mu} \Lam_0 \Lam Q_{\U \mu}  \\
 &\quad -\frac{1}{r^2} \Big( f(Q_\la - Q_\mu)  - f(Q_\la) + f(Q_\mu)  - 4  \big(\frac{r}{ \mu}\big)^k (\La Q_{\la})^2  -  4 \big(\frac{r}{ \la}\big)^{-k} (\La Q_{\mu})^2 \Big)\\
&\quad -\frac{1}{r^2} \Big( f(\Phi) - f(Q_\la - Q_\mu) -  f'(Q_\la- Q_\mu) (b^2 T_\la - a^2 \ti T_{\mu} \Big)  \\
&\quad -\frac{1}{r^2}  \Big(f'(Q_\la - Q_\mu) (b^2 T_\la - a^2 \ti T_\mu)  - f'(Q_\la) b^2 T_\la + f'(Q_\mu) a^2 \ti T_\mu\Big)
}
and hence,  
\begingroup
\allowdisplaybreaks
\begin{align} \label{eq:ptdotPhi} 
-&\p_t  \dot \Phi + \De \Phi  - \frac{1}{r^2} f( \Phi)  \\ 
& =  (b'+ \gamma_k \frac{\nu^k}{\la})\Bigg( -  \La Q_{\U \la} - 3   b^2 \La A_{\U \la} + 2 \gamma_k\nu^k  A_{\U \la}   -  \nu^k \La B_{\U \la}+  k \nu^k B_{\U \la} - k \nu^{k-1} \ti B_{\U \mu}\Bigg) \\
& \quad  + (a'  +  \gamma_k \frac{\nu^k}{\mu})\Bigg( -  \La Q_{\U \mu}   -3a^2 \La A_{\U \mu}   -  2  \gamma_k  \nu^k A_{\U \mu} +  \nu^k \La  \ti B_{\U \mu} + k \nu^{k+1} B_{\U \lam}  -  k \nu^k \ti B_{\U \mu}  \Bigg)  \\ 
& \quad + (b+\la') \Bigg( \frac{b}{\la}  \La_0 \La Q_{\U \la} +  \frac{b^3}{\la} \La_0 \La A_{\U \la} - 2 \gamma_k \frac{ b \nu^k}{\la}  \La_0 A_{\U \la} - k   \frac{ b \nu^{k}}{\la} \La_0 B_{ \U \la} +  \frac{b \nu^k}{\la} \La_0 \La B_{\U \la}  \\
& \qquad \qquad  \quad  + 2k \gamma_k \frac{ b \nu^{k}}{\la}  A_{\U \la}- k  \frac{ b\nu^k}{\lam}\La B_{\U\la}+   k^2\frac{b\nu^k}{\lam}  B_{ \U \la} - 2k \ga_k  \frac{a\nu^k}{\lam}  A_{\U \mu}- k  \frac{a\nu^k}{\lam}  \La \ti B_{\U \mu}  \\
& \qquad \qquad  \quad   - k(k-1) \frac{ b \nu^{k-1}}{\lam} \ti B_{\U \mu}  +  k(k+1) a \frac{\nu^{k+1}}{\lam} B_{\U \lam} - k a \frac{\nu^{k+1}}{\lam}  \La_0 B_{\U \lam}  -  k^2 a \frac{\nu^k}{\lam} \ti B_{\U \mu}   \Bigg)  \\ 
 & \quad + (a-\mu') \Bigg(-\frac{a}{\mu} \La_0 \La Q_{\U \mu}  - \frac{ a^3 \nu}{\lam}   \La_0 \La A_{\U \mu} - 2 \gamma_k \frac{a \nu^{k+1}}{\lam}  \La_0 A_{\U \mu}- \frac{a \nu^{k+1}}{\lam}  \La_0 \La \ti B_{\U \mu}  - k  \frac{ b \nu^{k}}{\lam}  \La_0  \ti B_{\U \mu}  \\
& \qquad \qquad  \quad   + 2k \gamma_k  \frac{b\nu^{k+1}}{\lam}   A_{\U \la} - k  \frac{b\nu^{k+1}}{\lam}  \La B_{\U\la} +  k^2 \frac{b \nu^{k+1}}{\la}  B_{ \U \la}   - 2k \ga_k \frac{a \nu^{k+1}}{\lam}  A_{\U \mu}  - k \frac{ a \nu^{k+1}}{\lam}  \La \ti B_{\U \mu} \\
& \qquad \qquad  \quad    - k(k-1) \frac{  b  \nu^{k}}{\lam}  \ti B_{\U \mu}  +   k(k+1) a \frac{\nu^{k+2}}{\lam} B_{\U \lam}  - k^2 a \frac{\nu^{k+1}}{\lam} \ti B_{\U \mu}  - k a \frac{\nu^{k+1}}{\lam}  \Lam_0 \ti B_{\U \mu} \Bigg)  \\
&\quad   + 3\gamma_k \frac{\nu^k}{\la}b^2 \La A_{\U \la}   -  \frac{b^4}{\la} \La_0 \La A_{\U \la}     -2 \gamma_k^2  \frac{\nu^{2k}}{\la}  A_{\U \la} + 2 \gamma_k\frac{ b^2 \nu^k}{\la}  \La_0 A_{\U \la}     +   \gamma_k  \frac{\nu^{2k}}{\lam} \La B_{\U \la} \\
&\quad    - 2k \gamma_k  \frac{b^2 \nu^k}{\lam}  A_{\U \la}  - 2k \gamma_k  \frac{a b\nu^{k+1}}{\lam}  A_{\U \la}  + k  \frac{b^2 \nu^k}{\lam}  \La B_{\U\la}  + k  \frac{ab \nu^{k+1}}{\lam}  \La B_{\U\la} \\
&\quad    - \frac{b^2 \nu^k}{\la} \La_0 \La B_{\U \la}   -  k \frac{\nu^{2k}}{\lam}  B_{\U \la} +  \frac{ k b^2 \nu^{k}}{\la} \La_0 B_{ \U \la} -  k^2 \frac{b^2 \nu^{k} }{\lam} B_{ \U \la}    -   k^2 \frac{a b  \nu^{k+1}}{\lam}  B_{ \U \la}  \\
& \quad    - k(k+1) ab \frac{\nu^{k+1}}{\lam} B_{\U \lam} + k a b\frac{\nu^{k+1}}{\lam}  \La_0 B_{\U \lam}  -   k(k+1) a^2 \frac{\nu^{k+2}}{\lam} B_{\U \lam}  -  k  \gamma_k \frac{\nu^{2k+2}}{\lam} B_{\U \lam}   \\ 
 &\quad   + 3  \gamma_k \frac{a^2\nu^k}{\mu} \La A_{\U \mu}    + \frac{a^4\nu}{\lam} \La_0 \La A_{\U \mu}   +   2  \gamma_k^2 \frac{\nu^{2k+1}}{\la} A_{\U \mu}   + 2k \ga_k \frac{ a b \nu^k}{\lam}  A_{\U \mu} +  2k \ga_k\frac{a^2 \nu^{k+1}}{\lam}  A_{\U \mu}\\
 &\quad  + 2 \gamma_k \frac{a^2 \nu^{k+1}}{\lam}   \La_0 A_{\U \mu}   -   \gamma_k  \frac{\nu^{2k+1}}{\lam} \La  \ti B_{\U \mu}  +  k  \frac{ ab \nu^k}{\lam}  \La \ti B_{\U \mu} +  k \frac{a^2 \nu^{k+1}}{\lam}  \La \ti B_{\U \mu}\\
& \quad   +  \frac{ a^2 \nu^{k+1}}{\lam}  \La_0 \La \ti B_{\U \mu}   + k \gamma_k \frac{\nu^{2k-1}}{\lam} \ti B_{\U \mu}  +  k \gamma_k  \frac{\nu^{2k+1}}{\lam} \ti B_{\U \mu}   + k^2 a b \frac{\nu^k}{\lam} \ti B_{\U \mu}  + k(k-1) \frac{b^2 \nu^{k-1}}{\lam}  \ti B_{\U \mu}   \\
& \quad + k(k-1) \frac{ab  \nu^{k}}{\lam}  \ti B_{\U \mu}  + k \frac{ a b \nu^{k}}{\lam}  \La_0  \ti B_{\U \mu}  +  k^2 a^2 \frac{\nu^{k+1}}{ \lam} \ti B_{\U \mu} + k a^2 \frac{ \nu^{k+1}}{\lam}  \Lam_0 \ti B_{\U \mu}\\
 &\quad -\frac{1}{r^2} \Big( f(Q_\la - Q_\mu)  - f(Q_\la) + f(Q_\mu)  - 4  \big(\frac{r}{ \mu}\big)^k (\La Q_{\la})^2  -  4 \big(\frac{r}{ \la}\big)^{-k} (\La Q_{\mu})^2 \Big)\\
&\quad -\frac{1}{r^2} \Big( f(\Phi) - f(Q_\la - Q_\mu) -  f'(Q_\la- Q_\mu) (b^2 T_\la - a^2 \ti T_{\mu} \Big)  \\
&\quad -\frac{1}{r^2}  \Big(f'(Q_\la - Q_\mu) (b^2 T_\la - a^2 \ti T_\mu)  - f'(Q_\la) b^2 T_\la + f'(Q_\mu) a^2 \ti T_\mu\Big)
\end{align} 
\endgroup

\subsubsection{Bounds on  $(\Phi, \dot \Phi)$ and related estimates} 
We require a collection of estimates. The reader may skip this technical subsection on the first pass and refer back when the estimates established here arise in later sections.

\begin{lem}    \label{l:fest1} 
Let $\bs \Phi = ( \Phi, \dot \Phi)$ be defined as in~\eqref{eq:Phidef} and let $(\mu, \lam, a, b)$ satisfy $ \nu:= \lam/ \mu \ll 1$ and $ \abs{a}, \abs{b} \ll 1$. 
Then, for $\al = 1, 2, 3$ we have 
\EQ{ \label{eq:ff'L2} 
 & \Big\| r^{-\al}  \Big( f(Q_\la - Q_\mu)  - f(Q_\la) + f(Q_\mu)  - 4  \big(\frac{r}{ \mu}\big)^k (\La Q_{\la})^2  -  4 \big(\frac{r}{ \la}\big)^{-k} (\La Q_{\mu})^2 \Big) \Big\|_{L^2}  \\
  &\qquad \qquad\lesssim \nu^{2k} \lam^{-\al +1}    \\
 &   \Big\| r^{-\al}  \Big( f(\Phi) - f(Q_\la - Q_\mu) -  f'(Q_\la- Q_\mu) (b^2 T_\la - a^2 \ti T_{\mu} \Big) \Big \|_{L^2} \\
  &\qquad \qquad \lesssim b^4 \lam^{-\al +1}  + a^4\nu^{\al-1}\lam^{-\al +1}  +\nu^{2k}  \lam^{-\al +1}    \\
&   \Big\| r^{-\al}  \Big(f'(Q_\la - Q_\mu) (b^2 T_\la - a^2 \ti T_\mu)  - f'(Q_\la) b^2 T_\la + f'(Q_\mu) a^2 \ti T_\mu\Big) \Big\|_{L^2} \\
 &\qquad \qquad\lesssim b^2 \nu^{k-1} \lam^{-\al +1} +a^2 \nu^{k+ \al - 2} \lam^{-\al +1}   + \nu^{2k  -1} \lam^{-\al +1} 
} 
 We also have the estimates, 
 \EQ{ \label{eq:ff'mu} 
 &  \abs{ \ang{ \La Q_{\U \mu} \mid  \frac{1}{r^2} \Big( f(Q_\la - Q_\mu)  - f(Q_\la) + f(Q_\mu)  - 4  \big(\frac{r}{ \mu}\big)^k (\La Q_{\la})^2  -  4 \big(\frac{r}{ \la}\big)^{-k} (\La Q_{\mu})^2 \Big) }}  \\
 &\qquad \qquad \lesssim  \frac{\nu^{2k+1}}{\la}   \\
& \abs{ \ang{ \La Q_{\U \mu} \mid \frac{1}{r^2} \Big( f(\Phi) - f(Q_\la - Q_\mu) -  f'(Q_\la- Q_\mu) (b^2 T_\la - a^2 \ti T_{\mu} \Big)}}  \\
 &\qquad \qquad\lesssim  \frac{b^4 \nu^{k+1}}{\la} + \frac{\nu^{3k+1}}{\la} + \frac{a^4\nu}{\la} + \frac{\nu^{2k+1}}{\la}  \\
& \abs{ \ang{ \La Q_{\U \mu} \mid \frac{1}{r^2}  \Big(f'(Q_\la - Q_\mu) (b^2 T_\la - a^2 \ti T_\mu)  - f'(Q_\la) b^2 T_\la + f'(Q_\mu) a^2 \ti T_\mu\Big)}}  \\
 &\qquad \qquad\lesssim   \frac{b^2 \nu^{k-1}}{\la}  + \frac{a^2 \nu^{k+1}}{\la}+ \frac{\nu^{2k-1}}{\la}
 }
 and
 \EQ{ \label{eq:ff'la} 
& \abs{ \ang{ \La Q_{\U \la} \mid  \frac{1}{r^2} \Big( f(Q_\la - Q_\mu)  - f(Q_\la) + f(Q_\mu)  - 4  \big(\frac{r}{ \mu}\big)^k (\La Q_{\la})^2  -  4 \big(\frac{r}{ \la}\big)^{-k} (\La Q_{\mu})^2 \Big) }}\\
 &\qquad \qquad \lesssim \frac{\nu^{2k}}{\la}   \\
& \abs{ \ang{ \La Q_{\U \la} \mid \frac{1}{r^2} \Big( f(\Phi) - f(Q_\la - Q_\mu) -  f'(Q_\la- Q_\mu) (b^2 T_\la - a^2 \ti T_{\mu} \Big)}} \\
 &\qquad \qquad \lesssim  \frac{b^4}{\la} + \frac{\nu^{2k}}{\la} + \frac{a^4 \nu^k}{\la} \\
& \abs{ \ang{ \La Q_{\U \la} \mid \frac{1}{r^2}  \Big(f'(Q_\la - Q_\mu) (b^2 T_\la - a^2 \ti T_\mu)  - f'(Q_\la) b^2 T_\la + f'(Q_\mu) a^2 \ti T_\mu\Big)}}\\
 &\qquad \qquad \lesssim  \frac{b^2 \nu^{k}}{\la} +  \frac{\nu^{2k-o(1)}}{\la}  + \frac{a^2 \nu^k}{\la}  
 }
 \end{lem}

 \begin{lem} \label{l:fest2} 
 Let $\bs \Phi = ( \Phi, \dot \Phi)$ be defined as in~\eqref{eq:Phidef} and let $(\mu, \lam, a, b)$ satisfy $ \nu:= \lam/ \mu \ll 1$ and $ \abs{a}, \abs{b} \ll 1$. 
Then, 
\EQ{ \label{eq:ff'L2dr} 
 & \Big\| r^{-2} \p_r \Big( f(Q_\la - Q_\mu)  - f(Q_\la) + f(Q_\mu)  - 4  \big(\frac{r}{ \mu}\big)^k (\La Q_{\la})^2  -  4 \big(\frac{r}{ \la}\big)^{-k} (\La Q_{\mu})^2 \Big) \Big\|_{L^2}  \lesssim \frac{\nu^{2k}}{\la^2} \\
 &   \Big\| r^{-2}  \p_r\Big( f(\Phi) - f(Q_\la - Q_\mu) -  f'(Q_\la- Q_\mu) (b^2 T_\la - a^2 \ti T_{\mu} \Big) \Big \|_{L^2} \lesssim \frac{b^4}{\la^2} + \frac{a^4\nu^2}{\la^2} +\frac{ \nu^{2k}}{\la^2} \\
&   \Big\| r^{-2} \p_r  \Big(f'(Q_\la - Q_\mu) (b^2 T_\la - a^2 \ti T_\mu)  - f'(Q_\la) b^2 T_\la + f'(Q_\mu) a^2 \ti T_\mu\Big) \Big\|_{L^2}  \\
 &\qquad \qquad\lesssim \frac{b^2 \nu^{k-1}}{\la^2} + \frac{a^2 \nu^k}{\la^2}  + \frac{\nu^{2k  -1}}{\la^2} 
}  \end{lem}

 \begin{lem}  \label{l:w^2est} 
 Let $ \bs w \in \HH \cap \HH^2$, let $\bs \Phi = ( \Phi, \dot \Phi)$ be defined as in~\eqref{eq:Phidef}, and  let $(\mu, \lam, a, b)$ satisfy $ \nu:= \lam/ \mu \ll 1$ and $ \abs{a}, \abs{b} \ll 1$. 
Then, 
\begin{align} 
  \left\| \frac{1}{r} \Big( f( \Phi + w) - f( \Phi) - f'(\Phi) w \Big) \right\|_{L^2} &\lesssim   \|  w \|_{H}^2     \label{eq:w^2L21} \\
  \left\| \frac{1}{r^2} \Big( f( \Phi + w) - f( \Phi) - f'(\Phi) w \Big) \right\|_{L^2} &\lesssim  \frac{1}{\la} \|  w \|_{H}^2  +  \| w \|_H^2 \| \p_r w \|_H    \label{eq:w^2L2} \\
  \left\| \frac{1}{r^3} \Big( f( \Phi + w) - f( \Phi) - f'(\Phi) w \Big) \right\|_{L^2} &\lesssim  \frac{1}{\la} \| \p_r w \|_{H} \| w \|_{H}  + \| \p_r w \|_{H}^2 \| w \|_H \label{eq:w^2L23} \\ 
   \left\| \frac{1}{r^2} \p_r \Big( f( \Phi + w) - f( \Phi) - f'(\Phi) w \Big) \right\|_{L^2} &\lesssim  \frac{1}{\la} \| \p_r w \|_{H} \| w \|_{H}  + \| \p_r w \|_{H}^2 \| w \|_H \label{eq:w^2L22}  \\ 
    \abs{ \ang{ \La Q_{\U \la} \mid \frac{1}{r^2} \Big( f( \Phi + w) - f( \Phi) - f'( Q_\la) w \Big) }} &\lesssim  \frac{1}{\la} \|w \|_{H}^2 + \frac{1}{\la} \| w \|_H \Big( \nu^k + b^2 + a^2 \nu^k\Big)\label{eq:w^2la} \\ 
 \abs{ \ang{ \La Q_{\U \mu} \mid \frac{1}{r^2} \Big( f( \Phi + w) - f( \Phi) - f'( Q_\mu) w \Big) }} &\lesssim \frac{1}{\mu} \| w \|_H^2  +  \frac{1}{\mu} \| w \|_H \Big( \nu^k + b^2 \nu^{k-2}   + a^2 \Big)  \label{eq:w^2mu} 
 \end{align} 
 \end{lem}


We will often make use of the following pointwise bounds, which follow directly from the definition of $\Phi$ and trigonometric identities.  
\begin{lem}  \label{l:cos2Phi} 
Let $\bs \Phi = ( \Phi, \dot \Phi)$ be defined as in~\eqref{eq:Phidef} and let $(\mu, \lam, a, b)$ satisfy $ \nu:= \lam/ \mu \ll 1$ and $ \abs{a}, \abs{b} \ll 1$. 
Then, 
\begin{align} \label{eq:cos2Phi1} 
\abs{ \cos 2 \Phi - \cos 2Q_\la  } &\lesssim \abs{ \La Q_{\la}( \La Q_\mu   + b^2 A_\la + \nu^k B_\la  + a^2 A_\mu + \nu^k\ti B_\mu)}   \\
& \quad + \La Q_\mu^2 + b^4 A_\la^2  + \nu^{2k} B_\la^2 + a^4 A_\mu^2   + \nu^{2k} \ti B_\mu^2 \\
\abs{ \cos 2 \Phi - \cos 2Q_\mu  } & \lesssim  \abs{\La Q_{\mu} (\La Q_\la   + b^2 A_\la + \nu^k B_\la  + a^2 A_\mu + \nu^k\ti B_\mu) }   \label{eq:cos2Phi2} \\
& \quad + \La Q_\la^2 + b^4 A_\la^2  + \nu^{2k} B_\la^2 + a^4 A_\mu^2   + \nu^{2k} \ti B_\mu^2 \\
\label{eq:sin(Phi-Qla)} 
\abs{\sin(\Phi- Q_\la)} &\lesssim \La Q_\mu + b^2 \abs{A_\la} + \nu^k \abs{B_\la} + a^2 \abs{A_\mu} + \nu^k \abs{\ti B_\mu} 
\end{align} 

\end{lem} 

\begin{proof}[Proof of Lemma~\ref{l:fest1}] 
For convenience we fix the  $\al = 2$ in the proof and note that changing the weight to $r^{-1}$ or $r^{-3}$ simply changes the scaling as in the statement of the estimates~\eqref{eq:ff'L2}. 
First note that 
\EQ{ \label{eq:fQQ} 
f(Q_\la - Q_\mu) &- f(Q_\la)+ f(Q_\mu) - 4  \big(\frac{r}{ \mu}\big)^k (\La Q_{\la})^2  -  4 \big(\frac{r}{ \la}\big)^{-k} (\La Q_{\mu})^2  \\
&  = \Big(\sin 2Q_\mu -   4  \big(\frac{r}{ \mu}\big)^k   \Big)[\La Q_\la]^2  - \Big(\sin 2Q_\la+ 4 \big(\frac{r}{ \la}\big)^{-k} \Big) [\La Q_\mu]^2 
}
Fix a large constant $R>0$ and note the estimates, 
\EQ{ \label{eq:sinrk} 
\abs{\sin 2Q(r) -   4  r^k}  &\lesssim r^{3k}   \mif r \le \frac{1}{R} \\
\abs{\sin 2Q(r) -   4  r^k} & \lesssim r^k \mif  r \ge \frac{1}{R}  \\
\abs{\sin 2Q(r) +  4  r^{-k}} & \lesssim r^{-k} \mif  r \le R \\ 
\abs{\sin 2Q(r) +  4  r^{-k}} & \lesssim r^{-3k} \mif  r \ge R
}
From~\eqref{eq:fQQ} we have 
\EQ{
 \Big\| r^{-2} & \Big( f(Q_\la - Q_\mu)  - f(Q_\la) + f(Q_\mu)  - 4  \big(\frac{r}{ \mu}\big)^k (\La Q_{\la})^2  -  4 \big(\frac{r}{ \la}\big)^{-k} (\La Q_{\mu})^2 \Big) \Big\|_{L^2}  \\
 & \lesssim \frac{1}{\mu}  \Big \| \frac{1}{r^2} \Big(\sin 2Q -   4  r^k   \Big)[\La Q_\nu]^2  \Big \|_{L^2}  +  \frac{1}{\mu}  \Big \| \frac{1}{r^2} \Big(\sin 2Q_\nu+ 4 \big(\frac{r}{ \nu}\big)^{-k} \Big) [\La Q]^2  \Big \|_{L^2}
}
We estimate each of the terms on the right above. Consider the first term. We divide the integrand into three regions $r \le R \nu$, $R \nu \le r \le R^{-1}$ and $R^{-1} \le r$. We have 
\EQ{
 \frac{1}{\mu}  \Big \| \frac{1}{r^2} \Big(\sin 2Q -   4  r^k   \Big)[\La Q_\nu]^2  \Big \|_{L^2(r \le R \nu)} &\lesssim   \frac{1}{\mu}  \Big \| r^{3k-2} (r/ \nu)^{2k}  \Big \|_{L^2(r \le R \nu)}  \lesssim \frac{\nu^{3k}}{\la}  \\
  \frac{1}{\mu}  \Big \| \frac{1}{r^2} \Big(\sin 2Q -   4  r^k   \Big)[\La Q_\nu]^2  \Big \|_{L^2( R \nu\le r \le R^{-1})} & \lesssim  \frac{1}{\mu}  \Big \| r^{3k-2} \nu^{2k} r^{-2k}   \Big \|_{L^2(R \nu \le r \le 1)}  \lesssim \frac{\nu^{2k+1}}{\la} \\
    \frac{1}{\mu}  \Big \| \frac{1}{r^2} \Big(\sin 2Q -   4  r^k   \Big)[\La Q_\nu]^2  \Big \|_{L^2( r \ge R^{-1})} & \lesssim  \frac{1}{\mu}  \Big \| r^{k-2} \nu^{2k} r^{-2k}   \Big \|_{L^2(r \ge R^{-1} )}  \lesssim \frac{\nu^{2k+1}}{\la}
}
For the second term we also divide the integrand into regions $r \le R \nu$, $R \nu \le r \le {R^{-1}} $ and $R^{-1} \le r$. We have 
\EQ{
 \frac{1}{\mu}  \Big \| \frac{1}{r^2} \Big(\sin 2Q_\nu+ 4 \big(\frac{r}{ \nu}\big)^{-k} \Big) [\La Q]^2  \Big \|_{L^2( r \le R \nu)}& \lesssim  \frac{1}{\mu}  \Big \|  \nu^k r^{-k} r^{2k-2}   \Big \|_{L^2( r \le R \nu)} \lesssim \frac{ \nu^{2k}}{\la}  \\
  \frac{1}{\mu}  \Big \| \frac{1}{r^2} \Big(\sin 2Q_\nu+ 4 \big(\frac{r}{ \nu}\big)^{-k} \Big) [\La Q]^2  \Big \|_{L^2( R \nu \le r \le R^{-1} )}& \lesssim  \frac{1}{\mu}  \Big \| \nu^{3k} r^{-3k} r^{2k-2}    \Big \|_{L^2( R \nu \le r \le R^{-1} )} \lesssim \frac{ \nu^{2k}}{\la}  \\
   \frac{1}{\mu}  \Big \| \frac{1}{r^2} \Big(\sin 2Q_\nu+ 4 \big(\frac{r}{ \nu}\big)^{-k} \Big) [\La Q]^2  \Big \|_{L^2(  r \ge R^{-1}  )}& \lesssim  \frac{1}{\mu}  \Big \|\nu^{3k} r^{-3k} r^{-2k-2}     \Big \|_{L^2( r \ge  R^{-1} )} \lesssim \frac{ \nu^{3k}}{\la}
}
This completes the first estimate in~\eqref{eq:ff'L2}. We argue similarly to prove the estimates in the first lines of~\eqref{eq:ff'mu} and~\eqref{eq:ff'la}. Indeed, arguing in the same way for~\eqref{eq:ff'mu}  yields,  
\EQ{
& \abs{ \ang{ \La Q_{\U \mu} \mid  \frac{1}{r^2} \Big( f(Q_\la - Q_\mu)  - f(Q_\la) + f(Q_\mu)  - 4  \big(\frac{r}{ \mu}\big)^k (\La Q_{\la})^2  -  4 \big(\frac{r}{ \la}\big)^{-k} (\La Q_{\mu})^2 \Big) }}  \\
 &  \lesssim \frac{1}{\mu}  \int_0^\infty \La Q ( \sin 2 Q - 4 r^k) [\La Q_\nu]^2 \, \frac{\ud r }{r}  + \frac{1}{\mu} \int_0^\infty [\La Q]^3 ( \sin 2 Q_\nu + 4(r/\nu)^{-k}) \, \frac{\ud r}{r}  \\
 & \lesssim \frac{\nu^{2k+1}}{\la} .
 }
Similarly, for~\eqref{eq:ff'la} we have 
\EQ{
&\abs{ \ang{ \La Q_{\U \la} \mid  \frac{1}{r^2} \Big( f(Q_\la - Q_\mu)  - f(Q_\la) + f(Q_\mu)  - 4  \big(\frac{r}{ \mu}\big)^k (\La Q_{\la})^2  -  4 \big(\frac{r}{ \la}\big)^{-k} (\La Q_{\mu})^2 \Big) }} \\
& \lesssim   \frac{1}{\la}  \int_0^\infty [\La Q_\nu]^3 ( \sin 2 Q - 4 r^k) \, \frac{\ud r }{r}  + \frac{1}{\la} \int_0^\infty \La Q_\nu ( \sin 2 Q_\nu + 4(r/\nu)^{-k}) [\La Q]^2 \, \frac{\ud r}{r}  \\
& \lesssim \frac{\nu^{2k}}{\la} 
}
Next, note that 
\EQ{ \label{eq:f'T} 
\abs{f(\Phi) - f(Q_\la - Q_\mu) -  f'(Q_\la- Q_\mu) (b^2 T_\la - a^2 \ti T_{\mu})} &\lesssim (b^2 A_\la + \nu^k B_\la + a^2 A_\mu + \nu^k \ti B_\mu)^2 \\
& \lesssim b^4 A_\la^2 + \nu^{2k} B_\la^2 + a^4 A_\mu^2 + \nu^{2k}B_\mu^2 
}
It follows from Lemma~\ref{l:ABB} that 
\EQ{
\Big\| r^{-2} \Big(f(\Phi) - f(Q_\la - Q_\mu) &-  f'(Q_\la- Q_\mu) (b^2 T_\la - a^2 \ti T_{\mu} \Big) \Big\|_{L^2}   \lesssim \frac{ b^4}{\la} + \frac{a^4}{\mu} + \frac{\nu^{2k}}{\la} 
}
which yields the second estimate in~\eqref{eq:ff'L2}. The second estimates in~\eqref{eq:ff'mu} and~\eqref{eq:ff'la} follow from~\eqref{eq:f'T} and Lemma~\ref{l:ABBQ}. 

Next, 
\EQ{
&\abs{f'(Q_\la - Q_\mu) (b^2 T_\la - a^2 \ti T_\mu) - f'(Q_\la)b^2 T_\la + f'(Q_\mu) a^2 \ti T_\mu } \\
&  = \Big|- \cos 2Q_\la \La Q_\mu^2(b^2 A_\la + \nu^k B_\la)  + \sin 2Q_\la \sin 2Q_\mu (b^2 A_\la + \nu^k B_\la)  \\
& \quad + \cos 2Q_\mu \La Q_\la(a^2 A_\mu+ \nu^k \ti B_\mu)  - \sin 2Q_\la \sin 2Q_\mu(a^2 A_\mu + \nu^k \ti B_\mu)\Big| \\
& \lesssim b^2 \La Q_\mu^2A_\la    + \nu^k \La Q_\mu^2 B_\la  +  b^2 \La Q_\la \La Q_\mu A_\la + \nu^k \La Q_\la \La Q_\mu B_\la  \\
& \quad +  a^2 \La Q_\la^2 A_\mu+ \nu^k  \La Q_\la^2 \ti B_\mu  +  a^2 \La Q_\la \La Q_\mu A_\mu + \nu^k \La Q_\la \La Q_\mu \ti B_\mu
}
Using~\eqref{eq:ABBQL2} from Lemma~\ref{l:ABBQL2} it follows that 
\EQ{
\| r^{-2} (f'(Q_\la - Q_\mu) &(b^2 T_\la - a^2 \ti T_\mu) - f'(Q_\la)b^2 T_\la + f'(Q_\mu) a^2 \ti T_\mu) \|_{L^2}  \\
& \lesssim \frac{b^2 \nu^{k-1}}{\la}  + \frac{a^2 \nu^{k-o(1)}}{\la}+ \frac{\nu^{2k-1}}{\la}
}
Similarly, using~\eqref{eq:ABBQpair} from Lemma~\ref{l:ABBQnl} we deduce that 
\EQ{
&\abs{\ang{ \La Q_{\U\mu} \mid \frac{1}{r^2} \Big(f'(Q_\la - Q_\mu) (b^2 T_\la - a^2 \ti T_\mu) - f'(Q_\la)b^2 T_\la + f'(Q_\mu) a^2 \ti T_\mu \Big)}} \\
& \qquad \quad \quad  \lesssim  \frac{b^2 \nu^{k-1}}{\la} + \frac{\nu^{2k-1}}{\la} + \frac{a^2 \nu^{k+1}}{\la}
}
and 
\EQ{
&\abs{\ang{ \La Q_{\U\la} \mid \frac{1}{r^2} \Big(f'(Q_\la - Q_\mu) (b^2 T_\la - a^2 \ti T_\mu) - f'(Q_\la)b^2 T_\la + f'(Q_\mu) a^2 \ti T_\mu \Big)}} \\ 
& \qquad \quad \quad  \lesssim  \frac{b^2 \nu^{k}}{\la} +  \frac{\nu^{2k-o(1)}}{\la}  + \frac{a^2 \nu^k}{\la} 
}
\end{proof}

\begin{proof}[Proof of Lemma~\ref{l:fest2}]

To prove the first estimate in~\eqref{eq:ff'L2dr} we differentiate the expression~\eqref{eq:fQQ}, which yields, 
\EQ{
 \p_r \Big(& f(Q_\la - Q_\mu)  - f(Q_\la) + f(Q_\mu)  - 4  \big(\frac{r}{ \mu}\big)^k (\La Q_{\la})^2  -  4 \big(\frac{r}{ \la}\big)^{-k} (\La Q_{\mu})^2 \Big)  \\
  &= \frac{2}{r}  \Big( \cos 2 Q_\mu \La Q_\mu - 2 k\big( \frac{r}{\mu} \big)^k \Big) ( \La Q_\la)^2 + \frac{2}{r}  \Big( \sin 2 Q_\mu - 4 \big( \frac{r}{\mu} \big)^k \Big) \La^2 Q_\la \La Q_\la  \\
&\quad   - \frac{2}{r} \Big(  \cos 2 Q_\la \La Q_\la - 2  \big( \frac{r}{\la} \big)^{-k}  \Big) (\La Q_\mu)^2 - \frac{2}{r}  \Big( \sin 2Q_\la + 4\big( \frac{r}{\la} \big)^{-k}  \Big) \La^2 Q_\mu \La Q_\mu \\
  &= \frac{2}{r}  \Big(  \La Q_\mu - 2 k\big( \frac{r}{\mu} \big)^k \Big) ( \La Q_\la)^2 + \frac{2}{r}  \Big( \sin 2 Q_\mu - 4 \big( \frac{r}{\mu} \big)^k \Big) \La^2 Q_\la \La Q_\la  \\
 &\quad  - \frac{2}{r} \Big( \La Q_\la - 2  \big( \frac{r}{\la} \big)^{-k}  \Big) (\La Q_\mu)^2 -\frac{2}{r} \Big( \sin 2Q_\la + 4\big( \frac{r}{\la} \big)^{-k}  \Big) \La^2 Q_\mu \La Q_\mu \\
  &\quad   -  \frac{4}{rk^2}( \La Q_\mu)^3 ( \La Q_\la)^2  +  \frac{4}{rk^2}( \La Q_\la)^3 ( \La Q_\mu)^2 
} 
Arguing as in the proof of~\eqref{eq:ff'L2} we deduce the estimate
\EQ{
\Big\| r^{-2}  \p_r \Big(& f(Q_\la - Q_\mu)  - f(Q_\la) + f(Q_\mu)  - 4  \big(\frac{r}{ \mu}\big)^k (\La Q_{\la})^2  -  4 \big(\frac{r}{ \la}\big)^{-k} (\La Q_{\mu})^2 \Big)\Big\|_{L^2}  \\
& \lesssim \Big\| r^{-3}  \Big(  \La Q_\mu - 2 k\big( \frac{r}{\mu} \big)^k \Big) ( \La Q_\la)^2 \Big\|_{L^2} + \Big\| r^{-3}  \Big( \sin 2 Q_\mu - 4 \big( \frac{r}{\mu} \big)^k \Big) \La^2 Q_\la \La Q_\la   \Big\|_{L^2}  \\
&\quad +  \Big\| r^{-3}  \Big( \La Q_\la - 2  \big( \frac{r}{\la} \big)^{-k}  \Big) (\La Q_\mu)^2 \Big\|_{L^2} + \Big\| r^{-3}  \Big( \sin 2Q_\la + 4\big( \frac{r}{\la} \big)^{-k}  \Big) \La^2 Q_\mu \La Q_\mu  \Big\|_{L^2}  \\
&\quad +  \| r^{-3} ( \La Q_\mu)^3 ( \La Q_\la)^2  \|_{L^2} + \| r^{-3} ( \La  Q_\la)^3 ( \La Q_\mu)^2\|_{L^2}  \\
& \lesssim \frac{\nu^{2k}}{\la^2} 
}
as claimed. The proofs of the remaining two estimates in~\eqref{eq:ff'L2dr} are similarly obtained via explicit calculations as in the proofs of the second to estimates in~\eqref{eq:ff'L2}. 
\end{proof} 

\begin{proof}[Proof of Lemma~\ref{l:w^2est}]
First observe the identity 
\EQ{ \label{eq:fident} 
f( \Phi +w) - f( \Phi) - f'( \Phi) w &= \frac{k^2}{2} \Big( \sin( 2\Phi  +2 w) - \sin(2 \Phi) - (\cos 2\Phi) 2 w \Big) \\
& = k^2 \sin 2 \Phi \sin^2 w + \frac{k^2}{2}( \sin 2 w - 2 w) \cos 2 \Phi   \\
& = k^2 \sin 2 \Phi \sin^2 w + O( \abs{w}^3) 
}
Since 
\EQ{ \label{eq:sin2Phi} 
\abs{ \sin 2 \Phi } =\abs{ \sin 2Q_\la \cos 2( \Phi - Q_\la) + \cos 2Q_\la \sin 2( \Phi - Q_\la) } \lesssim \frac{r}{\la} 
}
we thus we have the point-wise bound 
\EQ{ \label{eq:w^2-point} 
\frac{1}{r^2} \abs{f( \Phi +w) - f( \Phi) - f'( \Phi) w} & \lesssim \frac{1}{\la} \frac{ \abs{w}^2}{r} + \frac{\abs{w}^3}{r^2}  
}
which leads to  the estimate 
\EQ{
\| r^{-2} (f( \Phi +w) - f( \Phi) - f'( \Phi) w) \|_{L^2} &\lesssim \frac{\| w\|_{\infty} \| w \|_H}{\la} +  \| r^{-2} w^3 \|_{L^2}  \\
& \lesssim  \frac{1}{\la} \| w \|_{H}^2 + \| w \|_{H}^2 \| \p_r w \|_{H} 
}
where the last line follows from~\eqref{eq:winfty} and~\eqref{eq:w3} from Lemma~\ref{l:wHH2}. This proves~\eqref{eq:w^2L2}. Similarly, to prove~\eqref{eq:w^2L23} we have 
\EQ{
\| r^{-3} (f( \Phi +w) - f( \Phi) - f'( \Phi) w) \|_{L^2} &\lesssim \frac{\| r^{-1} w \|_{\infty} \| r^{-1} w \|_{L^2}}{\la} +  \| r^{-1} w \|_{L^\infty}^2 \| r^{-1} w \|_{L^2}  \\
& \lesssim  \frac{1}{\la} \| \p_r w \|_{H} \| w \|_{H}  + \| \p_r w \|_{H}^2 \| w \|_H
}
where in the last line we used~\eqref{eq:w/rinfty}. 

Next we prove~\eqref{eq:w^2L22}. First note that 
\EQ{
\p_r ( f( \Phi +w) - f( \Phi) - f'( \Phi) w)  
& = 2 k^2 \cos 2 \Phi \p_r \Phi \sin^2 w  - k^2  \sin 2 \Phi \p_r \Phi ( \sin 2 w - 2 w) \\
& \quad + 2 k^2 \sin 2\Phi \cos w \sin w \p_r w - 2k^2 \cos 2\Phi \sin^2 w \p_r w
}
which can also be seen directly from the identity~\eqref{eq:fident}. Using the pointwise bound, 
\EQ{
\abs{\p_r \Phi }  &=  \frac{1}{r} \abs{  \La Q_\la - \La Q_\mu + b^2 \La A_\la + \nu^k \La B_\la - a^2 \La A_\mu   - \nu^k \La \ti B_\mu}  \lesssim \frac{1}{\la} 
}
along with~\eqref{eq:sin2Phi} 
we arrive at the estimate 
\EQ{
\abs{\frac{1}{r^2} \p_r ( f( \Phi +w) - f( \Phi) - f'( \Phi) w)} \lesssim \frac{1}{\la}  \frac{\abs{w}^2}{r^2} + \frac{\abs{w}^3}{r^3} + \frac{1}{\la} \frac{\abs{ \p_r w} \abs{w}}{r}  + \frac{\abs{\p_r w} \abs{w}^2}{r^2}  
}
from which it follows that 
\EQ{
\| \frac{1}{r^2} \p_r ( f( \Phi +w) - f( \Phi) - f'( \Phi) w)\|_{L^2} &\lesssim  \frac{1}{\la} \| \p_r w \|_{H} \| w \|_{H}  +  \| \p_r w \|_{H}^2 \| w \|_H
}
which proves~\eqref{eq:w^2L22}. 

Next, we prove~\eqref{eq:w^2la}.  We have 
\EQ{ \label{eq:fPhiw1} 
f( \Phi + w) - f( \Phi) - f'( Q_\la) w  &= k^2  ( \cos 2 \Phi - \cos 2 Q_\la) w  \\
& \quad - k^2 \sin 2 \Phi \sin^2 w +  O( \abs{w}^3)
}
For the first term on the right we use the estimate~\eqref{eq:cos2Phi1} to obtain 
\EQ{
\abs{k^2  ( \cos 2 \Phi - \cos 2 Q_\la) w} &\lesssim  \abs{w} \La Q_{\la} (\La Q_\mu   + b^2 A_\la + \nu^k B_\la  + a^2 A_\mu + \nu^k\ti B_\mu)   \\
& \quad + \abs{w} \La Q_\mu^2 +\abs{w}  b^4 A_\la^2  + \abs{w} \nu^{2k} B_\la^2 +\abs{w}  a^4 A_\mu^2   + \abs{w} \nu^{2k} \ti B_\mu^2
} 
It follows that 
\EQ{
\abs{\ang{ \La Q_{\U \la} \mid \frac{k^2}{r^2}  ( \cos 2 \Phi - \cos 2 Q_\la) w} } &\lesssim  \frac{1}{\la} \| w \|_H  \Big( \| r^{-1} [\La Q_\la]^2 \La Q_\mu  \|_{L^2} + b^2  \| r^{-1} [\La Q]^2 A \|_{L^2}  \\
&\quad + \nu^k \| r^{-1} [\La Q]^2 B \|_{L^2}    + a^2  \| r^{-1} [\La Q_\la]^2 A_\mu \|_{L^2}  \\
&\quad + \nu^k \|r^{-1} [\La Q_\la]^2 \ti \B_\mu \|_{L^2}   + \| r^{-1} \La Q_\la \La Q_\mu^2 \|_{L^2}  \\
& \quad + b^4 \| r^{-1} \La Q A^2 \|_{L^2} + \nu^{2k} \|r^{-1} \La Q B^2 \|_{L^2} \\
& \quad +  a^4 \|r^{-1} \La Q_\la [A_\mu]^2 \|_{L^2} +\nu^{2k} \| r^{-1} \La Q_\la [\ti B_\mu]^2 \|_{L^2} \Big)  \\
& \lesssim \frac{1}{\la} \| w \|_H \Big( \nu^k + b^2 + a^2 \nu^k\Big)
}
where we have used Lemma~\ref{l:ABBQL21} in the second inequality above. The second term on the right-hand side of~\eqref{eq:fPhiw1} leads to the estimate 
\EQ{
\abs{ \ang{ \La Q_{\U \la}  \mid \frac{1}{r^2} \sin 2\Phi  \sin^2 w}} &\lesssim \frac{1}{\la} \int_0^\infty \La Q_{\la} \abs{\sin 2 \Phi }w^2 \, \frac{\ud r}{r}   \lesssim \frac{1}{\la} \|  w \|_{H}^2 
}
Finally, the last term satisfies 
\EQ{
\abs{\ang{ \La Q_{\U \la} \mid \frac{1}{r^2} \abs{w}^3 }} &\lesssim \frac{1}{\la} \|w \|_{\infty} \| w \|_{H}^2 \lesssim  \frac{1}{\la} \| w \|_{H}^3
}
Combining the previous three estimates gives~\eqref{eq:w^2la}. 
Similarly to prove~\eqref{eq:w^2la} we have 
\EQ{ \label{eq:fPhiw2} 
f( \Phi + w) - f( \Phi) - f'( Q_\la) w  &= k^2  ( \cos 2 \Phi - \cos 2 Q_\mu) w - k^2 \sin 2 \Phi \sin^2 w    +  O( \abs{w}^3)
}
For the first term on the right we use~\eqref{eq:cos2Phi2}, 
\EQ{
\abs{k^2  ( \cos 2 \Phi - \cos 2 Q_\mu) w} & \lesssim \abs{w} \abs{\La Q_{\mu} (\La Q_\la   + b^2 A_\la + \nu^k B_\la  + a^2 A_\mu + \nu^k\ti B_\mu) }   \\
& \quad + \abs{w} \La Q_\la^2 + \abs{w} b^4 A_\la^2  + \abs{w} \nu^{2k} B_\la^2 +\abs{w}  a^4 A_\mu^2   + \abs{w} \nu^{2k} \ti B_\mu^2 
}
from which we have 
\EQ{
\Big|\langle &\La Q_{\U \la} \mid \frac{k^2}{r^2}  ( \cos 2 \Phi - \cos 2 Q_\la) w \rangle \Big| \lesssim  \frac{1}{\mu} \| w \|_H \Big(  \| r^{-1} [\La Q_\mu]^2 \La Q_\la \|_{L^2}   + b^2 \| r^{-1} [\La Q_\mu]^2 A_\la \|_{L^2}  \\
&   + \nu^k \| r^{-1} [\La Q_\mu]^2B_\la \|_{L^2}   + a^2 \| r^{-1} [\La Q]^2 A \|_{L^2}   +  \nu^k \| r^{-1} [\La Q]^2  \ti B\|_{L^2}  +  \| r^{-1} \La Q_\mu [\La Q_\la]^2 \|_{L^2}  \\
& + b^4 \| r^{-1} \La Q_\mu [A_{\la}]^2 \|_{L^2} + \nu^{2k} \| r^{-1} \La Q_\mu[ B_\la]^2 \|_{L^2}  + a^4 \| r^{-1} \La Q A^2 \|_{L^2} + \nu^{2k}  \| \La Q \ti B^2 \|_{L^2}  \Big)  \\
& \lesssim \frac{1}{\mu} \| w \|_H \Big( \nu^k + b^2 \nu^{k-2}   + a^2 \Big) 
}
The two remaining terms on the right-hand side of~\eqref{eq:fPhiw2} satisfy, 
\EQ{
 \abs{ \ang{ \La Q_{\U \mu}  \mid \frac{1}{r^2} \sin 2\Phi  \sin^2 w}} & \lesssim \frac{1}{\mu} \|  w \|_{H}^2  , \quad 
 \abs{\ang{ \La Q_{\U \mu} \mid \frac{1}{r^2} \abs{w}^3 }}   \lesssim \frac{1}{\mu} \|  w \|_{H}^3
}
This completes the proof. 
\end{proof}

\subsection{Coercivity} 
The goal of this section is to record the following two lemmas, for which we sketch the standard proofs. 

\begin{lem}[Coercivity of $\LL_\Phi$]\label{l:coerce1} There exists $c_1>0$ with the following property.  Let $\bs u(t) \in \HH$ be a solution to~\eqref{eq:wmk} and assume that the hypothesis of Lemma~\ref{l:mod2}. 
 As long as $\eta_0>0$ is taken small enough in Lemma~\ref{l:mod2}, we have  
\EQ{ \label{eq:coerce1} 
\ang{ \LL_{\Phi} w \mid w }& \ge c_1 \| w \|_{H}^2 \\
}
\end{lem} 

\begin{lem}[Coercivity of $\LL^2_\Phi$] \label{l:coerce2}
In addition to the hypothesis of Lemma~\ref{l:mod2} assume in addition that $\bs u(t) \in \HH \cap \HH^2 \cap \bs \Lam^{-1} \HH$ for all $t \in J$. Then, for $\bs w(t)$ given by Lemma~\ref{l:mod2} we additionally have $\bs w(t) \in \HH \cap \HH^2 \cap \bs \Lam^{-1} \HH$. And, as long as $\eta_0>0$ is taken small enough, the uniform constant $c_1>0$ from~\eqref{eq:coerce1} can be chosen so that we also have the $\HH^2$ coercivity, 
\EQ{ \label{eq:coerce2} 
\ang{ \LL_{\Phi}^2 w \mid w} + \ang{ \LL_\Phi \dot w  \mid \dot w} &\ge c_1  \| \bs w \|_{\HH^2}^2.
}
\end{lem} 

We now give brief sketches of the proofs of Lemma~\ref{l:coerce1} and Lemma~\ref{l:coerce2}. As usual, we denote by $\LL$ the operator as in~\eqref{eq:LL-def} and recall that we write 
\EQ{
 \LL= \LL_0 + P 
}
where $\LL_0 = -\De + \frac{k^2}{r^2}$ and $P$ is as in~\eqref{eq:P-def}, i.e., $P(r):= \frac{1}{r^2}( f'(Q) - k^2)$. 

Lemma~\ref{l:coerce1} is a straightforward consequence of the following localized coercivity lemma proved in~\cite{JJ-AJM}. 
\begin{lem}[Localized coercivity for $\LL$] \emph{\cite[Lemma 5.4]{JJ-AJM}}  \label{l:loc-coerce} 
There exists a uniform constant $c_1>0$  with the following property. Suppose $w \in H$ is such that 
\EQ{ \label{eq:w-ortho1} 
\ang{ w \mid  \Lam Q} = 0. 
}
Then, 
\EQ{ \label{eq:L-coerce}
\ang{ \LL w \mid w} \ge c_1  \| w \|_{H}^2 .
}
In addition, for any $c>0$, there exists $R_1>0$ large enough so that for all $w \in H$ as in~\eqref{eq:w-ortho1}, we have 
\EQ{ \label{eq:L-loc1} 
\int_0^{R_1} \Big((\p_r w(r))^2 + k^2 \frac{w(r)^2}{r^2} \Big) \, r \ud r + \ang{ P w \mid w} \ge - c \| w \|_{H}^2 
}
Lastly, for any $c>0$, there exists $r_1>0$ small enough so that for all $w \in H$ as in~\eqref{eq:w-ortho1}, we have 
\EQ{ \label{eq:L-loc2} 
\int_{r_1}^\infty \Big((\p_r w(r))^2 + k^2 \frac{w(r)^2}{r^2}\Big)  \, r \ud r +  \ang{ P w \mid w} \ge - c \| w \|_{H}^2 
}
\end{lem} 

Lemma~\ref{l:coerce2} is a straightforward consequence of similar localized coercivity lemma for the operator $\LL^2$, for which we give a sketched proof below. Note that the statement that $\bs w(t) = \bs u(t) - \bs \Phi( \mu(t), \lam(t), a(t), b(t)) \in \HH \cap \HH^2 \cap \bs \Lam^{-1} \HH$ follows from the fact that both $\bs u(t)$ and $\bs \Phi( \mu, \lam, a, b)$ are in this space. 

 First, recall that $\LL^2$ is given by the following formula~\eqref{eq:LL2-def} and that we often will write, 
 \EQ{
 \LL^2 w = \LL_0^2 w + \Ks w 
 }
 where $\Ks$ is defined in~\eqref{eq:K-def}. 
\begin{lem}[Localized coercivity for $\LL^2$] There exists a uniform constant $c_1>0$  with the following property. Suppose $w \in H^2$ is such that 
\EQ{ \label{eq:w-ortho} 
\ang{ w \mid  \Lam Q} = 0. 
}
Then, 
\EQ{ \label{eq:L2-coerce}
\ang{ \LL^2 w \mid w} \ge c_1  \| w \|_{H^2}^2 
}
In addition, for any $c>0$, there exists $R_1>0$ large enough so that for all $w \in H^2$ as in~\eqref{eq:w-ortho}, we have 
\EQ{ \label{eq:L2-loc1} 
\int_0^{R_1} ( \LL_0 w)^2 \, r \ud r + \ang{ \Ks w \mid w} \ge - c \| w \|_{H^2}^2 
}
Lastly, for any $c>0$, there exists $r_1>0$ small enough so that for all $w \in H^2$ as in~\eqref{eq:w-ortho}, we have 
\EQ{\label{eq:L2-loc2} 
\int_{r_1}^\infty ( \LL_0 w)^2 \, r \ud r +  \ang{ \Ks w \mid w} \ge - c \| w \|_{H^2}^2 
}
\end{lem} 

\begin{proof} 
We first prove~\eqref{eq:L2-coerce}. The argument is standard and similar to the one used to prove the coercivity estimates for first order operators given in~\cite[Appendix B]{RS}.  Observe the identity, 
\EQ{
\int_0^\infty ( \LL_0 w)^2 \,  \rdr = \int_0^\infty (\p_r^2 w)^2 \, \rdr + (2k^2 +1) \int_0^\infty \frac{( \p_r w)^2}{r^2}  \, \rdr + (k^4 - 4 k^2) \int_0^\infty  \frac{ w^2}{r^4} \, \rdr  
}
Next, we compute 
\EQ{
 \ang{ \Ks w \mid w}  = \int_0^\infty 2 P \LL_0 w  \, w \, r \, \ud r - 2\int_0^\infty \p_r P \p_r w w \, r \, \ud r  + \int_0^\infty (P^2 -\De P) w^2 \,  \rdr 
}
Note that 
\EQ{
\int_0^\infty 2 P \LL_0 w  \, w \, r \, \ud r &= -2 \int_0^\infty  P \De w  \, w \, r \, \ud r + 2k^2 \int_0^\infty  P w^2 \, r^{-2} \, r \, \ud r \\
& = 2 \int_0^\infty \p_r P \p_r w w \,  r \, \ud r + 2 \int_0^\infty P (\p_r w)^2 \,  \rdr  + 2k^2 \int_0^\infty  P w^2 \, r^{-2} \, r \, \ud r 
}
and hence we obtain, 
\EQ{
 \ang{ \Ks w \mid w} & = + 2 \int_0^\infty P (\p_r w)^2 \,  \rdr   +  \int_0^\infty (P^2 -\De P  + \frac{2 k^2}{r^2}P) w^2 \,  \rdr 
}
and thus the identity, 
\EQ{ \label{eq:L2ww1} 
\ang{ \LL^2 w \mid w}  &=  \int_0^\infty (\p_r^2 w)^2 \, \rdr  + (2k^2 +1) \int_0^\infty \frac{( \p_r w)^2}{r^2}  \, \rdr +   \int_0^\infty 2 r^2 P(r)\frac{ (\p_r w)^2}{r^2}  \,  \rdr  \\
& \quad + (k^4 - 4 k^2) \int_0^\infty  \frac{ w^2}{r^4} \, \rdr   +  \int_0^\infty r^4(P^2 -\De P  + \frac{2 k^2}{r^2}P) \frac{w^2 }{r^4} \,  \rdr 
}
Define the function $h_1(r)$ by
\EQ{
h_1(r) := 2 r^2 P(r) = 2 (f'(Q) -k^2) = -2k^2 \sin^2 Q  \le 0
}
Note the identity 
\EQ{
r^4 (P^2 - \De P + \frac{2 k^2}{r^2} P) = \big(4k^4 +16 k^3(1- \cos Q) +  8k^2\big) \sin^2 Q   + h_2(r)
}
where the function $h_2(r)$ is given by 
\EQ{
h_2(r) =   - 8k^4 \sin^4 Q - 16k^3 \sin^2 Q  \le 0 
}
The functions $h_1, h_2$ satisfy the bounds, 
\EQ{ \label{eq:h_j-bound} 
\abs{h_1(r)}, \abs{ h_2(r)} \lesssim  \frac{ r^{2k}}{(1 + r^{2k})^2}
}
and we rewrite~\eqref{eq:L2ww1} as, 
\EQ{ \label{eq:LL2-ident} 
\ang{ \LL^2 w \mid w}  &=  \int_0^\infty (\p_r^2 w(r))^2 \, \rdr  + (2k^2 +1) \int_0^\infty \frac{( \p_r w(r))^2}{r^2}  \, \rdr   \\
&\, +  \int_0^\infty \Big(k^4 - 4 k^2 + \big(4k^4 +16 k^3(1- \cos Q) +  8k^2\big) \sin^2 Q \Big)\frac{ w(r)^2}{r^4} \, \rdr  \\\
& \,  +  \int_0^\infty  h_1(r) \frac{( \p_r w(r))^2}{r^2}  \, \rdr  +  \int_0^\infty h_2(r) \frac{ w(r)^2}{r^4} \, \rdr
}
and note that, 
\EQ{ \label{eq:k2-lower} 
k^4 - 4 k^2 + \big(4k^4 +16 k^3(1- \cos Q) +  8k^2\big) \sin^2 Q \ge 1 \mif k \ge 2
}
Now, suppose that~\eqref{eq:L2-coerce} fails. Then we can find a sequence $w_n \in H^2$ so that, 
\EQ{
\| w_n \|_{H^2} = 1, \quad 
\ang{w_n \mid \Lam Q} = 0 \\
\ang{ \LL^2 w_n \mid w_n}   \le  \eps_n   \to 0  \mas n \to \infty 
}
Passing to a subsequence we obtain a weak limit $w_n \rightharpoonup w_\infty \in H^2$ such that 
\EQ{ \label{eq:infty-ortho} 
\ang{ w_{\infty}  \mid \Lam Q} = 0
}
and so that $w_\infty$ is a weak solution to the equation $\LL w_\infty =0$. By elliptic regularity $w_\infty$ is in fact a strong solution. Since the kernel of $\LL$ is given by the span of $\La Q$, we deduce from~\eqref{eq:infty-ortho} that $w_\infty =0$. We will obtain a contradiction by showing that $w_\infty \neq 0$. From~\eqref{eq:LL2-ident} and the lower bound~\eqref{eq:k2-lower} we can find a uniform positive constant $c_2>0$ so that 
\EQ{
c_2 \| w_n \|_{H^2}^2 +   \int_0^\infty  h_1(r) \frac{( \p_r w_n(r))^2}{r^2}  \, \rdr  +  \int_0^\infty h_2(r) \frac{ w_n(r)^2}{r^4} \, \rdr \le \eps_n \| w_n \|_{H^2}^2 
}
From the normalization $\|w_n \|_{H^2} = 1$, we conclude that 
\EQ{
\limsup_{n \to \infty} \Big( \int_0^\infty  h_1(r) \frac{( \p_r w_n(r))^2}{r^2}  \, \rdr  +  \int_0^\infty h_2(r) \frac{ w_n(r)^2}{r^4} \, \rdr \Big)  \le - \frac{1}{2} c_2 < 0
}
Since we have $h_j(r) \le 0$ we can conclude from the above that $w_\infty \neq 0$ if we can show strong convergence of the above integrals. To see this define the norm, 
\EQ{
\| w \|_{Y}^2:=  \int_0^\infty  \frac{ r^{2k}}{(1 + r^{2k})^2} \frac{( \p_r w_n(r))^2}{r^2}  \, \rdr  +  \int_0^\infty \frac{ r^{2k}}{(1 + r^{2k})^2} \frac{ w_n(r)^2}{r^4} \, \rdr
}  
Because of the favorable weight at $r =0$ and $r =\infty$ we observe the compact embedding $H^2  \Subset Y$. Hence, from the estimate~\eqref{eq:h_j-bound} we conclude that 
\EQ{
 \int_0^\infty  h_1(r) \frac{( \p_r w_\infty(r))^2}{r^2}  \, \rdr  +  \int_0^\infty h_2(r) \frac{ w_\infty(r)^2}{r^4} \, \rdr < 0
}
and thus $w_\infty \neq 0$ a contradiction. This proves~\eqref{eq:L2-coerce}. Finally, we note that the localized estimates~\eqref{eq:L2-loc1},~\eqref{eq:L2-loc2} follow from the same argument used to prove the analogous estimates in Lemma~\ref{l:coerce1}, and for these we refer the reader to the detailed arguments given in~\cite[Proof of Lemma 2.1]{JJ-Pisa} or~\cite[Proof of Lemma 2.1]{MaMe16-arma}.  
\end{proof}

\subsection{Dynamical control of the modulation parameters} 
In this section we obtain preliminary estimates on the dynamics of the modulation parameters by differentiating the orthogonality conditions~\eqref{eq:law},~\eqref{eq:muw}, \eqref{eq:lawdot}, and \eqref{eq:muwdot}.


\begin{lem}[Dynamics of modulation parameters] \label{l:modc2} There exists $\eta_1>0$ with the following property. Let $ J \subset \R$ be a time interval and let $\bs u(t)$ be a solution to~\eqref{eq:wmk} on $J$. Suppose that 
\EQ{
\bfd_+( \bs u(t))  \le \eta_1  \quad \forall \, t \in J
}
Let $(\mu(t), \lam(t), a(t), b(t))$ and $\bs w(t)$ be as in Lemma~\ref{l:mod2} and let $\nu := \lam/ \mu$ as usual. Assume (for convenience, as this will later be bootstrapped) that the bound, 
$
\mu \simeq 1
$
is satisfied on $J$. 
Then, 
\begin{align}
\abs{ a - \mu'} &\lesssim (\abs{a} + \abs{b} ) \| \bs w \|_{\HH} +  (\abs{a }+ \abs{b})\nu^{2k} + ( \abs{a}^5 + \abs{b}^5) \nu^{k-2}  \label{eq:a-mu'} \\ 
\abs{ b+ \lam'}& \lesssim (\abs{a} + \abs{b} ) \| \bs w \|_{\HH} +( \abs{a} + \abs{b})\nu^{2k-1}  + ( \abs{a}^5 + \abs{b}^5) \nu^{k-2}\label{eq:b+la'}  \\ 
\abs{a' +  \ti \gamma_k \frac{\nu^k}{\mu}} & \lesssim  \frac{\abs{a}\nu  + \abs{b}\nu}{\la}   \| \dot w \|_{L^2}  +  \frac{\nu}{\la} \| \bs w \|_{\HH}^2 + \frac{\nu^{2k-1}}{\lam} + \frac{b^2 \nu^{k-1}}{\lam} +  \frac{b^4\nu}{\lam} + \frac{a^4\nu}{\lam}    \label{eq:a'est} \\ 
\abs{b' +  \gamma_k \frac{\nu^k}{\lam}} & \lesssim  \frac{\abs{a}  + \abs{b}}{\la}   \| \dot w \|_{L^2}  +  \frac{1}{\la} \| \bs w \|_{\HH}^2 + \frac{\nu^{2k-o(1)}}{\lam}  +  \frac{b^4}{\lam} + \frac{a^4}{\lam} \label{eq:b'est}
\end{align} 
\end{lem} 


\begin{proof}[Proof of Lemma~\ref{l:modc2}] 
We differentiate the orthogonality conditions~\eqref{eq:law}--\eqref{eq:muwdot} with the goal of setting up a linear system for $((a- \mu'), (b+ \la'),( a' +   \ga_k \frac{\nu^k}{\mu}),  ( b' + \gamma_k \frac{\nu^{k}}{\la}))$. We begin with~\eqref{eq:law} using~\eqref{eq:weq} in the third equality below. 
\EQ{
 0 &= \frac{\ud}{\ud t} \ang{ \La Q_{\U \la}  \mid w}  = -\frac{\la'}{\la} \ang{ \La_0 \La Q_{\U \la} \mid  w} + \ang{ \La Q_{\U \la}  \mid \p_t w} \\
 & = -\frac{\la'}{\la} \ang{ \La_0 \La Q_{\U \la} \mid  w} + \ang{ \La Q_{\U \la}  \mid  \dot w}   + \ang{  \La Q_{\U \la} \mid \dot \Phi - \p_t \Phi  }  \\ 
 & = -\frac{\la'}{\la} \ang{ \La_0 \La Q_{\U \la} \mid  w }  + \ang{  \La Q_{\U \la} \mid \dot \Phi - \p_t \Phi  }  \\
 & = I_a + I_b 
 }
 where in the second to last line we have used the orthogonality~\eqref{eq:lawdot}.  First we write $I_a$ as 
 \EQ{
 I_a = b\ang{ (r\La_0 \La Q)_{\U \la} \mid  w/ r } -   (\la' + b) \ang{ (r\La_0 \La Q)_{\U \la} \mid  w/ r } 
 }
  We expand out term $I_b$ above using~\eqref{eq:dotPhi-ptPhi} along with~\eqref{eq:AQ}~\eqref{eq:BQ}. 
 \EQ{
 I_b &= (b+ \la') \Big(  \| \La Q \|_{L^2}^2  + b^2 \ang{ \La Q\mid  \La A} + \nu^k \ang{ \La Q \mid \La B}  - k \nu^{k-1} \ang{ \La Q_{\U \lam} \mid \ti B_{\U \mu}} \Big)  \\
 & \quad +  (a- \mu') \Big( \ang{ \La Q_{\U \lam}  \mid  \La Q_{\U \mu} }   + a^2 \ang{ \La Q_{\U \lam}  \mid \La A_{\U \mu}}  + \nu^k  \ang{ \La Q_{\U \lam}  \mid   \La \ti B_{\U \mu}}  + k \nu^k \ang{ \Lam Q_{\U \lam} \mid \ti B_{\U \mu}} \Big)  \\
 &\quad + ( a' +   \ga_k \frac{\nu^k}{\mu}) 2 a \mu  \ang{ \La Q_{\U \lam}  \mid A_{\U\mu} }  
  }
Combining the above we arrive at the identity, 
\EQ{
(b+ \la') \Big(  \| \La &Q \|_{L^2}^2  + b^2 \ang{ \La Q\mid  \La A} + \nu^k \ang{ \La Q \mid \La B}  - k \nu^{k-1} \ang{ \La Q_{\U \lam} \mid \ti B_{\U \mu}} -  \ang{ (r\La_0 \La Q)_{\U \la} \mid  \frac{w}{r} } \Big)  \\
 +  (a- \mu')& \Big( \ang{ \La Q_{\U \lam}  \mid  \La Q_{\U \mu} }   + a^2 \ang{ \La Q_{\U \lam}  \mid \La A_{\U \mu}}  + \nu^k  \ang{ \La Q_{\U \lam}  \mid   \La \ti B_{\U \mu}}  + k \nu^k \ang{ \Lam Q_{\U \lam} \mid \ti B_{\U \mu}} \Big)  \\ 
 +  ( a' +   \ga_k& \frac{\nu^k}{\mu}) \Big(2 a \mu  \ang{ \La Q_{\U \lam}  \mid A_{\U\mu} }  \Big)\\ 
 & =  - b\ang{ (r\La_0 \La Q)_{\U \la} \mid  w/ r } 
}
Similarly, we differentiate~\eqref{eq:muw} again using~\eqref{eq:weq} and~\eqref{eq:muwdot}. 
\EQ{
0 &=\frac{\ud}{\ud t}  \ang{\La Q_{\U{\mu}} \mid w } =  -\frac{\mu'}{\mu} \ang{  \La_0 \La Q_{\U \mu} \mid w }  +  \ang{  \La Q_{\U \mu} \mid \p_t w } \\
& =  -\frac{\mu'}{\mu} \ang{  \La_0 \La Q_{\U \mu} \mid w }  +  \ang{ \La Q_{\U \mu} \mid\dot \Phi - \p_t \Phi}   \\
& = II_a + II_b 
}
Write, 
\EQ{
II_a = (a- \mu') \ang{ (r \La_0 \La Q)_{\U \mu} \mid r^{-1} w}  - a  \ang{ (r \La_0 \La Q)_{\U \mu} \mid r^{-1} w}
}
Using~\eqref{eq:dotPhi-ptPhi} along with~\eqref{eq:AQ} and~\eqref{eq:tiBQ} we see that 
\EQ{
II_b &=  (a - \mu') \Big(\|\La Q \|_{L^2}^2  + a^2\ang{ \La Q \mid \La A}   + \nu^k \ang{ \La Q \mid  \La \ti B} - k \nu^{k+1} \ang{ \Lam Q_{\U \mu} \mid   B_{ \U\la} } \Big) \\
& \quad + (b+ \la')\Big(\ang{ \Lam Q_{\U \mu} \mid \La Q_{\U \la}}   + b^2 \ang{ \Lam Q_{\U \mu} \mid  \La A_{\U \la} }   +  \nu^k \ang{ \Lam Q_{\U \mu} \mid \La B_{\U \la}}   - k \nu^k \ang{ \Lam Q_{\U \mu} \mid  B_{ \U\la} }  \Big)    \\ 
& \quad  +( b' + \gamma_k \frac{\nu^{k}}{\la})  \Big( -2 b \la \ang{ \Lam Q_{\U \mu} \mid  A_{\U \la} }\Big)   
}
Combining the above we obtain the identity 
\EQ{
(a - \mu')& \Big(\|\La Q \|_{L^2}^2  + a^2\ang{ \La Q \mid \La A}   + \nu^k \ang{ \La Q \mid  \La \ti B} - k \nu^{k+1} \ang{ \Lam Q_{\U \mu} \mid   B_{ \U\la} } + \ang{ (r \La_0 \La Q)_{\U \mu} \mid r^{-1} w}\Big) \\
  + (b+ \la')&\Big(\ang{ \Lam Q_{\U \mu} \mid \La Q_{\U \la}}   + b^2 \ang{ \Lam Q_{\U \mu} \mid  \La A_{\U \la} }   +  \nu^k \ang{ \Lam Q_{\U \mu} \mid \La B_{\U \la}}   - k \nu^k \ang{ \Lam Q_{\U \mu} \mid  B_{ \U\la} }  \Big)    \\ 
  + ( b' + \gamma_k& \frac{\nu^{k}}{\la})  \Big( -2 b \la \ang{ \Lam Q_{\U \mu} \mid  A_{\U \la} }\Big) \\
& = a  \ang{ (r \La_0 \La Q)_{\U \mu} \mid r^{-1} w} 
}
Define matrix entries,   
\EQ{
M_{11} &:= \|\La Q \|_{L^2}^2  + a^2\ang{ \La Q \mid \La A}   + \nu^k \langle \La Q \mid  \La \ti B\rangle - k \nu^{k+1} \ang{ \Lam Q_{\U \mu} \mid   B_{ \U\la} } + \ang{ (r \La_0 \La Q)_{\U \mu} \mid  w/r}\\
M_{22} &:= \| \La Q \|_{L^2}^2  + b^2 \ang{ \La Q\mid  \La A} + \nu^k \ang{ \La Q \mid \La B}  - k \nu^{k-1} \ang{ \La Q_{\U \lam} \mid \ti B_{\U \mu}} -  \ang{ (r\La_0 \La Q)_{\U \la} \mid  w/ r }\\
M_{12} &:= \ang{ \Lam Q_{\U \mu} \mid \La Q_{\U \la}}   + b^2 \ang{ \Lam Q_{\U \mu} \mid  \La A_{\U \la} }   +  \nu^k \ang{ \Lam Q_{\U \mu} \mid \La B_{\U \la}}   - k \nu^k \ang{ \Lam Q_{\U \mu} \mid  B_{ \U\la} }  \\
M_{21}&:=  \ang{ \La Q_{\U \lam}  \mid  \La Q_{\U \mu} }   + a^2 \ang{ \La Q_{\U \lam}  \mid \La A_{\U \mu}}  + \nu^k  \ang{ \La Q_{\U \lam}  \mid   \La \ti B_{\U \mu}}  + k \nu^k \ang{ \Lam Q_{\U \lam} \mid \ti B_{\U \mu}} \\
M_{14}& := -2 b \la \ang{ \Lam Q_{\U \mu} \mid  A_{\U \la} } \\
M_{23}& := 2 a \mu  \ang{ \La Q_{\U \lam}  \mid A_{\U\mu}} 
} 
and column vector entries, 
\EQ{
G_1 &:= a  \ang{ (r \La_0 \La Q)_{\U \mu} \mid r^{-1} w}\\ 
G_2 &:= - b\ang{ (r\La_0 \La Q)_{\U \la} \mid  w/ r } 
}
Next, we differentiate~\eqref{eq:muwdot} and~\eqref{eq:lawdot}. 
We compute, 
\EQ{
0 &= \frac{\ud}{\ud t}  \ang{ \La Q_{\U \mu} \mid  \dot w}  = -\frac{\mu'}{\mu} \ang{  \La_0 \La Q_{\U \mu} \mid  \dot w} + \ang{ \La Q_{\U \mu} \mid \p_t \dot w} \\
&  = -\frac{\mu'}{\mu} \ang{  \La_0 \La Q_{\U \mu} \mid  \dot w}  + \ang{ \La Q_{\U \mu} \mid - \LL_\mu \dot w} \\
& \quad + \ang{ \La Q_{\U \mu} \mid -\p_t \dot \Phi + \De \Phi -\frac{1}{r^2} f( \Phi)} + \ang{ \La Q_{\U \mu} \mid -\frac{1}{r^2} \Big( f( \Phi + w) - f( \Phi) - f'( Q_\mu) w \Big) }  \\  
& =( a-\mu')\frac{\nu}{\lam}\ang{  \La_0 \La Q_{\U \mu} \mid  \dot w}   - \frac{a \nu}{\la} \ang{  \La_0 \La Q_{\U \mu} \mid  \dot w}  \\
& \quad + \ang{ \La Q_{\U \mu} \mid -\p_t \dot \Phi + \De \Phi -\frac{1}{r^2} f( \Phi)}  + \ang{ \La Q_{\U \mu} \mid -\frac{1}{r^2} \Big( f( \Phi + w) - f( \Phi) - f'( Q_\mu) w \Big) }  
}
where we have used above that $ \LL_\mu \La Q_{\U \mu} = 0$. Inserting the expresssion~\eqref{eq:ptdotPhi} into the previous line and using 
the fact that $\ang{\La_0 \La Q \mid  \La Q} = 0$ along with~\eqref{eq:AQ},~\eqref{eq:BQ}, and~\eqref{eq:tiBQ}, we arrive at the expression, 
\EQ{
G_3 = M_{31} (a-\mu') + M_{32}  (b+\la') + M_{33} (a'  + \ti \gamma_k \frac{\nu^k}{\mu}) + M_{34} (b'+ \gamma_k \frac{\nu^k}{\la})
}
where 
\begingroup
\allowdisplaybreaks
\begin{align} 
- G_3& :=  - \frac{a \nu}{\la} \ang{  \La_0 \La Q_{\U \mu} \mid  \dot w} + 3\gamma_k \frac{\nu^k}{\la}b^2  \ang{ \La Q_{\U \mu} \mid \La A_{\U \la} }   -  \frac{b^4}{\la}\ang{ \La Q_{\U \mu} \mid \La_0 \La A_{\U \la} }   \\
& \quad    -2 \gamma_k^2  \frac{\nu^{2k}}{\la} \ang{ \La Q_{\U \mu} \mid A_{\U \la} } + 2 \gamma_k\frac{ b^2 \nu^k}{\la}\ang{ \La Q_{\U \mu} \mid  \La_0 A_{\U \la}  }    +   \gamma_k  \frac{\nu^{2k}}{\lam} \ang{ \La Q_{\U \mu} \mid\La B_{\U \la}} \\
&\quad    - 2k \gamma_k  \frac{b^2 \nu^k}{\lam} \ang{ \La Q_{\U \mu} \mid A_{\U \la} }  - 2k \gamma_k  \frac{a b\nu^{k+1}}{\lam} \ang{ \La Q_{\U \mu} \mid A_{\U \la} }  + k  \frac{b^2 \nu^k}{\lam} \ang{ \La Q_{\U \mu} \mid \La B_{\U\la} } \\
& \quad   + k  \frac{ab \nu^{k+1}}{\lam}\ang{ \La Q_{\U \mu} \mid  \La B_{\U\la} }     - \frac{b^2 \nu^k}{\la}\ang{ \La Q_{\U \mu} \mid \La_0 \La B_{\U \la} }   -  k \frac{\nu^{2k}}{\lam} \ang{ \La Q_{\U \mu} \mid B_{\U \la}}  \\
& \quad  +  \frac{ k b^2 \nu^{k}}{\la}\ang{ \La Q_{\U \mu} \mid \La_0 B_{ \U \la} } -  k^2 \frac{b^2 \nu^{k} }{\lam}\ang{ \La Q_{\U \mu} \mid B_{ \U \la} }    -   k^2 \frac{a b  \nu^{k+1}}{\lam} \ang{ \La Q_{\U \mu} \mid B_{ \U \la}}   \\
&\quad    - k(k+1) ab \frac{\nu^{k+1}}{\lam}\ang{\Lam Q_{\U \mu} \mid  B_{\U \lam}}  + k a b\frac{\nu^{k+1}}{\lam} \ang{\Lam Q_{\U \mu} \mid \La_0 B_{\U \lam}}   \\
&\quad -   k(k+1) a^2 \frac{\nu^{k+2}}{\lam}\ang{\Lam Q_{\U \mu} \mid B_{\U \lam}}  -  k  \ti \gamma_k \frac{\nu^{2k+2}}{\lam} \ang{\Lam Q_{\U \mu} \mid B_{\U \lam}}   \\ 
 &\quad   + 3 \ti \gamma_k \frac{a^2\nu^k}{\mu}\ang{ \La Q \mid \La A }    + \frac{a^4\nu}{\lam}\ang{ \La Q \mid \La_0 \La A }      
 + 2 \gamma_k \frac{a^2 \nu^{k+1}}{\lam} \ang{ \La Q\mid  \La_0 A}  \\
 & \quad    -  \ti \gamma_k  \frac{\nu^{2k+1}}{\lam}\ang{ \La Q \mid \La  \ti B }   +  k  \frac{ ab \nu^k}{\lam}\ang{ \La Q \mid  \La \ti B } +  k \frac{a^2 \nu^{k+1}}{\lam}\ang{ \La Q \mid  \La \ti B} \\
& \quad   +  \frac{ a^2 \nu^{k+1}}{\lam} \ang{ \La Q \mid \La_0 \La \ti B }    
 + k \frac{ a b \nu^{k}}{\lam}\ang{ \La Q \mid  \La_0  \ti B }  + k a^2 \frac{ \nu^{k+1}}{\lam} \ang{ \Lam Q \mid  \Lam_0 \ti B}  \\
 &\quad - \ang{ \Lam Q_{\U \mu} \mid \frac{1}{r^2} \Big( f(Q_\la - Q_\mu)  - f(Q_\la) + f(Q_\mu)  - 4  \big(\frac{r}{ \mu}\big)^k (\La Q_{\la})^2  -  4 \big(\frac{r}{ \la}\big)^{-k} (\La Q_{\mu})^2 \Big)} \\
&\quad - \ang{ \Lam Q_{\U \mu} \mid \frac{1}{r^2} \Big( f(\Phi) - f(Q_\la - Q_\mu) -  f'(Q_\la- Q_\mu) (b^2 T_\la - a^2 \ti T_{\mu} \Big) } \\
&\quad - \ang{ \Lam Q_{\U \mu} \mid \frac{1}{r^2}  \Big(f'(Q_\la - Q_\mu) (b^2 T_\la - a^2 \ti T_\mu)  - f'(Q_\la) b^2 T_\la + f'(Q_\mu) a^2 \ti T_\mu\Big)}  \\
&\quad - \ang{ \La Q_{\U \mu} \mid \frac{1}{r^2} \Big( f( \Phi + w) - f( \Phi) - f'( Q_\mu) w \Big) }
 \end{align} 
 \endgroup
and the matrix entries $M_{3j}$ are defined by, 
\EQ{
 M_{31} & = - \frac{ a^3 \nu}{\lam}  \ang{ \La Q \mid  \La_0 \La A } - 2 \gamma_k \frac{a \nu^{k+1}}{\lam}   \ang{ \La Q \mid\La_0 A} - \frac{a \nu^{k+1}}{\lam} \ang{ \La Q \mid  \La_0 \La \ti B}   \\
& \quad   - k  \frac{ b \nu^{k}}{\lam}  \ang{ \La Q \mid \La_0  \ti B}  + 2k \gamma_k  \frac{b\nu^{k+1}}{\lam}  \ang{ \La Q_{\U \mu} \mid  A_{\U \la}}  - k  \frac{b\nu^{k+1}}{\lam} \ang{ \La Q_{\U \mu} \mid  \La B_{\U\la}} \\
& \quad   +  k^2 \frac{b \nu^{k+1}}{\la} \ang{ \La Q_{\U \mu} \mid  B_{ \U \la}}     - k \frac{ a \nu^{k+1}}{\lam} \ang{ \La Q \mid  \La \ti B} +  \frac{\nu}{\lam}\ang{  \La_0 \La Q_{\U \mu} \mid  \dot w}  \\ 
&\quad +  k(k+1) a \frac{\nu^{k+2}}{\lam}\ang{ \Lam Q_{\U \mu} \mid  B_{\U \lam}} 
-  k a \frac{\nu^{k+1}}{\lam} \ang{ \Lam Q \mid  \Lam_0 \ti B } 
}
\EQ{
 M_{32} &= \frac{b}{\la}  \ang{ \La Q_{\U \mu} \mid \La_0 \La Q_{\U \la}}  +  \frac{b^3}{\la} \ang{ \La Q_{\U \mu} \mid \La_0 \La A_{\U \la}}  - 2 \gamma_k \frac{ b \nu^k}{\la} \ang{ \La Q_{\U \mu} \mid \La_0 A_{\U \la} }  \\
& \quad   - k   \frac{ b \nu^{k}}{\la} \ang{ \La Q_{\U \mu} \mid \La_0 B_{ \U \la}}  +  \frac{b \nu^k}{\la} \ang{ \La Q_{\U \mu} \mid \La_0 \La B_{\U \la}}   + 2k \gamma_k \frac{ b \nu^{k}}{\la}  \ang{ \La Q_{\U \mu} \mid A_{\U \la}} \\
& \quad   - k  \frac{ b\nu^k}{\lam}  \ang{ \La Q_{\U \mu} \mid\La B_{\U\la}} +   k^2\frac{b\nu^k}{\lam}  \ang{ \La Q_{\U \mu} \mid B_{ \U \la}}     - k  \frac{a\nu^k}{\lam}  \ang{ \La Q\mid \La \ti B}    \\
&\quad +  k(k+1)  \frac{a\nu^{k+1}}{\lam}\ang{ \Lam Q_{\U \mu} \mid  B_{\U \lam} } - k  \frac{a\nu^{k+1}}{\lam} \ang{ \Lam Q_{\U \mu} \mid  \La_0 B_{\U \lam}}   
}
\EQ{
 M_{33} & = -  \| \La Q \|_{L^2}^2   -3a^2 \ang{ \La Q \mid \La A}    +  \nu^k \ang{ \La Q \mid \La  \ti B} +  k \nu^{k+1} \ang{ \Lam Q_{\U \mu} \mid B_{\U \lam} }   
 }
 \EQ{
 M_{34} & = -  \ang{ \La Q_{\U \mu} \mid \La Q_{\U \la}}  - 3   b^2  \ang{ \La Q_{\U \mu} \mid\La A_{\U \la}} + 2 \gamma_k\nu^k   \ang{ \La Q_{\U \mu} \mid A_{\U \la} }  \\
&  \quad  -  \nu^k  \ang{ \La Q_{\U \mu} \mid \La B_{\U \la}} +  k \nu^k \ang{ \La Q_{\U \mu} \mid B_{\U \la}} 
}

Lastly, we compute, 
\EQ{
0 &= \frac{\ud}{\ud t} \ang{ \La Q_{\U \la} \mid \dot w}  \\
& =  -\frac{\la'}{\la} \ang{ \La_0 \La Q_{\U \la} \mid \dot w} + \ang{ \La Q_{\U \la} \mid \p_t \dot w}  \\
& = -\frac{\la'}{\la} \ang{ \La_0 \La Q_{\U \la} \mid \dot w} + \ang{ \La Q_{\U \la} \mid - \LL_\la  w }  + \ang{ \La Q_{\U \la} \mid  -  \p_t \dot \Phi  + \De \Phi  - \frac{1}{r^2} f( \Phi) }  \\
&\quad -  \ang{ \La Q_{\U \la} \mid   \frac{1}{r^2}  \Big ( f( \Phi +w) - f( \Phi) - f'(Q_\la) w \Big) }    \\
& = ( b+ \la') (-\frac{1}{\la} \ang{ \La_0 \La Q_{\U \la} \mid \dot w}   +\frac{b}{\la} \ang{ \La_0 \La Q_{\U \la} \mid \dot w}  \\
& \quad + \ang{ \La Q_{\U \la} \mid  -  \p_t \dot \Phi  + \De \Phi  - \frac{1}{r^2} f( \Phi) }  -  \ang{ \La Q_{\U \la} \mid   \frac{1}{r^2}  \Big ( f( \Phi +w) - f( \Phi) - f'(Q_\la) w \Big) } 
}
where in the second-to-last line we used that $\ang{ \La Q_{\U \la}  \mid -  \LL_\la w  } = -\ang{ \LL_\la \La Q_{\U \la} \mid w} = 0$. 
Inserting the expresssion~\eqref{eq:ptdotPhi} into the previous line and using 
the fact that $\ang{\La_0 \La Q \mid  \La Q} = 0$ along with~\eqref{eq:AQ},~\eqref{eq:BQ}, and~\eqref{eq:tiBQ}, we arrive at the expression, 
\EQ{
G_4 = M_{41} (a-\mu') + M_{42} (b+\la') + M_{43} (a'  + \ti \gamma_k \frac{\nu^k}{\mu}) + M_{44} (b'+ \gamma_k \frac{\nu^k}{\la})
}
where 
\begingroup
\allowdisplaybreaks
\begin{align} 
-G_4 & := \frac{b}{\la} \ang{ \La_0 \La Q_{\U \la} \mid \dot w}   + 3\gamma_k \frac{\nu^k}{\la}b^2 \ang{  \Lam Q \mid \La A}   -  \frac{b^4}{\la} \ang{  \Lam Q \mid \La_0 \La A }   + 2 \gamma_k\frac{ b^2 \nu^k}{\la} \ang{  \Lam Q \mid  \La_0 A}    \\
& \quad   +   \gamma_k  \frac{\nu^{2k}}{\lam} \ang{  \Lam Q \mid \La B }      + k  \frac{b^2 \nu^k}{\lam} \ang{  \Lam Q \mid  \La B }  + k  \frac{ab \nu^{k+1}}{\lam} \ang{  \Lam Q \mid  \La B}  \\
&\quad    - \frac{b^2 \nu^k}{\la} \ang{  \Lam Q \mid \La_0 \La B}    +  \frac{ k b^2 \nu^{k}}{\la} \ang{  \Lam Q \mid \La_0 B} + k a b\frac{\nu^{k+1}}{\lam} \ang{ \Lam Q \mid  \La_0 B}  \\
 &\quad   + 3 \ti \gamma_k \frac{a^2\nu^k}{\mu} \ang{  \Lam Q_{\U \lam} \mid \La A_{\U \mu} }    + \frac{a^4\nu}{\lam}  \ang{  \Lam Q_{\U \lam} \mid\La_0 \La A_{\U \mu}}    +   2  \ti \gamma_k^2 \frac{\nu^{2k+1}}{\la}  \ang{  \Lam Q_{\U \lam} \mid A_{\U \mu} }   \\
 & \quad  + 2k \ga_k \frac{ a b \nu^k}{\lam}  \ang{  \Lam Q_{\U \lam} \mid A_{\U \mu} } +  2k \ga_k\frac{a^2 \nu^{k+1}}{\lam}  \ang{  \Lam Q_{\U \lam} \mid A_{\U \mu}}  +  k \ti \gamma_k  \frac{\nu^{2k+1}}{\lam} \ang{ \Lam Q_{\U \lam} \mid  \ti B_{\U \mu} }   \\
 &\quad  + 2 \gamma_k \frac{a^2 \nu^{k+1}}{\lam}  \ang{  \Lam Q_{\U \lam} \mid  \La_0 A_{\U \mu} }   -  \ti \gamma_k  \frac{\nu^{2k+1}}{\lam} \ang{  \Lam Q_{\U \lam} \mid \La  \ti B_{\U \mu} }  +  k  \frac{ ab \nu^k}{\lam}  \ang{  \Lam Q_{\U \lam} \mid \La \ti B_{\U \mu} } \\
 & \quad  +  k \frac{a^2 \nu^{k+1}}{\lam}   \ang{  \Lam Q_{\U \lam} \mid\La \ti B_{\U \mu}}   +  \frac{ a^2 \nu^{k+1}}{\lam}  \ang{  \Lam Q_{\U \lam} \mid \La_0 \La \ti B_{\U \mu} }   + k \gamma_k \frac{\nu^{2k-1}}{\lam}  \ang{  \Lam Q_{\U \lam} \mid\ti B_{\U \mu}}   \\
 & \quad  + k(k-1) \frac{b^2 \nu^{k-1}}{\lam}  \ang{  \Lam Q_{\U \lam} \mid \ti B_{\U \mu} }  + k(k-1) \frac{ab  \nu^{k}}{\lam}  \ang{  \Lam Q_{\U \lam} \mid \ti B_{\U \mu}}   + k \frac{ a b \nu^{k}}{\lam}  \ang{  \Lam Q_{\U \lam} \mid \La_0  \ti B_{\U \mu} } \\
 &\quad   + k^2 a^2 \frac{\nu^{k+1}}{ \lam}\ang{ \Lam Q_{\U \lam} \mid  \ti B_{\U \mu}}  + k a^2 \frac{ \nu^{k+1}}{\lam} \ang{ \Lam Q_{\U \lam} \mid  \Lam_0 \ti B_{\U \mu}}   + k^2 a b \frac{\nu^k}{\lam} \ang{ \Lam Q_{\U \lam} \mid  \ti B_{\U \mu}}   \\ 
 &\quad - \ang{  \Lam Q_{\U \lam} \mid\frac{1}{r^2} \Big( f(Q_\la - Q_\mu)  - f(Q_\la) + f(Q_\mu)  - 4  \big(\frac{r}{ \mu}\big)^k (\La Q_{\la})^2  -  4 \big(\frac{r}{ \la}\big)^{-k} (\La Q_{\mu})^2 \Big)} \\
&\quad - \ang{  \Lam Q_{\U \lam} \mid\frac{1}{r^2} \Big( f(\Phi) - f(Q_\la - Q_\mu) -  f'(Q_\la- Q_\mu) (b^2 T_\la - a^2 \ti T_{\mu} \Big)}   \\
&\quad - \ang{  \Lam Q_{\U \lam} \mid\frac{1}{r^2}  \Big(f'(Q_\la - Q_\mu) (b^2 T_\la - a^2 \ti T_\mu)  - f'(Q_\la) b^2 T_\la + f'(Q_\mu) a^2 \ti T_\mu\Big)}  \\
& \quad  -  \ang{ \La Q_{\U \la} \mid   \frac{1}{r^2}  \Big ( f( \Phi +w) - f( \Phi) - f'(Q_\la) w \Big) }
\end{align} 
\endgroup 
and  the last row of matrix entries $M_{4j}$ are given by, 
\EQ{
 M_{41} &= -\frac{a \nu }{\lam} \ang{  \Lam Q_{\U \lam} \mid \La_0 \La Q_{\U \mu} }  - \frac{ a^3 \nu}{\lam}  \ang{  \Lam Q_{\U \lam} \mid  \La_0 \La A_{\U \mu} } - 2 \gamma_k \frac{a \nu^{k+1}}{\lam} \ang{  \Lam Q_{\U \lam} \mid  \La_0 A_{\U \mu}} \\
& \quad  - \frac{a \nu^{k+1}}{\lam}  \ang{  \Lam Q_{\U \lam} \mid \La_0 \La \ti B_{\U \mu} }  
   - k  \frac{ b \nu^{k}}{\lam}  \ang{  \Lam Q_{\U \lam} \mid \La_0  \ti B_{\U \mu}}  - k  \frac{b\nu^{k+1}}{\lam}  \ang{  \Lam Q_{\U \lam} \mid \La B_{\U\la} }  \\
&  \quad     - 2k \ga_k \frac{a \nu^{k+1}}{\lam}  \ang{  \Lam Q_{\U \lam} \mid A_{\U \mu}}   - k \frac{ a \nu^{k+1}}{\lam}  \ang{  \Lam Q_{\U \lam} \mid \La \ti B_{\U \mu}}   - k(k-1) \frac{  b  \nu^{k}}{\lam}  \ang{  \Lam Q_{\U \lam} \mid \ti B_{\U \mu} } \\
& \quad -  k^2 a \frac{\nu^{k+1}}{\lam} \ang{ \Lam Q_{\U \lam} \mid \ti B_{\U \mu} }  - k a \frac{\nu^{k+1}}{\lam}  \ang{ \Lam Q_{\U \lam} \mid  \Lam_0 \ti B_{\U \mu}  }\\ 
}
\EQ{
 M_{42} &  =    \frac{b^3}{\la} \ang{  \Lam Q\mid \La_0 \La A}  - 2 \gamma_k \frac{ b \nu^k}{\la}     \ang{  \Lam Q\mid \La_0 A}  - k   \frac{ b \nu^{k}}{\la} \ang{  \Lam Q\mid \La_0 B}  +  \frac{b \nu^k}{\la} \ang{  \Lam Q\mid \La_0 \La B }  \\
&   \quad  - k  \frac{ b\nu^k}{\lam}  \ang{  \Lam Q \mid \La B}  - 2k \ga_k  \frac{a\nu^k}{\lam}  \ang{  \Lam Q_{\U \lam} \mid A_{\U \mu}} - k  \frac{a\nu^k}{\lam}  \ang{  \Lam Q_{\U \lam} \mid \La \ti B_{\U \mu}  } \\
&  \quad   - k(k-1) \frac{ b \nu^{k-1}}{\lam} \ang{  \Lam Q_{\U \lam} \mid \ti B_{\U \mu}}  -\frac{1}{\la} \ang{ \La_0 \La Q_{\U \la} \mid \dot w}   \\ 
&\quad -  k^2 a \frac{\nu^k}{\lam}\ang{ \Lam Q_{\U \lam} \mid  \ti B_{\U \mu}} -  k a \frac{\nu^{k+1}}{\lam} \ang{ \Lam Q\mid  \La_0 B}  
}
\EQ{
 M_{43} & =  -  \ang{  \Lam Q_{\U \lam} \mid \La Q_{\U \mu}}    -3a^2  \ang{  \Lam Q_{\U \lam} \mid\La A_{\U \mu}}    -  2 \ti \gamma_k  \nu^k  \ang{  \Lam Q_{\U \lam} \mid A_{\U \mu} } +  \nu^k \ang{  \Lam Q_{\U \lam} \mid \La  \ti B_{\U \mu}}  \\
 &\quad  -  k \nu^k \ang{ \Lam Q_{\U \lam} \mid \ti B_{\U \mu} }  \\
 M_{44} & = - \| \Lam Q \|_{L^2}^2 - 3   b^2 \ang{  \Lam Q \mid \La A}     -  \nu^k \ang{  \Lam Q \mid\La B}   - k \nu^{k-1}\ang{  \Lam Q_{\U \lam} \mid \ti B_{\U \mu}}
 }
Putting this all together we arrive at a $4\times 4$ linear system for $ \Big((a- \mu'), (b + \la'), (a' + \ti \gamma_k \nu^k/ \mu ), (b' + \gamma_k \nu^k/ \la)\Big)$.
  \EQ{
 \pmat{ M_{11} & M_{12} & 0 & M_{14}  \\ M_{21} & M_{22} & M_{23} & 0  \\ M_{31}   & M_{32} & M_{33} & M_{34} \\ M_{41} & M_{42}   & M_{43} & M_{44}}\pmat{ (a- \mu') \\  (b + \la') \\  (a' + \ti \gamma_k \nu^k/ \mu ) \\  (b' + \gamma_k \nu^k/ \la)}  =  \pmat{ G_1 \\ G_2 \\ G_3 \\ G_4} 
 }
 We claim that the matrix $M$ above is invertible. To see this we 
 express $M$ in $2 \times 2$ block form, 
  \EQ{
 M = \pmat{M_1 & M_2 \\ M_3 & M_4} 
 } 
For the entries in $M_1, M_4$ we deduce using Lemma~\ref{l:ABBQ} that, 
\EQ{
M_{11}, M_{22},  M_{33}, M_{44} = O(1)  \\
\abs{M_{12}}, \abs{M_{21}}, \abs{M_{34}}, \abs{ M_{43}} \lesssim \nu^{k-1}  
}
Hence each of $M_1$ and $M_4$ are diagonally dominant and hence invertible with determinant $ = O(1)$. Next note that for $M_2$ we have, 
\EQ{
 \abs{ M_{14}} \lesssim \abs{b }\lam \nu^{k-1}, \quad \abs{ M_{23}} \lesssim \abs{a} \mu \nu^{k-1}  \Longrightarrow \| M_2 \| \lesssim (\abs{a} + \abs{b}) \nu^{k-1} 
}
 and for $M_3$, 
 \EQ{
 \abs{M_{31}} &\lesssim\frac{\nu}{\lam} \| \dot w \|_{L^2} +  \abs{ a}^3 \frac{\nu}{\lam} + \abs{a} \frac{\nu^{k+1}}{\lam} + \abs{b} \frac{\nu^k}{\lam}   \\
\abs{M_{32}} &\lesssim  \abs{b} \frac{\nu^{k-1}}{\lam}  + \abs{a} \frac{\nu^k}{\lam} \\
\abs{  M_{41}} & \lesssim \abs{a} \frac{\nu^k}{\lam} + \abs{b} \frac{\nu^k}{\lam}  \\
 \abs{M_{42}} & \lesssim \frac{1}{\lam} \| \dot w \|_{L^2} + \frac{ \abs{b}^3}{\lam} + \frac{\abs{b} \nu^k}{\lam}  + \frac{\abs{a} \nu^k}{\lam}    \\
 \| M_3 \| &\lesssim \frac{1}{\lam} \| \dot w \|_{L^2} + ( \abs{a} + \abs{b}) \frac{\nu^{k-1}}{\lam} 
 } 
 It follows that the matrix $M_4 - M_3 M_1^{-1} M_2$ is invertible with determinant $= O(1)$ under the hypothesis of the lemma and thus the inverse of $M$ is given by 
%
 \EQ{
 M^{-1} = \pmat{M_1^{-1} + M_1^{-1} M_2( M_4 - M_3 M_1^{-1} M_2)^{-1} M_3 M_1^{-1}  & - M_1^{-1} M_2   ( M_4 - M_3 M_1^{-1} M_2)^{-1} \\  - ( M_4 - M_3 M_1^{-1} M_2)^{-1} M_3 M_1^{-1} &   ( M_4 - M_3 M_1^{-1} M_2)^{-1} }  
 }
 and we may solve 
 \EQ{
\pmat{ a - \mu' \\ b + \la' \\ a'+ \ti \gamma_k \frac{\nu^{k}}{\mu}  \\ b' + \gamma_k \frac{\nu^k}{\lam}} = M^{-1} \pmat{ G_1 \\G_2 \\G_3 \\ G_4} 
 }
Next using Lemma~\ref{l:ABBQ} we observe the bounds, 
\EQ{  \label{eq:G12est}
\abs{G_1} &\lesssim  \abs{a} \| w \|_H  \\
\abs{G_2} & \lesssim  \abs{b} \| w \|_H  
}
We next estimate the size of $G_3$ and $G_4$ using Lemma~\ref{l:ABBQ},  Lemma~\ref{l:fest1}, and  Lemma~\ref{l:w^2est}. 
\EQ{
 \abs{G_3} 
 & \lesssim \frac{\abs{a}  \nu}{\la}   \| \dot w \|_{L^2}  +  \frac{\nu}{\la} \| \bs w \|_{\HH}^2 + \frac{\nu^{2k-1}}{\lam} + \frac{b^2 \nu^{k-1}}{\lam} +  \frac{b^4}{\lam} + \frac{a^4}{\lam} \\ 
 \abs{ G_4} 
 & \lesssim  \frac{ \abs{b}}{\la} \| \dot w \|_{L^2} + \frac{1}{\la} \| \bs w \|_{\HH}^2 +  \frac{\nu^{2k-o(1)}}{\lam} + \frac{b^4}{\la} + \frac{a^4}{\lam} 
 }
It follows, using that $\nu \simeq \lam$,  that 
\EQ{
\abs{ a - \mu'} &\lesssim \abs{ G_1} + \nu^{k-1} \abs{G_2} +  \| M_2\| ( \abs{G_3} + \abs{G_4})  \\
& \lesssim  (\abs{a} + \abs{b} ) \| \bs w \|_{\HH} +  (\abs{a }+ \abs{b})\nu^{2k} + ( \abs{a}^5 + \abs{b}^5) \nu^{k-2} 
}
Similarly, 
\EQ{
\abs{ b+ \lam'}& \lesssim  \abs{ G_2} + \nu^{k-1} \abs{G_1} + \|M_2\|( \abs{G_3} + \abs{G_4})\\
& \lesssim (\abs{a} + \abs{b} ) \| \bs w \|_{\HH} +( \abs{a} + \abs{b})\nu^{2k-1}  + ( \abs{a}^5 + \abs{b}^5) \nu^{k-2} 
}
We also deduce that, 
\EQ{
\abs{a' +  \ti \gamma_k \frac{\nu^k}{\mu}} & \lesssim  \abs{ G_3} + \nu^{k-1} \abs{G_4} + \|M_3 \|( \abs{G_1} + \abs{G_2})  \\ 
& \lesssim  \frac{\abs{a}\nu  + \abs{b}\nu}{\la}   \| \dot w \|_{L^2}  +  \frac{\nu}{\la} \| \bs w \|_{\HH}^2 + \frac{\nu^{2k-1}}{\lam} + \frac{b^2 \nu^{k-1}}{\lam} +  \frac{b^4\nu}{\lam} + \frac{a^4\nu}{\lam}
}
and
\EQ{
\abs{b' +  \gamma_k \frac{\nu^k}{\lam}} & \lesssim \abs{ G_4} + \nu^{k-1} \abs{G_3} + \|M_3 \|( \abs{G_1} + \abs{G_2})   \\
& \lesssim  \frac{\abs{a}  + \abs{b}}{\la}   \| \dot w \|_{L^2}  +  \frac{1}{\la} \| \bs w \|_{\HH}^2 + \frac{\nu^{2k-o(1)}}{\lam}  +  \frac{b^4}{\lam} + \frac{a^4}{\lam}
}
This completes the proof. 
\end{proof}

\section{Construction of the solution} \label{s:construction} 

The goal of this section is to prove Theorem~\ref{t:refined}. First, we establish a few preliminary technical lemmas in Section~\ref{s:virial}. Then we prove energy type estimates in the spaces $\HH$, $\HH^2$ and $\bs \Lam^{-1} \HH$ for functions $\bs w(t)$ given by modulation in Lemma~\ref{l:mod2} -- this is the technical heart of the proof of Theorem~\ref{t:refined}. Finally, we construct the desired two-bubble solution in Section~\ref{s:construct}.

%
%

\subsection{The truncated virial operators} \label{s:virial} 
We define truncated virial operators $\calA(\la)$ and $\calA_0(\la)$, and state related estimates. Nearly identical operators were introduced by the first author in~\cite{JJ-AJM} and used crucially by the authors in~\cite{JL1}. We require a slight modification, which we give below. 

\begin{lem} \emph{\cite[Lemma 4.6]{JJ-AJM}}
  \label{l:pdef}
  For each $c, R > 0$ there exists a function $p(r) = p_{c, R}(r) \in C^{5,1}((0, +\infty))$ with the following properties:
  \begin{enumerate}
    \item $p(r) = \frac{1}{2} r^2$ for $r \leq R$, 
    \item there exists $\ti R = \ti R(R, c)> R$ such that $p(r) \equiv \tx{const}$ for $r \geq \ti R$, 
    \item $|p'(r)| \lesssim r$ and $|p''(r)| \lesssim 1$ for all $r > 0$, with constants independent of $c, R$, 
    \item $p''(r) \geq -c$ and $\frac 1r p'(r) \geq -c$, for all $r > 0$, 
    \item $ \abs{r \p_r \De p} \le c$ for all $r > 0$,
    \item $\De^2 p(r) \leq c\cdot r^{-2}$, for all $r > 0$,
    \item $\De^3 p(r)  \ge -c  \cdot r^{-4}$ for all $r >0$, 
    \item $\big|r\big(\frac{p'(r)}{r}\big)'\big| \leq c$, for all $r > 0$, 
    \item $\abs{ r \bigg(r\big(\frac{p'(r)}{r}\big)'  \bigg)' } \leq c$ for all $r >0$.
 \end{enumerate}
\end{lem}

\begin{proof} 
The proof is essentially the same as~\cite[Proof of Lemma 4.6]{JJ-AJM}, but the function constructed there is only $C^{3, 1}$ so we briefly demonstrate how to further smooth it. Note that it suffices to consider the case $R=1$.   Following~\cite{JJ-AJM} we define, 
\EQ{
q_0(r) := \begin{cases} \frac{1}{2} r^2 \mif r \le 1 \\ \frac{1}{2} r^2 + c_1 \Big( \frac{1}{2} (r-1)^2 - \frac{1}{4} \log r (r^2 -1)\Big) \mif r \ge 1 \end{cases} 
}
Note that for $r>1$ 
\EQ{
q_0'(r) &= r \Big( 1- \frac{c_1 \log r}{2} \Big) + c_1 \Big( \frac{3}{4} r - 1 + \frac{1}{4r} \Big)  , \quad 
q_0''(r) = \Big( 1- \frac{c_1 \log r}{2} \Big) + c_1  \Big( \frac{1}{4} - \frac{1}{4r^2} \Big)  \\ 
q_0'''(r) &= -c_1 \frac{r^2 - 1}{2 r^3}   , \quad 
q_0^{(4)}(r) = \frac{ \frac{1}{2} c_1 r^2 - \frac{3}{2} c_1}{r^4}  , \quad 
q_0^{(5)}(r) = \frac{ -c_1 r^2 + 6 c_1}{r^5} 
}
From the above it is clear that $q_0$ is $C^{3, 1}(0, \infty)$.  Since we require a $C^{5,1}$ function we modify it slightly. Define $p_0(r)$ by 
\EQ{
p_0(r) = \begin{cases} \frac{1}{2} r^2  \mif r \le 1 \\ q_0(r) + \frac{c_1}{24} ( \log r)^4  - \frac{c_1}{120} ( \log r)^5 \mif r  \ge 1 \end{cases} 
}
A direct computation shows that $p_0^{(4)}(1) = p_0^{(5)}(1) = 0$, which yields the desired smoothness. First we check all the properties (except for {\it (2)}) for $q_0(r)$ in the regime $1 <r < R_0:= e^{2/c_1}$  The properties {\it (3), (4), (6), (8)} were verified in~\cite[Lemma 4.6]{JJ-AJM}. For property {\it (5)} we have using {\it (4)} that 
\EQ{
\abs{r \p_r \De q_0 (r) }&= \abs{r \p_r ( q_0'' + \frac{q'}{r})} = \abs{r q_0 ''' + r  \big( \frac{q_0'}{r})' }  \lesssim c_1 \frac{r^2 -1}{2r^2}  + c_1  \lesssim c_1
}
as desired. Similarly, 
\EQ{
r \p_r \Big( r \p_r \big( \frac{q_0'}{r} \big) \Big) &=  r \p_r \big( \frac{q_0'}{r} \big) + r^2 \p_r^2 \big( \frac{q_0'}{r} \big) =  r \p_r \big( \frac{q_0'}{r} \big)  - r^2 \p_r \Big(  \frac{q_0'}{r^2} \Big) + r^2 \p_r \Big( \frac{q_0''}{r} \Big) \\ 
& = r \p_r \big( \frac{q_0'}{r} \big) - 2 q_0'' + 2 \frac{q_0'}{r} + r q_0''' = - r  \big( \frac{q_0'}{r} \big)'  + r q_0''' 
}
and {\it (9)} follows. Finally, 
\EQ{
\De^2 q_0 = -\frac{c_1}{r^3} , \quad \De^3 q_0 = - \frac{9 c_1}{r^5}, 
}
from which {\it (7)} follows.

Since the terms $ \frac{c_1}{24} ( \log r)^4  - \frac{c_1}{120} ( \log r)^5$ are perturbative with respect to properties {\it (3) - (9)} in the regime $1 <r <R_0 = e^{2/c_1}$, the function $p_0(r)$ inherits these properties from $q_0(r)$.  It remains to find a suitable truncation to ensure property {\it (2)}, but the identical truncation used in~\cite[Lemma 4.6]{JJ-AJM} works here as well,  suitably smoothed to ensure $C^{5, 1}$. Define the function $E_j(r):= \frac{1}{j \!} r^j \chi(r)$ where $\chi(r)$ is a standard smooth cutoff. Set 
\EQ{
p(r) = \begin{cases} p_0(r)  \mif r \le R_0 \\  p_0(R_0) + \sum_{j =1}^5 p_0^{(j)}(R_0) R_0^j E_j( r/ R_0 - 1) \mif r \ge R_0 \end{cases} 
}
Then, as in~\cite[Proof of Lemma 4.6]{JJ-AJM}, one can verify that $p(r)$ satisfies $p(r) =$ constant for $r \ge 3 R_0$ and inherits the rest of the desired properties from $p_0(r)$.  
%
\end{proof}

For each $\lambda > 0$ define $\A(\lambda)$ and $\A_0(\lambda)$ as follows, 
\begin{align}
  [\A(\lambda)w](r) &:= p'\big(\frac{r}{\lambda}\big)\cdot \p_r w(r), \label{eq:opA} \\
  [\A_0(\lambda)w](r) &:= \big(\frac{1}{2\lambda}p''\big(\frac{r}{\lambda}\big) + \frac{1}{2r}p'\big(\frac{r}{\lambda}\big)\big)w(r) + p'\big(\frac{r}{\lambda}\big)\cdot\p_r w(r). \label{eq:opA0}
\end{align}
Note the similarity between $\A$ and $\frac{1}{\la} \La$ and between $\A_0$ and $\frac{1}{\la} \La_0$. 


 Recall the notation, 
$
\LL_0:= -\De + \frac{k^2}{r^2} . 
$
\begin{lem} \emph{\cite[Lemma 5.5]{JJ-AJM}}
  \label{l:opA}
  Let $c_0>0$ be arbitrary. There exists $c>0$ small enough and $R, \ti R>0$ large enough in Lemma~\ref{l:pdef} so that the operators $\A(\lambda)$ and $\A_0(\lambda)$ defined in~\eqref{eq:opA} and~\eqref{eq:opA0} have the following properties:
  \begin{itemize}[leftmargin=0.5cm]
    \item the families $\{\A(\lambda): \lambda > 0\}$, $\{\A_0(\lambda): \lambda > 0\}$, $\{\lambda\partial_\lambda \A(\lambda): \lambda > 0\}$
      and $\{\lambda\partial_\lambda \A_0(\lambda): \lambda > 0\}$ are bounded in $\mathscr{L}(H; L^2)$, with the bound depending only on the choice of the function $p(r)$,
      \item In addition, the operators $\A_0(\lam)$ and $\lam \p_\lam \A_0(\lam)$ satisfy the bounds
      \EQ{ \label{eq:prA0} 
       \| \p_r \A_0(\lam) w \|_{L^2}  +  \| r^{-1} \A_0(\lam) w \|_{L^2} &\lesssim  \| \p_r w \|_H +  \frac{1}{\lam} \| w \|_H \\
       \| \p_r \lam \p_\lam \A_0(\lam) w \|_{L^2}  +  \| r^{-1} \lam \p_\lam \A_0(\lam) w \|_{L^2} &\lesssim  \| \p_r w \|_H +  \frac{1}{\lam} \| w \|_H 
      }
      with a constant that depends only on the choice of the function $p(r)$, 
  
   
    \item For all $w \in H \cap H^2$ we have  
\EQ{
        \label{eq:A-pohozaev}
        \ang{\A_0(\lambda)w  \mid  \LL_0 w} \ge -\frac{c_0}{\lambda}\|w\|_{H}^2 + \frac{1}{\lambda}\int_0^{R\lambda}\Big((\partial_r w)^2 + \frac{k^2}{r^2}w^2\Big) \udr, 
        }
        
        \item For all $w, \p_r w \in H \cap H^2$ we have 
        \EQ{ \label{eq:A-pohozaev2} 
        \ang{\A_0(\lambda)w  \mid \LL_0^2w} \ge - \frac{c_0}{\lam} \| w \|_{H^2}^2 + \frac{2}{\lam}  \int_0^{R \lam}  ( \LL_0 w)^2 \, \rdr 
        }
        \item Moreover, for $\la, \mu >0$ with $\la/\mu \ll 1$, 
        \EQ{
         \label{eq:L0-A0-wm}
      \|\Lambda_0 \Lambda Q_\uln{\lambda} - \A_0(\lambda)\Lambda Q_{\lambda}\|_{L^2} \leq c_0,
        }
\item Finally, let $P_\lam(r)$ denote the potential, 
\EQ{
P_{\lam}(r):= \frac{1}{r^2} ( f'(Q_\lam) - k^2) 
}
and let $\Ks_\lam$ denote the operator given by  
\EQ{ 
\Ks_\lam w := 2 P_\lam \LL_0 w  - 2 \p_r P_\lam \p_r w + (P_\lam^2 - \De P_\lam) w
}
We have, 
\EQ{ \label{eq:vir-new} 
 \abs{  \ang{ \calA_0(\la) w \mid P_\lam(r) w}  - \ang{  \frac{1}{\la} \La_0 w \mid P_\lam(r) w} } \le \frac{c_0}{\la} \| w \|_{H}^2 
}
as well as, 
\EQ{ \label{eq:vir-new2} 
\abs{ \ang{ \calA_0(\la) w \mid \Ks_\lam w}  - \ang{  \frac{1}{\la} \La_0 w \mid \Ks_\lam w} } \le \frac{c_0}{\la} \| w \|_{H^2}^2 
}

  \end{itemize}
\end{lem}


\begin{proof} 
The conditions $ w, \p_r w \in H^2$ are required only to ensure that the left-hand-side of 
  ~\eqref{eq:A-pohozaev} and~\eqref{eq:A-pohozaev2} are well defined,
  but do not appear on the right-hand-side of the estimates. 

The first, third, and fifth bullet points we proved in~\cite[Lemma~5.5]{JJ-AJM}. The second bullet point is a direct consequence of the definition of $\calA_0(\lam)$. 

Note that it suffices to prove~\eqref{eq:A-pohozaev2} for $\lam =1$. We may also assume (by approximation) that $p, w \in C^\infty_0(0, \infty)$ -- here we will use that $p \in C^{5, 1}$. 
First, we use the fact that $\ang{ \A_0(1) h \mid h} = 0$ to deduce that  
\EQ{
\ang{\A_0(1)w  \mid \LL_0^2w} &= \ang{ \LL_0\A_0(1)w  \mid  \LL_0 w} \\
& = \ang{ [ \LL_0, \A_0(1)] w \mid \LL_0 w} + \ang{ \A_0(1) \LL_0 w \mid \LL_0 w}  \\
& = \ang{ [ \LL_0, \A_0(1)] w \mid \LL_0 w} 
}
Next, we compute the commutator, $[ \LL_0, \A_0(1)] w$. Writing $\A_0(1)w = \frac{1}{2} \De p w + \p_r p \p_r w$ we have 
\EQ{
[\LL_0, \frac{1}{2} \De p] w = -\frac{1}{2} (\De^2 p )w - ( \p_r \De p) \p_r w
}
and, 
\EQ{
[\LL_0, \p_r p \p_r] w &= - \De( \p_r p \p_r w) +\frac{k^2}{r^2} \p_r p \p_r w + \p_r p \p_r( \De w) - \p_r p \p_r( \frac{k^2}{r^2} w)  \\
& =   2 \frac{p'}{r}  \LL_0 w    - r\p_r \Big(  r \p_r \big( \frac{p'}{r} \big)\Big) \frac{ \p_r w}{r}    -2  r \p_r \big( \frac{p'}{r} \big)  \De w
} 
Hence, 
\EQ{ \label{eq:ll0a0} 
\ang{ [ \LL_0, \A_0(1)] w \mid \LL_0 w}  &=  2  \int_0^\infty \frac{p'}{r}  (\LL_0 w)^2 \, \rdr  - \int_0^\infty   r\p_r \Big(  r \p_r \big( \frac{p'}{r} \big)\Big) \frac{ \p_r w}{r}  \LL_0 w \, \rdr  \\
& \quad - \int_0^\infty r \p_r \big( \frac{p'}{r} \big)  \De w \LL_0 w \, \rdr   \\
& \quad - \frac{1}{2}  \int_0^\infty (\De^2 p )w \LL_0 w \rdr - \int_0^\infty ( \p_r \De p) \p_r w  \LL_0 w \rdr 
} 
We consider separately each of the terms on the right above. For the first term we use properties {\it (1),(4)} from Lemma~\ref{l:pdef} to deduce that 
\EQ{
2  \int_0^\infty \frac{p'}{r}  (\LL_0 w)^2 \, \rdr  &= 2  \int_0^R  (\LL_0 w)^2 \, \rdr + 2  \int_R^\infty \frac{p'}{r}  (\LL_0 w)^2 \, \rdr  \\
& \ge 2  \int_0^R  (\LL_0 w)^2 \, \rdr - c \| w \|_{H^2}^2 
}
By property {\it(9)}  from Lemma~\ref{l:pdef} we estimate the second term by 
\EQ{
\abs{ \int_0^\infty  r \p_r \Big(  r \p_r \big( \frac{p'}{r} \big)\Big)  \frac{ \p_r w }{r}  \LL_0 w \, \rdr } \le c  \| r^{-1}\p_r w \|_{L^2} \|\LL_0 w \|_{L^2} \lesssim c \| w \|_{H^2} 
}
The third term is controlled by property {\it(8)}, 
\EQ{
\abs{- \int_0^\infty    r \p_r \big( \frac{p'}{r} \big) \De w   \LL_0 w \, \rdr}   \le c  \| \De w \|_{L^2} \|\LL_0 w \|_{L^2} \lesssim c \| w \|_{H^2} 
}
For the fourth term we expand and integrate by parts as follows, 
\EQ{
- \frac{1}{2}  \int_0^\infty (\De^2 p )w& \LL_0 w \rdr  =  \frac{1}{2}  \int_0^\infty (\De^2 p )w \De w \rdr  - \frac{1}{2}  \int_0^\infty (\De^2 p )w^2 \frac{k^2}{r^2}  \rdr  \\
& = \frac{1}{4}  \int_0^\infty (\De^3 p ) w^2  \rdr  - \frac{1}{2} \int_0^\infty (\De^2 p ) (\p_r w)^2 \rdr  - \frac{1}{2}  \int_0^\infty (\De^2 p )w^2 \frac{k^2}{r^2}  \rdr 
}
For the first term on the right we use property {\it (7)}, 
\EQ{
\frac{1}{4}  \int_0^\infty (\De^3 p ) w^2  \rdr   \ge -  \frac{1}{4} c \int_0^\infty r^{-4} w^2  \rdr   \gtrsim  -  c \| r^{-2} w \|_{L^2}^2 \gtrsim  - c \|  w \|_{H^2}^2
}
For the remaining two terms we use property {\it (6)}, 
\EQ{
- \frac{1}{2} \int_0^\infty (\De^2 p ) (\p_r w)^2 \rdr  - \frac{1}{2}  \int_0^\infty (\De^2 p )w^2 \frac{k^2}{r^2}  \rdr  \gtrsim -c  \|  w \|_{H^2}^2
}
Finally, we estimate the last term in~\eqref{eq:ll0a0} using property {\it(5)} and Cauchy-Schwarz. 

Lastly, the estimates~\eqref{eq:vir-new} and~\eqref{eq:vir-new2}  are straightforward consequences of the definition of $\calA_0(\lam)$ and Lemma~\ref{l:pdef} along with the estimates, 
\EQ{
\abs{P(r) } + r \abs{\p_r P(r)} + r^2 \abs{ \p_r^2 P(r)} \lesssim r^{-2} R^{-2k} \mif r \ge R 
}
We omit the details. 
\end{proof}


%

\subsection{Energy estimates for $\bs w(t)$ in $\HH$} In this section we prove energy-type estimates for $\bs w(t)$ as given by Lemma~\ref{l:mod2}. Recall that the equation for $\bs w(t)$ can be written as 
\EQ{
\p_t w &= \dot w + (\dot \Phi - \p_t \Phi)   \\
\p_t \dot w &= - \LL_\Phi  w + \Big( -  \p_t \dot \Phi  + \De \Phi  - \frac{1}{r^2} f( \Phi) \Big)    - \frac{1}{r^2}  \Big ( f( \Phi +w) - f( \Phi) - f'( \Phi) w \Big)
}
where we also recall the notation 
\EQ{
\LL_{\Phi}  w := -\De w + \frac{1}{r^2} f'( \Phi) w, \quad \LL_\la w = \LL_0 w + P_\lam w
}
where $\LL_0 = -\De + \frac{k^2}{r^2}$ and $P_\lam$ is the potential
\EQ{\label{eq:Plam}
P_\lam(r) = \frac{1}{r^2} ( f'(Q_\lam) - k^2)
}
To ease notation we define,  
\EQ{ \label{eq:Psi1} 
 \Psi_1&:= \dot \Phi - \p_t \Phi  
}
as well as,  
\EQ{ \label{eq:Psi2} 
\Psi_2 &:= -\p_t  \dot \Phi + \De \Phi  - \frac{1}{r^2} f( \Phi)   - \frac{1}{r^2}  \Big ( f( \Phi +w) - f( \Phi) - f'( \Phi) w \Big)  
}
With this notation we rewrite the equation for $\bs w(t)$ as 
\EQ{ \label{eq:weq3} 
 \p_t w  &=  \dot w  + \Psi_1 \\
 \p_t \dot w &= - \LL_\Phi w  + \Psi_2 
}
We define a modified energy functional 
for $\bs w$ as follows: 
\EQ{ \label{eq:E1w} 
\Es_1(t) = \frac{1}{2} \ang{ \dot w \mid \dot w} + \frac{1}{2} \ang{ w \mid \LL_\Phi w} 
}
We also define the virial correction, 
 \EQ{ \label{eq:Vs1} 
  \mathscr{V}_{1, \la}(t):= \ang{ \calA_0(\la) w \mid  \dot w}
 }
 and the mixed energy/virial functional, 
 \EQ{ \label{eq:H1} 
 \Hs_1(t) = \Es_1(t) - b(t) \Vs_{1, \la}(t)  
 }
 These functionals compare as follows, 
 \begin{lem} \label{l:w-E1-H1} Let $J$ be a time interval one which $\bs u(t)$ satisfies the hypothesis Lemma~\ref{l:mod2}. Let $a, b, \la, \mu$ and $\bs w(t)$ be given by Lemma~\ref{l:mod2} and let $\nu:= \la/\mu$ as usual. Then, as long as $\eta_0$ is small enough in Lemma~\ref{l:mod2} we have 
 \EQ{ \label{eq:w-E1-H1} 
\| \bs w \|_{\HH}^2 \simeq  \Es_1(t) \simeq \Hs_1(t)   
 }
 with for a uniform constant, independent of $J$. 
\end{lem} 
\begin{proof}
The first estimate in~\eqref{eq:w-E1-H1} follows from the coercivity estimate in Lemma~\ref{l:coerce1}. The second estimate follows since taking $\eta_0>0$ small in Lemma~\ref{l:mod2} ensures $\abs{b} \ll 1$ and thus, 
\EQ{
 \abs{ b(t) \Vs_{1, \la}(t) } \lesssim \abs{b} \| \bs w \|_{\HH}^2  \ll  \| \bs w \|_{\HH}^2 
}
\end{proof} 

\begin{prop}  \label{p:Hs1'} 
Let $J$ be a time interval on which $\bs u(t)$ satisfies the hypothesis  Lemma~\ref{l:mod2}. Let $a, b, \la, \mu$  and $\bs w(t)$ be given by  Lemma~\ref{l:mod2} and let $\nu:= \la/\mu$ as usual. Let $\Hs_1(t)$ be defined as in~\eqref{eq:H1}. Then, 
\EQ{ \label{eq:Hs1'} 
\Hs_1'(t) &\ge   \frac{b}{\lambda}\int_0^{R\lambda}\Big((\partial_r w)^2 + \frac{k^2}{r^2}w^2\Big) \udr + \frac{b}{\lam} \int_0^\infty (f'(Q_\lam) - k^2) \frac{w^2}{r^2} \, \udr \\
&\quad   -c_0 \frac{\abs{b}}{\la} \| w \|_H^2  
- C_1\Big( \frac{\nu^{2k-1}}{\la} + \frac{b^2 \nu^{k-1}}{\la} +  \frac{b^4}{\la} + \frac{a^4}{\la} \Big) \| \bs w \|_{\HH} \\
& \quad -C_1\Big( \frac{\abs{b}\lam}{\lam} +  \frac{\abs{a}\nu}{\la}  + \frac{\nu^{k}}{\la} + \frac{a^2}{\la}+  \frac{b^2}{\la} \Big) \| \bs w \|_{\HH}^2 -  C_1 \frac{1}{\la} \| \dot w \|_{L^2} \| \bs w \|_{\HH}^2\\
& \quad  -C_1\Big( \frac{ \abs{a} + \abs{b}}{\lam}  \Big) \| \bs w \|_{\HH}^3  -C_1\| \dot w \|_{L^2} \| \bs w \|_{\HH}^2 \| \p_r w \|_H - C_1 \frac{1}{\lam}  \| \dot w \|_{L^2} \| \bs w \|_{\HH}^3   \\
&\quad  -C_1  \abs{b} \| w \|_H^2   \| w \|_{H^2} -C_1  \abs{b} \|w \|_H^3\| w \|_{H^2} 
}
where $c_0>0$ is a constant that can be taken as small as we like by choosing $R>0$ large enough in Lemma~\ref{l:opA}.  Importantly, $c_0$ can be taken small independently of $\mu, \lam, a, b$ and $J$. 
\end{prop}

\begin{lem}  \label{l:E1'}
Let $J \subset \R$ be an interval on which $\bs u(t)$ satisfies the hypothesis of Lemma~\ref{l:mod2}. Let $a, b, \la, \mu$ and $\bs w(t)$ be given by Lemma~\ref{l:mod2} and let $\nu:= \la/\mu$ as usual. Let $\Es_1(t)$ be defined as in~\eqref{eq:E1w}. Then,  
\EQ{
\Big| \Es_1'(t) + & \frac{b}{\lam} \ang{  \La_0 w \mid P_\lam w} - \frac{b}{\lam} \ang{ w \mid P_\lam w} \Big|  \le  C_1\Big( \frac{\nu^{2k-1}}{\la} + \frac{b^2 \nu^{k-1}}{\la} +  \frac{b^4}{\la} + \frac{a^4}{\la} \Big) \| \bs w \|_{\HH} \\
& \quad + C_1 \Big( \frac{\abs{a}\nu}{\la}  + \frac{b \la^{\frac{k}{2}}}{\la} + \frac{\nu^{k}}{\la} + \frac{a^2}{\la}+  \frac{b^2}{\la} \Big) \| \bs w \|_{\HH}^2    + C_1 \frac{1}{\la} \| \dot w \|_{L^2} \| \bs w \|_{\HH}^2\\
& \quad + C_1\Big( \frac{ \abs{a} + \abs{b}}{\lam}  \Big) \| \bs w \|_{\HH}^3 +C_1 \Big( \frac{ \abs{a} + \abs{b}}{\lam}  \Big) \| \bs w \|_{\HH}^4+ C_1\| \dot w \|_{L^2} \| \bs w \|_{\HH}^2 \| \p_r w \|_H  
}

for some uniform constant $C_1>0$. 
\end{lem} 

\begin{lem}  \label{l:Vs1'} 
Let $J \subset \R$ be an interval on which $\bs u(t)$ satisfies the hypothesis of Lemma~\ref{l:mod2}. Let $a, b, \la, \mu$ and $\bs w(t)$ be given by Lemma~\ref{l:mod2} and let $\nu= \la/\mu$.  Let $\Vs_{1, \la}$ be as in~\eqref{eq:Vs1}. Then, as long as $a, b, \nu \ll1$ on $J$ we have, 
\EQ{
 (b &\Vs_{1, \la})'(t)   + \frac{ b}{\lam}   \ang{ \Lam_0 w \mid P_\lam w}   
 \le c_0 \frac{b}{\la} \| w \|_H^2  -  \frac{1}{\lambda}\int_0^{R\lambda}\Big((\partial_r w)^2 + \frac{k^2}{r^2}w^2\Big) \udr  \\
&  + C_1 \Big( \frac{\abs{b} \nu^{2k-1}}{\la}  + \frac{\abs{b}^5}{\la}  + \frac{\abs{a}^5}{\la} \Big) \| \dot w \|_{L^2}+ C_1 \Big(  \frac{ \abs{b} \lam + a^2 + b^2 + \nu^{k} }{\lam}  \Big)   \| \bs w \|_{\HH}^2   \\
&   + C_1 \Big( \frac{ \abs{a} + \abs{b}}{\lam} \| \dot w \|_{L^2} \| \bs w \|_{L^2}^2  + \frac{1}{\lam}  \| \dot w \|_{L^2} \| \bs w \|_{\HH}^3+  \abs{b} \| w \|_H^2   \| w \|_{H^2} +  \abs{b} \|w \|_H^3\| w \|_{H^2}\Big)
}
for some uniform constant $C_1>0$,
and where $c_0>0, R>0$ are given by Lemma~\ref{l:opA}. We note that $c_0>0$ can be taken arbitrarily small independently of $a, b, \la, \mu$. 
\end{lem}



\begin{proof}[Proof of Proposition~\ref{p:Hs1'} assuming Lemmas~\ref{l:E1'},~\ref{l:Vs1'}]
Using Lemma~\ref{l:E1'} and Lemma~\ref{l:Vs1'} we have 
\EQ{
&\Hs_1'(t)   = \Es_1'(t) - (b\Vs_{1, \la})'(t)  \\
&  \ge -c_0 \frac{b}{\la} \| w \|_H^2  +  \frac{b}{\lambda}\int_0^{R\lambda}\Big((\partial_r w)^2 + \frac{k^2}{r^2}w^2\Big) \udr + \frac{b}{\lam} \int_0^\infty (f'(Q_\lam) - k^2) \frac{w^2}{r^2} \, \udr\\
&\quad - C_1\Big( \frac{\nu^{2k-1}}{\la} + \frac{b^2 \nu^{k-1}}{\la} +  \frac{b^4}{\la} + \frac{a^4}{\la} \Big) \| \bs w \|_{\HH} -C_1\Big( \frac{\abs{b}\lam}{\lam} +  \frac{\abs{a}\nu}{\la}  + \frac{\nu^{k}}{\la} + \frac{a^2}{\la}+  \frac{b^2}{\la} \Big) \| \bs w \|_{\HH}^2 \\
& \quad -  C_1 \frac{1}{\la} \| \dot w \|_{L^2} \| \bs w \|_{\HH}^2  -C_1\Big( \frac{ \abs{a} + \abs{b}}{\lam}  \Big) \| \bs w \|_{\HH}^3  -C_1\| \dot w \|_{L^2} \| \bs w \|_{\HH}^2 \| \p_r w \|_H - C_1 \frac{1}{\lam}  \| \dot w \|_{L^2} \| \bs w \|_{\HH}^3  \\
&\quad -C_1  \abs{b} \| w \|_H^2   \| w \|_{H^2} -C_1  \abs{b} \|w \|_H^3\| w \|_{H^2} 
}
which completes the proof. 
\end{proof}

\begin{proof}[Proof of Lemma~\ref{l:E1'}] 
First we note the identity, 
\EQ{ \label{eq:E1'1} 
\Es_1'(t) &=\ang{ \Psi_2 \mid  \dot w } + \ang{ \Psi_1 \mid  \calL_\Phi w}   
 +\frac{1}{2} \ang{ w \mid [\p_t, \LL_\Phi] w}
}
To see this, we compute using~\eqref{eq:weq3}, 
\EQ{
\Es_1'(t) &= \ang{ \p_t  \dot w \mid\dot  w} + \ang{ \p_t w \mid \LL_\Phi w} + \frac{1}{2} \ang{ w \mid [\p_t, \LL_\Phi] w}  \\
& = \ang{ -\calL_\Phi w \mid \dot w} + \ang{ \Psi_2 \mid  \dot w} 
 + \ang{ \dot w \mid \calL_\Phi w} + \ang{ \Psi_1 \mid  \calL_\Phi w} 
 +\frac{1}{2} \ang{ w \mid [\p_t, \LL_\Phi] w} \\
& = \ang{ \Psi_2 \mid  \dot w } + \ang{ \Psi_1 \mid  \calL_\Phi w} +\frac{1}{2} \ang{ w \mid [\p_t, \LL_\Phi] w}  \\
}
which is precisely~\eqref{eq:E1'1}. 

We estimate each of the terms on the right-hand side of~\eqref{eq:E1'1}. 
\begin{claim} We have,   
\EQ{ \label{c:Psi2wdot} 
\abs{ \ang{ \Psi_2 \mid  \dot w } } &\lesssim   \Big( \frac{ \nu^{2k-1}}{\la} + \frac{b^2 \nu^{k-1}}{\la}  + \frac{b^4}{\la} + \frac{a^4}{\la}       \Big)  \| \dot w \|_{L^2} + \frac{\abs{a}^3 + \abs{b}^3 + ( \abs{a} + \abs{b}) \nu^{k-1}}{\la} \| \dot w \|_{L^2}   \|  \bs w \|_{\HH}  \\ 
&\quad + \frac{1}{\la} \| \dot w \|_{L^2} \| \bs w \|_{\HH}^2  + \| \dot w \|_{L^2} \| \bs w \|_{\HH}^2 \| \p_r w \|_H  
  }
\end{claim}
\begin{proof}[Proof of Claim~\ref{c:Psi2wdot}] 
First, note that by the definition of ~$\Psi_2$ in~\eqref{eq:Psi2} we have 
\EQ{
\ang{ \Psi_2 \mid  \dot w } & = \ang{  -\p_t \dot \Phi + \De \Phi - \frac{1}{r^2} f( \Phi)  \mid \dot w }  - \ang{  \frac{1}{r^2}  \Big ( f( \Phi +w) - f( \Phi) - f'( \Phi) w \Big) \mid \dot w} 
}
The claim then follows from two estimates: 
 \EQ{ \label{eq:ptdotPhiest} 
\Big|\langle  - \p_t \dot \Phi + &\De \Phi - \frac{1}{r^2} f( \Phi)  \mid  \dot w \rangle \Big|    \lesssim \| \dot w \|_{L^2} \Big(\frac{ \nu^{2k-1}}{\la} + \frac{b^2 \nu^{k-1}}{\la}  + \frac{b^4}{\la} + \frac{a^4\nu}{\la} \Big)   \\ 
& \,  + \frac{\abs{a}^3 + \abs{b}^3 +( \abs{a} + \abs{b})\nu^{k-1}}{\la} \| \dot w \|_{L^2}   \|  \bs w \|_{\HH} +    \frac{a^2 + b^2 + \nu^{k-1}}{\la}   \| \dot w \|_{L^2} \| \bs w \|_{\HH}^2  
 }
 and 
 \EQ{ \label{eq:w^2-est1} 
 \abs{ \ang{  \frac{1}{r^2}  \Big ( f( \Phi +w) - f( \Phi) - f'( \Phi) w \Big) \mid \dot w} } &\lesssim  \Big(\frac{ \|  w \|_{H}^2 }{\la}  +  \| w \|_H^2 \| \p_r w \|_H \Big) \| \dot w \|_{L^2}
}
Note that~\eqref{eq:w^2-est1} follows from~\eqref{eq:w^2L2} from Lemma~\ref{l:w^2est}. It remains the prove~\eqref{eq:ptdotPhiest}. Using crucially the orthogonality conditions~\eqref{eq:lawdot} and \eqref{eq:muwdot} we see that  
 \EQ{
&\langle  -\p_t \dot \Phi  + \De \Phi - \frac{1}{r^2} f( \Phi)  \mid \dot w \rangle  = I + II + III + IV 
}
where 
\EQ{
I = b \frac{(b+\la')}{\la} \ang{ \La_0 \La Q_{\U \la} \mid  \dot w}   +   a\frac{(\mu'- a)}{\mu} \ang{ \La_0 \La Q_{\U \mu}\mid \dot w}  
} 
and satisfies the estimates, 
\EQ{
\abs{I} \lesssim \frac{\abs{b}}{\lam} \abs{ b+ \lam'} \| \dot w \|_{L^2}  +  \frac{\abs{a} \nu}{\lam} \abs{ a - \mu'}   \| \dot w \|_{L^2} 
}
The second term $II$ is given by
\begingroup
\allowdisplaybreaks
\begin{align} 
&II  =  (b'+ \gamma_k \frac{\nu^k}{\la})\Bigg( - 3   b^2 \ang{\La A_{\U \la} \mid  \dot w}  + 2 \gamma_k\nu^k  \ang{ A_{\U \la}  \mid  \dot w}  -  \nu^k \ang{ \La B_{\U \la} \mid  \dot w}+  k \nu^k \ang{ B_{\U \la}  \mid  \dot w} \\
 &\qquad \qquad \qquad  \qquad - k \nu^{k-1} \ang{\ti B_{\U \mu} \mid  \dot w}\Bigg) \\
& \quad  + (a'  + \ti \gamma_k \frac{\nu^k}{\mu})\Bigg(    -3a^2 \ang{ \La A_{\U \mu}  \mid  \dot w}  -  2 \ti \gamma_k  \nu^k \ang{ A_{\U \mu}  \mid  \dot w}+  \nu^k  \ang{ \La  \ti B_{\U \mu}  \mid  \dot w}+  k \nu^{k+1} \ang{B_{\U \lam}  \mid  \dot w}   \\
&\qquad\qquad\qquad \qquad - k \nu^k \ang{\ti B_{\U \mu}  \mid  \dot w} \Bigg)  \\ 
& \quad + (b+\la') \Bigg(    \frac{b^3}{\la}\ang{  \La_0 \La A_{\U \la} \mid  \dot w} - 2 \gamma_k \frac{ b \nu^k}{\la}\ang{  \La_0 A_{\U \la} \mid  \dot w} - k   \frac{ b \nu^{k}}{\la}\ang{ \La_0 B_{ \U \la} \mid  \dot w} \\
& \qquad \qquad  \quad  + 2k \gamma_k \frac{ b \nu^{k}}{\la} \ang{  A_{\U \la} \mid  \dot w}- k  \frac{ b\nu^k}{\lam}\ang{\La B_{\U\la} \mid  \dot w}+   k^2\frac{b\nu^k}{\lam} \ang{ B_{ \U \la}  \mid  \dot w}- 2k \ga_k  \frac{a\nu^k}{\lam}  \ang{ A_{\U \mu}  \mid  \dot w} \\
& \qquad \qquad  \quad - k  \frac{a\nu^k}{\lam} \ang{ \La \ti B_{\U \mu}   \mid  \dot w}  - k(k-1) \frac{ b \nu^{k-1}}{\lam}\ang{ \ti B_{\U \mu} \mid  \dot w}  +  k(k+1) a \frac{\nu^{k+1}}{\lam} \ang{ B_{\U \lam}  \mid  \dot w} \\
& \qquad \qquad  \quad- k a \frac{\nu^{k+1}}{\lam}  \ang{ \La_0 B_{\U \lam} \mid  \dot w}  -  k^2 a \frac{\nu^k}{\lam}\ang{ \ti B_{\U \mu} \mid  \dot w} +  \frac{b \nu^k}{\la}\ang{ \La_0 \La B_{\U \la}   \mid  \dot w} \Bigg)  \\ 
 & \quad + (a-\mu') \Bigg(  - \frac{ a^3 \nu}{\lam}  \ang{ \La_0 \La A_{\U \mu}  \mid  \dot w}- 2 \gamma_k \frac{a \nu^{k+1}}{\lam} \ang{  \La_0 A_{\U \mu} \mid  \dot w} - \frac{a \nu^{k+1}}{\lam}  \ang{ \La_0 \La \ti B_{\U \mu}  \mid  \dot w}\\
& \qquad \qquad  \quad  - k  \frac{ b \nu^{k}}{\lam} \ang{ \La_0  \ti B_{\U \mu}  \mid  \dot w}   + 2k \gamma_k  \frac{b\nu^{k+1}}{\lam} \ang{   A_{\U \la}  \mid  \dot w} - k  \frac{b\nu^{k+1}}{\lam} \ang{ \La B_{\U\la}  \mid  \dot w} \\
& \qquad \qquad  \quad +  k^2 \frac{b \nu^{k+1}}{\la}  \ang{ B_{ \U \la}  \mid  \dot w}  - 2k \ga_k \frac{a \nu^{k+1}}{\lam} \ang{ A_{\U \mu}   \mid  \dot w}- k \frac{ a \nu^{k+1}}{\lam} \ang{ \La \ti B_{\U \mu} \mid  \dot w} \\
& \qquad \qquad  \quad    - k(k-1) \frac{  b  \nu^{k}}{\lam}\ang{  \ti B_{\U \mu} \mid  \dot w}  +   k(k+1) a \frac{\nu^{k+2}}{\lam}\ang{ B_{\U \lam} \mid  \dot w}  -  k^2 a \frac{\nu^{k+1}}{\lam}\ang{ \ti B_{\U \mu} \mid  \dot w}  \\
& \qquad \qquad  \quad  - k a \frac{\nu^{k+1}}{\lam} \ang{ \Lam_0 \ti B_{\U \mu}  \mid  \dot w}\Bigg)  
\end{align} 
\endgroup
and satisfies, 
\EQ{
\abs{II} &\lesssim \abs{ b'+ \gamma_k \frac{\nu^k}{\la}}( b^2 + \nu^{k-1}) \| \dot w \|_{L^2} + \abs{a'  + \ti \gamma_k \frac{\nu^k}{\mu}} (a^2 + \nu^k) \| \dot w \|_{L^2} \\
& \quad  + (\abs{b+ \la'}+ \abs{ a - \mu'})( \frac{\abs{b}^3}{\lam} + \frac{\abs{a}^3}{\lam} + \frac{\nu^{\frac{3k}{2}}}{\lam} + \frac{\abs{b} \nu^{k-1}}{\lam}) \| \dot w \|_{L^2} 
}

The third term $III$ is given by, 
\EQ{
III &=    3\gamma_k \frac{\nu^k}{\la}b^2 \ang{  \La A_{\U \la}   \mid  \dot w} -  \frac{b^4}{\la} \ang{ \La_0 \La A_{\U \la}   \mid  \dot w}   -2 \gamma_k^2  \frac{\nu^{2k}}{\la} \ang{  A_{\U \la} \mid  \dot w} + 2 \gamma_k\frac{ b^2 \nu^k}{\la} \ang{  \La_0 A_{\U \la}  \mid  \dot w}   \\
&\quad  +   \gamma_k  \frac{\nu^{2k}}{\lam} \ang{ \La B_{\U \la} \mid  \dot w}  - 2k \gamma_k  \frac{b^2 \nu^k}{\lam} \ang{  A_{\U \la}  \mid  \dot w} - 2k \gamma_k  \frac{a b\nu^{k+1}}{\lam} \ang{  A_{\U \la}  \mid  \dot w} + k  \frac{b^2 \nu^k}{\lam}\ang{  \La B_{\U\la}   \mid  \dot w} \\
&\quad  + k  \frac{ab \nu^{k+1}}{\lam} \ang{  \La B_{\U\la} \mid  \dot w}  - \frac{b^2 \nu^k}{\la}\ang{  \La_0 \La B_{\U \la}  \mid  \dot w}  -  k \frac{\nu^{2k}}{\lam} \ang{ B_{\U \la}  \mid  \dot w} +  \frac{ k b^2 \nu^{k}}{\la}\ang{ \La_0 B_{ \U \la} \mid  \dot w}   \\
& \quad  -  k^2 \frac{b^2 \nu^{k} }{\lam} \ang{ B_{ \U \la}   \mid  \dot w}  -   k^2 \frac{a b  \nu^{k+1}}{\lam} \ang{ B_{ \U \la} \mid \dot w}  - k(k+1) ab \frac{\nu^{k+1}}{\lam}\ang{  B_{\U \lam}  \mid  \dot w} \\ 
 &\quad   -  k(k+1) a^2 \frac{\nu^{k+2}}{\lam} \ang{B_{\U \lam}  \mid  \dot w} -  k  \ti \gamma_k \frac{\nu^{2k+2}}{\lam}\ang{  B_{\U \lam}  \mid  \dot w}  + 3 \ti \gamma_k \frac{a^2\nu^k}{\mu}\ang{  \La A_{\U \mu}   \mid  \dot w} \\
 &\quad  + \frac{a^4\nu}{\lam}\ang{  \La_0 \La A_{\U \mu}   \mid  \dot w} +   2  \ti \gamma_k^2 \frac{\nu^{2k+1}}{\la} \ang{ A_{\U \mu}  \mid  \dot w}  + 2k \ga_k \frac{ a b \nu^k}{\lam}\ang{  A_{\U \mu}  \mid  \dot w} \\
 &\quad   +  2k \ga_k\frac{a^2 \nu^{k+1}}{\lam} \ang{  A_{\U \mu} \mid  \dot w}+ 2 \gamma_k \frac{a^2 \nu^{k+1}}{\lam}  \ang{ \La_0 A_{\U \mu}   \mid  \dot w} -  \ti \gamma_k  \frac{\nu^{2k+1}}{\lam} \ang{ \La  \ti B_{\U \mu} \mid  \dot w} \\
 &\quad   +  k  \frac{ ab \nu^k}{\lam} \ang{  \La \ti B_{\U \mu}  \mid  \dot w}+  k \frac{a^2 \nu^{k+1}}{\lam} \ang{  \La \ti B_{\U \mu} \mid  \dot w}  +  \frac{ a^2 \nu^{k+1}}{\lam}\ang{   \La_0 \La \ti B_{\U \mu}  \mid  \dot w}   \\
& \quad  + k \gamma_k \frac{\nu^{2k-1}}{\lam} \ang{ \ti B_{\U \mu} \mid  \dot w} +  k \ti \gamma_k  \frac{\nu^{2k+1}}{\lam} \ang{  \ti B_{\U \mu}  \mid  \dot w}  + k^2 a b \frac{\nu^k}{\lam}\ang{  \ti B_{\U \mu}  \mid  \dot w}  \\
& \quad  + k(k-1) \frac{b^2 \nu^{k-1}}{\lam} \ang{  \ti B_{\U \mu}   \mid  \dot w}+ k(k-1) \frac{ab  \nu^{k}}{\lam} \ang{  \ti B_{\U \mu}  \mid  \dot w} + k \frac{ a b \nu^{k}}{\lam}\ang{   \La_0  \ti B_{\U \mu}  \mid  \dot w}  \\
&\quad + k^2 a^2 \frac{\nu^{k+1}}{ \lam} \ang{ \ti B_{\U \mu}  \mid  \dot w}+ k a^2 \frac{ \nu^{k+1}}{\lam}  \ang{ \Lam_0 \ti B_{\U \mu} \mid  \dot w}+ k a b\frac{\nu^{k+1}}{\lam} \ang{  \La_0 B_{\U \lam}  \mid  \dot w}
}
and satisfies, 
\EQ{
\abs{ III} & \lesssim   \frac{b^4 + a^4 + b^2 \nu^{k-1} + \nu^{2k-1} }{\lam}   \| \dot w \|_{L^2} 
}
Finally, 
\EQ{
 IV= &-\ang{ \frac{1}{r^2} \Big( f(Q_\la - Q_\mu)  - f(Q_\la) + f(Q_\mu)  - 4  \big(\frac{r}{ \mu}\big)^k (\La Q_{\la})^2  -  4 \big(\frac{r}{ \la}\big)^{-k} (\La Q_{\mu})^2 \Big) \mid \dot w} \\
&-\ang{\frac{1}{r^2} \Big( f(\Phi) - f(Q_\la - Q_\mu) -  f'(Q_\la- Q_\mu) (b^2 T_\la - a^2 \ti T_{\mu} \Big)  \mid \dot w}  \\
 &-\ang{ \frac{1}{r^2}  \Big(f'(Q_\la - Q_\mu) (b^2 T_\la - a^2 \ti T_\mu)  - f'(Q_\la) b^2 T_\la + f'(Q_\mu) a^2 \ti T_\mu\Big) \mid \dot w} 
 }
 and satisfies, 
 \EQ{
\abs{ IV} \lesssim   \frac{ b^4 + a^4 + \nu^{2k-1} + b^2 \nu^{k-1}}{\lam} \| \dot w \|_{L^2} 
 }
 where we have used Lemma~\ref{l:fest1} above. Combining these estimates we have 
\EQ{
\Big| \langle  - \p_t \dot \Phi +& \De \Phi - \frac{1}{r^2} f( \Phi)  \mid  \dot w \rangle \Big|   \lesssim  \| \dot w \|_{L^2} \Bigg(\frac{ \nu^{2k-1}}{\la} + \frac{b^2 \nu^{k-1}}{\la}  + \frac{b^4}{\la} + \frac{a^4}{\la} + \abs{a'  + \ti \gamma_k \frac{\nu^k}{\mu}}( a^2 + \nu^k) \\
 & \quad   +  \abs{b'+ \gamma_k \frac{\nu^k}{\la}}( b^2 + \nu^{k-1})  + \Big(\frac{\abs{b}}{\la}  + \frac{\abs{a}}{\lam}  + \frac{\nu^{\frac{3}{2}k}}{\lam} \Big)(\abs{ b + \la'}  +\abs{a - \mu'}   \Big) \Bigg)
 }
Using Lemma~\ref{l:modc2} on the right above we obtain, 
\EQ{
\Big|\langle  - \p_t \dot \Phi + &\De \Phi - \frac{1}{r^2} f( \Phi)  \mid  \dot w \rangle \Big|    \lesssim \| \dot w \|_{L^2} \Big(\frac{ \nu^{2k-1}}{\la} + \frac{b^2 \nu^{k-1}}{\la}  + \frac{b^4}{\la} + \frac{a^4\nu}{\la} \Big)   \\ 
& \,  + \frac{\abs{a}^3 + \abs{b}^3 +( \abs{a} + \abs{b})\nu^{k-1}}{\la} \| \dot w \|_{L^2}   \|  \bs w \|_{\HH} +    \frac{a^2 + b^2 + \nu^{k-1}}{\la}   \| \dot w \|_{L^2} \| \bs w \|_{\HH}^2  
}
which is precisely~\eqref{eq:ptdotPhiest}. 
\end{proof} 

\begin{claim} \label{c:Psi1Lw} We have,  
\EQ{
\abs{ \ang{ \Psi_1 \mid  \calL_\Phi w} } &\lesssim  \Big(  \frac{\abs{a}^5}{\lam} + \frac{\abs{b}^5}{\lam} + \frac{(\abs{a} + \abs{b}) \nu^{2k-1}}{\lam}  + \frac{(\abs{a} + \abs{b})b^2 \nu^{k-1}}{\lam} \Big) \| \bs w \|_{\HH} \\
& \quad  + \Big( \frac{ a^2 + b^2 + \nu^{2k}}{\lam} \Big) \| \bs w \|_{\HH}^2  + \Big( \frac{ \abs{a} + \abs{b}}{\lam} \Big) \| \bs w \|_{\HH}^3 
}
\end{claim}
\begin{proof}[Proof of Claim~\ref{c:Psi1Lw}]  
Using~\eqref{eq:Psi1} 
we have 
\EQ{ \label{eq:Psi1w} 
 \ang{ \Psi_1 \mid  \calL_\Phi w} &=   (b+ \la') \ang{ \La Q_{\U \la} \mid \LL_\Phi w}+ (a - \mu')\ang{ \La Q_{\U \mu} \mid  \LL_\Phi w} \\
 & \quad  +  b^2( b + \la') \ang{\La A_{\U \la}\mid  \LL_\Phi w}    - 2 b \la ( b' + \gamma_k \frac{\nu^{k}}{\la}) \ang{A_{\U \la}  \mid  \LL_\Phi w}   \\
 & \quad +  \nu^k( b + \la') \ang{\La B_{\U \la}\mid  \LL_\Phi w} - k \nu^k( b +\la')  \ang{B_{ \U\la} \mid  \LL_\Phi w}  \\
 &\quad  - k \nu^{k+1}( a - \mu')\ang{  B_{ \U\la} \mid \LL_\Phi w}+ a^2(  a - \mu')\ang{ \La A_{\U \mu}\mid  \LL_\Phi w}  \\
 &\quad + 2a \mu ( a' +  \ga_k \frac{\nu^k}{\mu}) \ang{A_{\U\mu}  \mid  \LL_\Phi w} + \nu^k(   a- \mu') \ang{\La \ti B_{\U \mu}\mid  \LL_\Phi w} \\
   &\quad + k \nu^{k-1} (  b + \la') \ang{\ti B_{\U \mu} \mid  \LL_\Phi w}+ k \nu^{k} (  a- \mu') \ang{\ti B_{\U \mu}   \mid \LL_{\Phi} w} 
}
To estimate the first two terms on the right-hand side above we exploit the fact that $\LL_\la \La Q_{\U \la} =0$ and $\LL_\mu \La Q_{\U \mu} = 0$. To handle the first term,  we  write 
\EQ{
\LL_{\Phi} h  = (\LL_\Phi - \LL_\la) h+  \LL_\la h  = \frac{k^2}{r^2} ( \cos 2 \Phi - \cos 2 Q_\la) h + \LL_\la h 
}
Then we have, 
\EQ{ \label{eq:Psi1w-t1} 
(b+ \la') \ang{ \La Q_{\U \la} \mid \LL_\Phi w}  &= (b+ \la') \ang{ \La Q_{\U \la} \mid \LL_\la w} + (b+ \la') \ang{ \La Q_{\U \la} \mid (\LL_{\Phi} - \LL_\la) w} \\
& = (b+ \la') \ang{ \La Q_{\U \la} \mid\frac{k^2}{r^2} ( \cos 2 \Phi  - \cos 2 Q_\la) w}
}
We now estimate, 
\EQ{
\abs{ \ang{ \La Q_{\U \la} \mid\frac{k^2}{r^2} ( \cos 2 \Phi  - \cos 2 Q_\la) w}} &\lesssim \frac{1}{\la} \| w\|_{H} \left( \int_0^\infty \La Q_{\la}^2 \big(\cos 2 \Phi  - \cos 2 Q_\la\big)^2 \, \frac{\ud r}{r} \right)^{\frac{1}{2}} \\
& \lesssim  \frac{ \nu^k + b^2 + a^2}{\la}  \| w \|_H
}
where the last line follows from an explicit computation using~\eqref{eq:cos2Phi1} from Lemma~\ref{l:cos2Phi}. 
Plugging the above back into~\eqref{eq:Psi1w-t1} we obtain, 
\EQ{
\abs{(b+ \la') \ang{ \La Q_{\U \la} \mid \LL_\Phi w} }  \lesssim \frac{ \nu^k + b^2 + a^2}{\la} \abs{b+ \la'} \| w \|_H
}
An identical analysis, this time using~\eqref{eq:cos2Phi2} from Lemma~\ref{l:cos2Phi} yields, 
\EQ{
\abs{(a - \mu')\ang{ \La Q_{\U \mu} \mid  \LL_\Phi w}} \lesssim \frac{ \nu^k + b^2 + a^2}{\mu} \abs{ a - \mu'} \| w \|_H 
}
For the third term on the right-hand side of~\eqref{eq:Psi1w} we integrate by parts to obtain, 
\EQ{
\abs{ b^2( b + \la') \ang{\La A_{\U \la}\mid  \calL_\Phi w} }  &\lesssim \abs{ \frac{b^2}{\la} ( b + \la')} \Big( \abs{ \ang{ \p_r ( \La A_{ \la}) \mid  \p_r w} }  + \abs{ \ang{  \La A_{ \la} \mid \frac{k^2}{r^2} \cos 2 \Phi w} }   \Big) \\
& \lesssim \frac{b^2}{\la}  \abs{ b + \la'} \| \La A \|_{H} \|  w \|_{H}  \\
& \lesssim \frac{b^2}{\la}  \abs{ b + \la'} \| w \|_H
}
In a similar vein for the fourth term we have 
\EQ{
\abs{b \la ( b' + \gamma_k \frac{\nu^{k}}{\la}) \ang{A_{\U \la}  \mid  \LL_\Phi w}}  &\lesssim b \abs{ b' + \gamma_k \frac{\nu^{k}}{\la}}  \| A \|_H \| w \|_H \\
& \lesssim b \abs{ b' + \gamma_k \frac{\nu^{k}}{\la}}  \| A \|_H \| w \|_H 
}
Estimating the remaining terms in an essentially identical fashion yields, 
\EQ{
\abs{ \ang{ \Psi_1 \mid  \calL_\Phi w} } &\lesssim \frac{ \nu^k+ b^2 + a^2}{\la} \Big( \abs{ b + \la'}   + \nu \abs{ a- \mu'} \Big) \| w \|_H \\
& \quad +  b \abs{b'+ \gamma_k   \frac{\nu^{k}}{\la} }  \| w \|_{H} + a\abs{a'  +   \ga_k \frac{\nu^k}{\mu}} \|w \|_H 
}
Next, applying Lemma~\ref{l:modc2} to the right-hand side above gives, 
\EQ{
\abs{ \ang{ \Psi_1 \mid  \calL_\Phi w} } &\lesssim \frac{ \nu^k+ b^2 + a^2}{\la} \Big( (\abs{a} + \abs{b}) \| \bs w \|_{\HH} + ( \abs{a} + \abs{b})  \nu^{2k-1} + ( \abs{a}^5 + \abs{b}^5) \nu^{k-2} \Big)  \| \bs w \|_{\HH}  \\
& \quad + \frac{ \abs{a} + \abs{b}}{\lam} \Big( (\abs{a} + \abs{b})  \| \dot w \|_{L^2} +  \| \bs w \|_{\HH}^2 + \nu^{2k-1} + b^2 \nu^{k-1} + a^4 + b^4 \Big) \|  \bs w \|_{\HH}  \\
& \lesssim \| \bs w \|_{\HH}  \Big(  \frac{\abs{a}^5}{\lam} + \frac{\abs{b}^5}{\lam} + \frac{(\abs{a} + \abs{b}) \nu^{2k-1}}{\lam}  + \frac{(\abs{a} + \abs{b})b^2 \nu^{k-1}}{\lam} \Big) \\
& \quad  + \Big( \frac{ a^2 + b^2 + \nu^{2k}}{\lam} \Big) \| \bs w \|_{\HH}^2  + \Big( \frac{ \abs{a} + \abs{b}}{\lam} \Big) \| \bs w \|_{\HH}^3 
}
as claimed. 
\end{proof} 


Finally we treat the last term on the right-hand side of~\eqref{eq:E1'1}. 
\begin{claim}  \label{c:vir} 
Under the assumptions that $a, b,  \la, \nu \ll 1$ and $\mu \simeq 1$ we have 
\EQ{
\Big|\frac{1}{2} &\ang{ w \mid [\p_t, \LL_\Phi] w} +  b\ang{  \frac{1}{\la} \La_0 w \mid  P_\la w}  - \frac{b}{\la} \ang{ w \mid P_\la w} \Big|    \\
& \, \, \lesssim \Big( \frac{\abs{b} \la^{\frac{k}{2}} + \abs{a}\nu + b^2 + \nu^k}{\la} \Big) \| w \|_{H}^2  + \Big( \frac{ \abs{a} + \abs{b}}{\lam}  \Big) \| \bs w \|_{\HH}^3 + \Big( \frac{ \abs{a} + \abs{b}}{\lam}  \Big) \| \bs w \|_{\HH}^4
}
where we recall the notation $ 
P_\lam(r) := \frac{1}{r^2}( f'(Q_\lam(r)) - k^2) $.
\end{claim} 
\begin{proof} 
First we claim the estimate, 
\EQ{ \label{eq:comm-est1} 
\abs{\frac{1}{2} \ang{ w \mid [\p_t, \LL_\Phi] w} -  \frac{1}{2} \ang{ w \mid [\p_t, \LL_\la]  w}}  &\lesssim \Big( \frac{\abs{b} \la^{\frac{k}{2}} + \abs{a}\nu + b^2 + \nu^k}{\la} \Big) \| w \|_{H}^2  + \Big( \frac{ \abs{a} + \abs{b}}{\lam}  \Big) \| \bs w \|_{\HH}^3  \\
& \quad + \Big( \frac{ \abs{a} + \abs{b}}{\lam}  \Big) \| \bs w \|_{\HH}^4
}
Postponing the proof of~\eqref{eq:comm-est1} we complete the proof of Claim~\ref{c:vir}. 
Noting the formula, 
$\LL_\lam  = -\De + \frac{k^2}{r^2} + P_\lam$, 
it follows that  
\EQ{
\frac{1}{2} \ang{  w \mid [\p_t, \LL_\la]  w} = \frac{1}{2} \ang{  w \mid [\p_t, P_\lam]  w} 
}
Noting the the identity, 
\EQ{
\p_t P_\lam = -\frac{\la'}{\lam} \frac{1}{r^2}f''(Q_\lam) \Lam Q_\lam  = \frac{-\la'}{\lam} \Lam P_\lam + 2 \frac{-\la'}{\la} P_\lam
}
we have, 
\EQ{
\frac{1}{2} \ang{  w \mid [\p_t, \LL_\la]  w} &= \frac{1}{2} \ang{ w \mid \p_t P_\lam   w}  =  \frac{-\la'}{\la} \frac{1}{2} \ang{  w \mid \Lam P_\lam w} + \frac{-\la'}{\la} \ang{  w \mid P_\lam w} 
}
Using the identity~\eqref{eq:ibpLa} we arrive at 
\EQ{
\frac{1}{2} \ang{  w \mid [\p_t, \LL_\la]  w} & =  - \frac{-\la'}{\la} \ang{ \La  w \mid P_\la w}  \\
& = - \frac{b}{\la} \ang{ \La  w \mid P_\la  w} + \frac{b+ \la'}{\la} \ang{ \La w \mid P_\la  w}  \\
&= - \frac{b}{\la} \ang{ \La_0 w \mid P_\la  w}  + \frac{b}{\la} \ang{   w \mid P_\la  w}+ \frac{b+ \la'}{\la} \ang{ \La w \mid P_\la w}
}
For the last term on the right above we have the estimate, 
\EQ{
\abs{\frac{b+ \la'}{\la} \ang{ \La w \mid P_\la w}} & \lesssim \abs{\frac{b+ \la'}{\la}} \| w \|_H^2  \\
& \lesssim  \frac{1}{\la}\Big((\abs{a} + \abs{b} ) \| \bs w \|_{\HH} +( \abs{a} + \abs{b})\nu^{2k-1}  + ( \abs{a}^5 + \abs{b}^5) \nu^{k-2} \Big) \| w \|_H^2
}
Combing the above with~\eqref{eq:comm-est1} completes the proof. 

It remains to prove~\eqref{eq:comm-est1}.
First,  \EQ{
 [\p_t, \LL_\Phi] \dot w -  [\p_t, \LL_\la] \dot w &=  \frac{1}{r^2} f''(\Phi) \p_t \Phi - \frac{1}{r^2} f''(Q_\lam) \p_t Q_\lam   \\
 & = - \frac{\la'}{\lam}\frac{1}{r^2}( f''(\Phi) - f''(Q_\lam))  \Lam Q_\lam  + \frac{1}{r^2} ( \p_t \Phi - \p_t Q_\lam) f''( \Phi) 
}
First, noting that $f''(\rho) = -2k^2 \sin 2\rho$ 
we write  
\EQ{
\sin 2 \Phi - \sin 2 Q_\la = -  2\sin 2 Q_\la  \sin^2( \Phi - Q_\la) +  \cos 2Q_\la  \sin 2( \Phi - Q_\la)
}
We establish the estimates corresponding to each of the two terms on the right above. For the first term we will divide the integral into two regions $r \lesssim \sqrt{ \la}$ and $r \gtrsim \sqrt{\la}$. From~\eqref{eq:sin(Phi-Qla)} and the above we have the pointwise estimates 
\EQ{ \label{eq:f''Phi-f''Qla} 
\abs{ \chi_{r \le \sqrt{\la}}  \La Q_\la (\sin 2 \Phi - \sin 2 Q_\la )} &\lesssim r^k + b^2 +  \nu^k  \lesssim \la^{\frac{k}{2}} + b^2 +  \nu^k \\
\abs{ \chi_{r \ge \sqrt{\la}}  \La Q_\la (\sin 2 \Phi - \sin 2 Q_\la )} &\lesssim  \la^{\frac{k}{2}}
}
Using these estimates we see that 
\EQ{
\abs{ \frac{\la'}{\la}  \ang{ w \mid  \frac{k^2}{r^2} (\sin 2 \Phi - \sin 2 Q_\la) \La Q_{\la} w}} & \lesssim \frac{b}{\la}  \left( \la^{\frac{k}{2}} + b^2 +  \nu^k\right) \| w \|^2_H  \\
& \quad  + \frac{\abs{b+ \la'}}{\la}\left( \la^{\frac{k}{2}} + b^2 +  \nu^k\right) \| w \|^2_H
}
For the second term we note the formula 
\EQ{
-\p_t \Phi + \p_t Q_\lam  &=   b^2( b + \la') \La A_{\U \la}  - 2 b \la ( b' + \gamma_k \frac{\nu^{k}}{\la}) A_{\U \la}  +  \nu^k( b + \la') \La B_{\U \la}  - k \nu^k( b + \mu \nu')  B_{ \U\la}  \\
 & \quad +(a - \mu') \La Q_{\U \mu}  + a^2(  a - \mu') \La A_{\U \mu} + 2a \mu ( a' +   \ga_k \frac{\nu^k}{\mu}) A_{\U\mu} + \nu^k(   a- \mu') \La \ti B_{\U \mu} \\
& \quad  +k \nu^{k-1} (  b + \mu \nu') \ti B_{\U \mu}     - b^3 \La A_{\U \la}  + 2 \gamma_k b \nu^k A_{\U \la}  - b \nu^k \La B_{\U \la} + k b \nu^k B_{\U \la} \\
& \quad -  a \La Q_{\U \mu} -  a^3 \La A_{\U \mu}  - 2  \ga_k a \nu^k A_{ \U\mu} - a \nu^k \La \ti B_{\U \mu}-  k b \nu^{k-1} \ti B_{ \U\mu} 
}
It follows then that 
\EQ{
&\abs{- \ang{ w \mid  \frac{1}{r^2} f'' (\Phi )(-\p_t \Phi  + \p_t Q_{\U \la} )w}} \lesssim \frac{b^2 + \nu^k}{\la} \abs{ b+ \la'}  \| w \|_H^2   +  b \abs{ b' + \gamma_k \la^{-1} \nu^{k}}  \| w \|_H^2 \\
 & \qquad  \qquad    +\frac{ 1+ a^2 + \nu^k}{\mu} \abs{a - \mu'}  \| w \|_H^2   +  a \abs{ a' +   \ga_k \frac{\nu^k}{\mu}}  \| w \|_H^2   + \frac{b \nu^k +  b^3}{\la}   \| w \|_H^2  \\
& \qquad \qquad  + (  \frac{\nu^k}{\la} + \frac{ \nu^{k-1}}{\mu})  \abs{  b + \mu \nu'}  \| w \|_H^2    +  \frac{a + a^3 + a \nu^k}{\mu}  \| w \|_H^2  +   \frac{ b}{\mu} \nu^{k-1}  \| w \|_H^2
}
Plugging all of the above yields the estimate 
\EQ{
\Big|\frac{1}{2} \ang{ w \mid [\p_t, \LL_\Phi] w} &-   \ang{  w \mid  [\p_t, \LL_\Phi] w} \Big|    \\
& \, \,  \lesssim  \| w \|_H^2 \Bigg(  \frac{b \la^{\frac{k}{2}} + a\nu + b^2 + \nu^k}{\la}   + \frac{\abs{b+ \la'}}{\la}\left( \la^{\frac{k}{2}} + b^2 +  \nu^k\right)   \\
& \, \,  \quad +  b \abs{ b' + \gamma_k \la^{-1} \nu^{k}}  
      + \frac{\nu}{\la}  \abs{a - \mu'}    +  a \abs{ a' +   \ga_k \frac{\nu^k}{\mu}}      +  \frac{\nu^{k}}{\la}    \abs{  b + \mu \nu'}        \Bigg)
}
An application of Lemma~\ref{l:modc2} to the right-hand side above gives~\eqref{eq:comm-est1}
as claimed. 
\end{proof} 
Finally, to complete the proof of Lemma~\ref{l:E1'} we combine the estimates from Claim~\ref{c:Psi2wdot}, Claim~\ref{c:Psi1Lw} and Claim~\ref{c:vir} to obtain, 
\EQ{
\Big| \Es_1'(t) &+  b\ang{  \frac{1}{\la} \La w \mid  P_\lam w}  - \frac{b}{\lam} \ang{ w \mid P_\lam w} \Big| \lesssim   \Big( \frac{ \nu^{2k-1}}{\la} + \frac{b^2 \nu^{k-1}}{\la}  + \frac{b^4}{\la} + \frac{a^4\nu}{\la}       \Big)  \| \dot w \|_{L^2} \\ 
&\quad + \frac{\abs{a}^3 + \abs{b}^3 + ( \abs{a} + \abs{b})\nu^{k-1}}{\la} \| \dot w \|_{L^2}   \|  \bs w \|_{\HH} + \frac{1}{\la} \| \dot w \|_{L^2} \| \bs w \|_{\HH}^2  + \| \dot w \|_{L^2} \| \bs w \|_{\HH}^2 \| \p_r w \|_H  \\
& \quad +\Big(  \frac{\abs{a}^5}{\lam} + \frac{\abs{b}^5}{\lam} + \frac{(\abs{a} + \abs{b}) \nu^{2k-1}}{\lam}  + \frac{(\abs{a} + \abs{b})b^2 \nu^{k-1}}{\lam} \Big) \| \bs w \|_{\HH} \\
& \quad  + \Big( \frac{ a^2 + b^2 + \nu^{2k}}{\lam} \Big) \| \bs w \|_{\HH}^2  + \Big( \frac{ \abs{a} + \abs{b}}{\lam} \Big) \| \bs w \|_{\HH}^3  \\
& \quad +  \Big( \frac{\abs{b} \la^{\frac{k}{2}} + \abs{a}\nu + b^2 + \nu^k}{\la} \Big) \| w \|_{H}^2  + \Big( \frac{ \abs{a} + \abs{b}}{\lam}  \Big) \| \bs w \|_{\HH}^3 + \Big( \frac{ \abs{a} + \abs{b}}{\lam}  \Big) \| \bs w \|_{\HH}^4
}
which reduces to 
\EQ{
\Big| \Es_1'(t) &+  b\ang{  \frac{1}{\la} \La w \mid  P_\lam w}  - \frac{b}{\lam} \ang{ w \mid P_\lam w} \Big| \lesssim  \Big( \frac{\nu^{2k-1}}{\la} + \frac{b^2 \nu^{k-1}}{\la} +  \frac{b^4}{\la} + \frac{a^4}{\la} \Big) \| \bs w \|_{\HH} \\
& \quad +  \Big( \frac{a\nu}{\la}  + \frac{b \la^{\frac{k}{2}}}{\la} + \frac{\nu^{k}}{\la} + \frac{a^2}{\la}+  \frac{b^2}{\la} \Big) \| \bs w \|_{\HH}^2    +  \frac{1}{\la} \| \dot w \|_{L^2} \| \bs w \|_{\HH}^2\\
& \quad + \Big( \frac{ \abs{a} + \abs{b}}{\lam}  \Big) \| \bs w \|_{\HH}^3 + \Big( \frac{ \abs{a} + \abs{b}}{\lam}  \Big) \| \bs w \|_{\HH}^4+ \| \dot w \|_{L^2} \| \bs w \|_{\HH}^2 \| \p_r w \|_H  
}
which completes the proof. 
\end{proof} 

The proof of Lemma~\ref{l:Vs1'} (as well as the proof of Lemma~\ref{l:V2'} in the next subsection) requires additional estimates that we group together and prove here. 

\begin{lem}  \label{l:ptdotPhiL2} The following estimates hold true. 
\EQ{ \label{eq:Psi1-H} 
\| \dot \Phi - \p_t \Phi \|_H &\lesssim   \frac{(\abs{a} + \abs{b} ) }{\lam} \| \bs w \|_{\HH} +   \frac{(\abs{a} + \abs{b} ) }{\lam} \| \bs w \|_{\HH}^2  \\
&\quad +(\abs{a} + \abs{b}) \Big( \frac{\nu^{2k-1}}{\lam} + \frac{b^2 \nu^{k-1}}{\lam} +  \frac{b^4}{\lam} + \frac{a^4}{\lam}  \Big)
 }
\EQ{  \label{eq:Psi1-H2} 
\| \dot \Phi - \p_t \Phi \|_{H^2}  &\lesssim  \frac{(\abs{a} + \abs{b} ) }{\lam^2} \| \bs w \|_{\HH} +   \frac{(\abs{a} + \abs{b} ) }{\lam^2} \| \bs w \|_{\HH}^2  \\
&\quad +(\abs{a} + \abs{b}) \Big( \frac{\nu^{2k-1}}{\lam^2} + \frac{b^2 \nu^{k-1}}{\lam^2} +  \frac{b^4}{\lam^2} + \frac{a^4}{\lam^2}  \Big)
}
\EQ{ \label{eq:ptdotPhiL2} 
\| - \p_t \dot \Phi + \De \Phi  -\frac{1}{r^2} f(\Phi)  \|_{L^2} &\lesssim   \frac{ \nu^{2k-1}}{\la} + \frac{b^2 \nu^{k-1}}{\la}  + \frac{b^4}{\la} + \frac{a^4\nu}{\la}   \\
&\quad  +  \frac{\abs{a}  + \abs{b}}{\la}   \| \bs w \|_{\HH}     +  \frac{1}{\la} \| \bs w \|_{\HH}^2 
}
\EQ{ \label{eq:ptdotPhi-H} 
\| - \p_t \dot \Phi + \De \Phi  -\frac{1}{r^2} f(\Phi)  \|_{H} &\lesssim \frac{ \nu^{2k-1}}{\la^2} + \frac{b^2 \nu^{k-1}}{\la^2}  + \frac{b^4}{\la^2} + \frac{a^4}{\la^2}   \\
&\quad  +  \frac{\abs{a}  + \abs{b}}{\la^2}   \| \bs w \|_{\HH}     +  \frac{1}{\la^2} \| \bs w \|_{\HH}^2
}
\end{lem} 
\begin{proof} 
We prove~\eqref{eq:Psi1-H} by directly estimating~\eqref{eq:dotPhi-ptPhi} in $H$. This yields, 
\EQ{
\| \dot \Phi - \p_t \Phi \|_H &\lesssim \frac{\abs{b+ \la'}}{\lam} +  \nu \frac{ \abs{ a- \mu'}}{\lam}  + \abs{b} \abs{b' + \gamma_k \la^{-1} \nu^k} + \abs{a} \abs{ a' + \ti\gamma_k \la^{-1} \nu^{k+1}} \\
& \lesssim \frac{(\abs{a} + \abs{b} ) }{\lam} \| \bs w \|_{\HH}   + ( \abs{a}^5 + \abs{b}^5) \frac{\nu^{k-2}}{\lam} + ( \abs{a} + \abs{b})\frac{\abs{a}  + \abs{b}}{\la}   \| \dot w \|_{L^2}  \\
& \quad  +  \frac{\abs{a} + \abs{b}}{\la} \| \bs w \|_{\HH}^2 + (\abs{a} + \abs{b}) \Big( \frac{\nu^{2k-1}}{\lam} + \frac{b^2 \nu^{k-1}}{\lam} +  \frac{b^4}{\lam} + \frac{a^4}{\lam}  \Big) 
}
The proof of~\eqref{eq:Psi1-H2} is similar. 


We proceed exactly as in the proof of~\eqref{eq:ptdotPhiest}, but noting that here we are not pairing $ - \p_t \dot \Phi + \De \Phi  -\frac{1}{r^2} f(\Phi)$ with $\dot w$ so we do not see the additional gain from the orthogonality conditions~\eqref{eq:lawdot} and~\eqref{eq:muwdot}. Directly estimating~\eqref{eq:ptdotPhi} in $L^2$ gives 
\EQ{
\| - \p_t \dot \Phi + \De \Phi  -\frac{1}{r^2} f(\Phi)  \|_{L^2} & \lesssim \frac{ \nu^{2k-1}}{\la} + \frac{b^2 \nu^{k-1}}{\la}  + \frac{b^4}{\la} + \frac{a^4\nu}{\la} + \frac{a^2 \nu^{k+1}}{\la}  + \frac{ab \nu^{k}}{\la}   \\
&\quad   +  \abs{b'+ \gamma_k \frac{\nu^k}{\la}}   + \abs{a'  + \ti \gamma_k \frac{\nu^k}{\mu}} + \frac{b}{\la} \abs{ b + \la'}  + \frac{a\nu}{\la} \abs{a - \mu'} 
 }
 An application of Lemma~\ref{l:modc2} to the right-hand side above yields, 
 \EQ{
 \| - \p_t \dot \Phi + \De \Phi & -\frac{1}{r^2} f(\Phi)  \|_{L^2}  \lesssim \frac{ \nu^{2k-1}}{\la} + \frac{b^2 \nu^{k-1}}{\la}  + \frac{b^4}{\la} + \frac{a^4\nu}{\la}   \\
& \quad +  \frac{\abs{a}  + \abs{b}}{\la}   \| \dot w \|_{L^2}  +  \frac{1}{\la} \| \bs w \|_{\HH}^2 + \frac{\nu^{2k-1}}{\lam} + \frac{b^2 \nu^{k-1}}{\lam} +  \frac{b^4}{\lam} + \frac{a^4}{\lam} \\
&\quad + \frac{\abs{a}  + \abs{b}}{\la} \Big((\abs{a} + \abs{b} ) \| \bs w \|_{\HH} +( \abs{a} + \abs{b})\nu^{2k-1}  + ( \abs{a}^5 + \abs{b}^5) \nu^{k-2} \Big)  \\
& \lesssim  \frac{ \nu^{2k-1}}{\la} + \frac{b^2 \nu^{k-1}}{\la}  + \frac{b^4}{\la} + \frac{a^4\nu}{\la}   +  \frac{\abs{a}  + \abs{b}}{\la}   \| \bs w \|_{\HH}     +  \frac{1}{\la} \| \bs w \|_{\HH}^2 
 }
 as claimed. The proof of~\eqref{eq:ptdotPhi-H} is similar. 
\end{proof} 

\begin{proof}[Proof of Lemma~\ref{l:Vs1'}]
For the purpose of this  proof we rewrite the equation for $\bs w$ as follows, 
\EQ{ \label{eq:w-eq4} 
\p_t w &=  \dot w + \Psi_1, \\
\p_t \dot w &= -\LL_\lam w +  \ti \Psi_2,
}
where $\Psi_1$ is defined as usual as in~\eqref{eq:Psi1} and  $\ti \Psi_2$ is defined as 
\EQ{ \label{eq:tiPsi2-def} 
 \ti \Psi_2 = (-\p_t \dot \Phi + \De \Phi -\frac{1}{r^2} f( \Phi)) - \frac{1}{r^2}( f( \Phi + w) - f( \Phi) - f'(Q_\lam) w).
}
We differentiate the virial correction, using~\eqref{eq:w-eq4} below. 
\EQ{ \label{eq:bV1'} 
(b \Vs_1)'(t)   &= b' \ang{ \calA_0(\la) w \mid \dot w} + \frac{b \la' }{\la}  \ang{ \la  \p_\la \calA_{0}(\la) w \mid \dot w}   + b \ang{ \calA_0(\la)\p_t  w \mid \dot w}   + b \ang{ \calA_0(\la)  w \mid \p_t  \dot w}  \\
& =  b' \ang{ \calA_0(\la) w \mid \dot w} + \frac{b \la' }{\la}  \ang{ \la  \p_\la \calA_{0}(\la) w \mid \dot w}  +  b \ang{ \calA_0(\la)\dot w \mid \dot w} +  b \ang{ \calA_0(\la)\Psi_1 \mid \dot w}  \\
&\quad + b   \ang{ \calA_0(\la) w \mid  \ti \Psi_2}  -  b \ang{ \calA_0(\la)w \mid \LL_\lam w } 
}
First, consider the last term on the right-hand side of~\eqref{eq:bV1'} above. Writing $\LL_\lam = \LL_0 + P_\lam$ as usual we have, 
\EQ{
 -  b \ang{ \calA_0(\la)w \mid \LL_\lam w }  &=  -  b \ang{ \calA_0(\la)w \mid \LL_0 w }  -    b \ang{ \calA_0(\la)w \mid P_\lam w }  \\
 & \le c_0 \frac{b}{\la} \| w \|_{H}^2  - \frac{b}{\la} \int_0^{R \la}\Big( ( \p_r  w)^2 + \frac{k^2}{r^2} w^2  \Big) \, \rdr  - \frac{b}{\lam} \ang{\Lam_0 w \mid P_\la  w}
}
 where the last line follows by using the estimate~\eqref{eq:A-pohozaev} from Lemma~\ref{l:opA} on the first term along with the estimate~\eqref{eq:vir-new} for the second term. Note that  $c_0>0$, $R>0$ are as in Lemma~\ref{l:opA} and $c_0>0$ can be taken as small as we like independent of $a, b, \la, \mu$. 

Next, we estimate each of the remaining terms on the right-hand side of~\eqref{eq:bV1'}. For the first term note that we have 
\EQ{
\abs{b' \ang{ \calA_0(\la) w \mid \dot w}}  &\lesssim  \Big(\frac{\nu^k}{\la}+ \Big| b' + \gamma_k \frac{\nu^k}{\la}\Big|\Big) \| w \|_{H} \| \dot w \|_{L^2}  \\
& \lesssim  \Big(\frac{\nu^k}{\la}+  \frac{\abs{a}  + \abs{b}}{\la}   \| \dot w \|_{L^2}  +  \frac{1}{\la} \| \bs w \|_{\HH}^2   +  \frac{b^4}{\lam} + \frac{a^4}{\lam} \Big) \| w \|_{H} \| \dot w \|_{L^2} \\
&\lesssim \Big( \frac{\nu^k}{\la} +\frac{a^4}{\lam}  +  \frac{b^4}{\la}  \Big)\| \bs w \|_{\HH}^2 +  \frac{\abs{a}  + \abs{b}}{\la}   \| \dot w \|_{L^2}^2\| \bs w \|_{\HH} + \frac{1}{\la}  \| \dot w \|_{L^2} \| \bs w \|_{\HH}^3 
}
using the first bullet point in Lemma~\ref{l:opA}  and Lemma~\ref{l:modc2} above. 
For the second term, we again use Lemma~\ref{l:opA} and Lemma~\ref{l:modc2} to conclude that 
\EQ{
\abs{\frac{b \la' }{\la}  \ang{ \la  \p_\la \calA_{0}(\la) w \mid \dot w}}  &\lesssim \Big(\frac{b^2}{\la} +  \frac{\abs{b}\abs{ b+ \la'} }{\la} \Big)\| \dot w \|_{L^2} \|  w \|_{H} \\
& \lesssim \Bigg(\frac{b^2}{\la} +  \frac{\abs{b}}{\la} \Big[(\abs{a} + \abs{b} ) \| \bs w \|_{\HH} +( \abs{a} + \abs{b})\nu^{2k-1}  \Big]  \Bigg)\| \dot w \|_{L^2} \|  w \|_{H} \\
 & \lesssim \Big( \frac{b^2}{\la} + \frac{a^2 \nu^{4k-2}}{\la}   \Big) \| \dot w \|_{L^2} \|  w \|_{H}  + \frac{a^2 + b^2}{\la} \| \dot w \|_{L^2} \|  \bs w \|_{\HH}^2   
}
For the third term we note that 
\EQ{
 b \ang{ \calA_0(\la)\dot w \mid \dot w} = 0
}
which follows directly from the definition of $\calA_0(\la)$ and integration by parts. For the fourth term we apply the first bullet point in~Lemma~\ref{l:opA} and the definition of $\Psi_1$ followed by the estimate~\eqref{eq:Psi1-H}, 
\EQ{
&\Big| b \big\langle  \calA_0(\la) \Psi_1 \mid \dot w \big\rangle \Big| \lesssim \abs{b} \| \dot \Phi- \p_t \Phi \|_H \| \dot w \|_{L^2}  \\
& \lesssim \abs{b} \Bigg(  \frac{(\abs{a} + \abs{b} ) }{\lam} \| \bs w \|_{\HH} +   \frac{(\abs{a} + \abs{b} ) }{\lam} \| \bs w \|_{\HH}^2   +(\abs{a} + \abs{b}) \Big( \frac{\nu^{2k-1} + b^2 \nu^{k-1} +  b^4 + a^4}{\lam}  \Big) \Bigg) \| \dot w \|_{L^2}  \\
& \lesssim  \frac{a^6 + b^6 +( \abs{a} + \abs{b})( \abs{b} \nu^{2k-1} + \abs{b}^3 \nu^{k-1} )}{\lam} \| \dot w  \|_{L^2} +  \frac{a^2 + b^2}{\lam} \Big( \| \bs w \|_{\HH} \| \dot w \|_{L^2} +  \| \bs w \|_{\HH}^2 \| \dot w \|_{L^2} \Big) 
}
We use the definition of $\ti \Psi_2$ to write the fifth term in~\eqref{eq:bV1'} as follows, 
\EQ{ \label{eq:5-term} 
 b   \ang{ \calA_0(\la) w \mid  \ti \Psi_2} = b   \Big\langle \calA_0(\la) w & \mid  - \p_t \dot \Phi + \De \Phi  -\frac{1}{r^2} f(\Phi) \Big\rangle \\
 & \quad  -  b \ang{ \calA_0(\la) w  \mid \frac{1}{r^2}( f( \Phi + w) - f( \Phi) - f'(Q_\lam) w)}
}
For the first term on the right above we use~\eqref{eq:ptdotPhiL2} from Lemma~\ref{l:ptdotPhiL2} to conclude that 
\EQ{
\Big| b   \Big\langle \calA_0(\la) w & \mid  - \p_t \dot \Phi + \De \Phi  -\frac{1}{r^2} f(\Phi) \Big\rangle \Big| \lesssim \abs{b} \| w \|_H  \| - \p_t \dot \Phi + \De \Phi  -\frac{1}{r^2} f(\Phi) \|_{L^2} \\
&\lesssim\abs{ b} \| w \|_H \Bigg(\ \frac{ \nu^{2k-1}}{\la} + \frac{b^2 \nu^{k-1}}{\la}  + \frac{b^4}{\la} + \frac{a^4}{\la}   +  \frac{\abs{a}  + \abs{b}}{\la}   \| \bs w \|_{\HH}     +  \frac{1}{\la} \| \bs w \|_{\HH}^2 \Bigg) \\
& \lesssim \Big( \frac{\abs{b} \nu^{2k-1}}{\la}  + \frac{\abs{b}^5}{\la}  + \frac{\abs{a}^5}{\la} \Big) \| \bs w \|_{\HH}    + \frac{b^2 + a^2 }{\la} \| \bs w \|_{\HH}^2   + \frac{\abs{b}}{\la} \| \bs w \|_{\HH}^3 
}

Finally, we examine the last term on the right-hand side of~\eqref{eq:5-term}. Note that 
\EQ{
f( \Phi + w) - f(\Phi) - f'(Q_\lam) w  &=  -k^2 \sin 2 \Phi \sin^2 w + \frac{k^2}{2} ( \sin 2 w - 2w) \cos 2 \Phi  \\
&\quad +k^2 ( \cos 2 \Phi - \cos 2Q_\lam) w 
}
and further we have 
\EQ{ \label{eq:f'Phi-f'Q-1} 
k^2 ( \cos 2 \Phi - \cos 2Q_\lam) w 
& =  -  2k^2\cos 2 Q_\la  \sin^2( \Phi - Q_\la)w - k^2\sin 2 Q_\la \sin 2( \Phi - Q_\la)w
}
Hence, 
\EQ{
 - b \Big\langle \calA_0(\la) w &\mid \frac{1}{r^2} \Big( f( \Phi + w) - f(\Phi) - f'(Q_\lam) w \Big) \Big\rangle =    b \Big\langle \calA_0(\la) w \mid \frac{k^2}{r^2} \sin 2 \Phi \sin^2 w \Big\rangle\\  
 &  - b \Big\langle \calA_0(\la) w \mid \frac{k^2}{2r^2}  ( \sin 2 w - 2w) \cos 2 \Phi \Big\rangle  + b   \ang{ \calA_0(\la) w \mid \frac{2k^2}{r^2}\cos 2 Q_\la  \sin^2( \Phi - Q_\la) w } \\
 &  + b   \ang{ \calA_0(\la) w \mid \frac{k^2}{r^2}\sin 2 Q_\la \sin 2( \Phi - Q_\la) w}
}
To estimate the last two terms on the right above we recall that by~\eqref{eq:sin(Phi-Qla)} that we have point-wise estimates 
\EQ{ \label{eq:f'Phi-f'Q-2} 
\frac{1}{r}\abs{\sin^2( \Phi- Q_\la) } \lesssim  1  + \frac{b^4}{\la} + \frac{\nu^{2k}}{\la} + a^4  \\
\frac{1}{r}\abs{\sin2Q_\la \sin2( \Phi- Q_\la) } \lesssim  1 + \frac{b^2}{\la} + \frac{\nu^k}{\la} +a^2 
}
Combining these estimates above with the previous line and Lemma~\ref{l:opA} gives, 
\EQ{ \label{eq:tiPsi2-1} 
 \Big|  b \Big\langle  \calA_0(\la) w \mid \frac{1}{r^2} \Big( f( \Phi + w) - f(\Phi) - f'(Q_\lam) w \Big) \Big\rangle \Big| & \lesssim \abs{b} \| w \|_H^2   \| w \|_{H^2} +  \abs{b} \|w \|_H^3\| w \|_{H^2}\\
 &\quad + \Big( \abs{b}  + \frac{\abs{a}^3 + \abs{b}^3 + \nu^{\frac{3}{2}k}}{\lam} \Big)\| \bs w \|_{\HH}^2 
}
Putting together all of the estimates for the terms on the right-hand side of~\eqref{eq:bV1'} we obtain, 
\EQ{
& (b \Vs_1)'(t)   + \frac{b}{\lam}  \ang{ \La_0 w \mid P_\lam  w}   
 \le c_0 \frac{b}{\la} \| w \|_H^2  -  \frac{1}{\lambda}\int_0^{R\lambda}\Big((\partial_r w)^2 + \frac{k^2}{r^2}w^2\Big) \udr  \\
 &\quad  + C_1 \Big(\frac{\abs{b} \nu^{2k-1}}{\la}  + \frac{\abs{b}^5}{\la}  + \frac{\abs{a}^5}{\la}  \Big) \| \dot w \|_{L^2}  + C_1 \Big(  \frac{ \abs{b} \lam + a^2 + b^2 + \nu^{k} }{\lam}  \Big)   \| \bs w \|_{\HH}^2 \\
& \quad  + C_1  \frac{ \abs{a} + \abs{b}}{\lam} \| \dot w \|_{L^2} \| \bs w \|_{L^2}^2  + C_1 \frac{1}{\lam}  \| \dot w \|_{L^2} \| \bs w \|_{\HH}^3 +  \abs{b} \| w \|_H^2   \| w \|_{H^2} +  \abs{b} \|w \|_H^3\| w \|_{H^2}
}
for some uniform constant $C_1>0$, as claimed. 
\end{proof}

\subsection{Energy estimates for $\bs w(t)$ in $\HH^2$} 
The goal of this section is to establish energy-type estimates for $\bs w$ in $\HH^2$. 
We define the $\HH^2$-type energy, 
\EQ{ \label{eq:E2} 
\Es_2(t)&
 := \frac{1}{2} \ang{  \dot w \mid \LL_\Phi  \dot w} + \frac{1}{2} \ang{ w \mid \LL_{\Phi}^2 w} 
}
We also define a virial correction, 
\EQ{ \label{eq:V2} 
\Vs_{2, \la}(t):= 
\ang{ \A_0(\la) w \mid \LL_\la \dot w}  - 2  \ang{\A_0(\la) \dot w \mid \LL_\la w}
}
where $\A_0(\lam)$ is as in Lemma~\ref{l:opA}. 
Finally, we define the mixed $\HH^2$-energy-virial functional 
\EQ{ \label{eq:H2} 
\Hs_2(t) = \Es_2(t) - b(t) \Vs_{2, \la}(t)
}
These functionals as related as follows. 
\begin{lem} 
Let $J \subset \R$ be an interval on which $\bs u(t)$ satisfies the hypothesis of Lemma~\ref{l:mod2}.  Assume in addition that $\bs u(t) \in \HH \cap \HH^2 \cap \bs \Lam^{-1} \HH$ for all $t \in J$.  Let $a, b, \la, \mu$  and $\bs w(t) \in \HH \cap \HH^2 \cap \bs \Lam^{-1} \HH$ be given by Lemma~\ref{l:mod2} and let $\nu:= \la/\mu$ as usual. Then, 
\EQ{ \label{eq:E2H2} 
 \| \bs w \|_{\HH^2}^2   \simeq  \Es_2 (t)  &  \\
\Hs_2(t)  - \frac{\abs{b}}{\lam} \| \bs w \|_{\HH} \| \bs w \|_{\HH^2}  \lesssim      \Es_2(t)  &\lesssim  \Hs_2(t)  + \frac{\abs{b}}{\lam} \| \bs w \|_{\HH} \| \bs w \|_{\HH^2} 
}
\end{lem} 
\begin{proof} 
The first line in~\eqref{eq:E2H2} is a consequence of the coercivity estimate in Lemma~\ref{l:coerce2}. The second line follows from the fact that $\abs{b}  \ll 1$ on $J$ along with the second bullet point in Lemma~\ref{l:opA}. 
\end{proof} 

We next state the main result of this section. 
\begin{prop}  \label{p:H2'} Let $J \subset \R$ be an interval on which $\bs u(t)$ satisfies the hypothesis of Lemma~\ref{l:mod2}.  Assume in addition that $\bs u(t) \in \HH \cap \HH^2 \cap \bs \Lam^{-1} \HH$ for all $t \in J$.  Let $a, b, \la, \mu$  and $\bs w(t) \in  \HH \cap \HH^2 \cap \bs \Lam^{-1} \HH$ be given by Lemma~\ref{l:mod2} and let $\nu:= \la/\mu$ as usual. Then, there exists a uniform constant $C_1>0$ so that 
\EQ{ \label{eq:H2'} 
\Hs_2'(t)    &\ge  2\frac{b}{\lam} \int_0^{R\lam} ( \LL_0 w)^2  \, r\, \ud r  + 2 \frac{b}{\lam}\ang{ w \mid \Ks_\lam w}  \\
&\quad + \frac{b}{\la} \int_0^{R \la}\Big( ( \p_r \dot w)^2 + \frac{k^2}{r^2} \dot w^2  \Big) \, \rdr + \frac{b}{\lam} \ang{ \dot w \mid P_\lam \dot w} \\
&\quad - c_0 \frac{b}{\lam} \| \bs w \|_{\HH^2}^2    -  C_1   \Big( \abs{a} + \abs{b} +  \frac{b^2+a^2 + \nu^k}{\lam} \Big) \| \bs w \|_{\HH^2}^2  - C_1\frac{ a^2 + b^2}{\lam} \frac{\| \bs w \|_{\HH}^2}{\lam^2}  \\ 
&\quad   - C_1  \Big( \frac{b^2 \nu^{k-1}+ \nu^{2k-1}+ b^4 + a^4}{\la^2}  \Big)  \| \bs w \|_{\HH^2}   \\
&\quad  -C_1 \Big( \frac{\abs{a}^5 + \abs{b}^5 + \abs{b} \nu^{2k-1} + \abs{b}^3 \nu^{k-1}}{\lam^2}\Big)    \frac{\| \bs w \|_{\HH}}{\lam}    \\
&\quad - C_1   \Big( \abs{a} + \abs{b} +  \frac{b^2+a^2 + \nu^k}{\lam} \Big) \frac{ \| \bs w \|_{\HH}}{\lam} \| \bs w \|_{\HH^2}  \\
&\quad  -C_1  \Big( \abs{a} + \abs{b} +  \frac{b^2+a^2 + \nu^k}{\lam} \Big) \frac{ \| \bs w \|_{\HH}^2}{\lam^2} \| \bs w \|_{\HH^2}  \\
&\quad   -C_1  \frac{1}{\la}\| \bs w \|_{\HH} \| \bs w \|_{\HH^2}^2      - C_1\frac{1}{\lam^2} \| \bs w \|_{\HH}^3 \| \bs w\|_{\HH^2}  - C_1\abs{b} \frac{ \| \bs w\|_{\HH}^3}{\lam^3}   \\ 
& \quad - C_1 \Big( \abs{a} + \abs{b}  + \| \bs w \|_{\HH} \Big)   \| \bs w \|_{\HH^2}^3- C_1 \Big( \abs{a} + \abs{b} + \lam \Big) \frac{ \| \bs w \|_{\HH}^2}{\lam^2}  \| \bs w \|_{\HH^2}^2 \\
}
where $c_0>0$ is a constant that can be taken as small as we like by choosing $R>0$ large enough in Lemma~\ref{l:opA}.  Importantly, $c_0$ can be taken small independently of $\mu, \lam, a, b$ and $J$. Note that above $\Ks$ is as in~\eqref{eq:K-def} and $P$ is as in~\eqref{eq:P-def}. 
\end{prop} 

Proposition~\ref{p:H2'} is a consequence of the following two lemmas. 
\begin{lem}  \label{l:E2'} 
Let $J \subset \R$ be an interval on which $\bs u(t)$ satisfies the hypothesis of Lemma~\ref{l:mod2}.  Assume in addition that $\bs u(t) \in \HH \cap \HH^2 \cap \bs \Lam^{-1} \HH$ for all $t \in J$.  Let $a, b, \la, \mu$  and $\bs w(t) \in  \HH \cap \HH^2 \cap \bs \Lam^{-1} \HH$ be given by Lemma~\ref{l:mod2} and let $\nu:= \la/\mu$ as usual. Then, 
\EQ{
\Es_2'(t) &= - \frac{b}{\la} \ang{ \La_0 \dot w \mid P_{\la} \dot w} + \frac{b}{\lam} \ang{ \dot w \mid P_\lam \dot w}  + \frac{b}{\la}  \ang{w \mid \Ks_\la w}   +\frac{b}{\la} \ang{ w \mid P_\la^2 w}  \\
& \quad  +  \frac{b}{\la} \ang{ w \mid  \La P_\la\LL_0 w  -  \p_r (\La P_\la) \p_r w  +  \frac{1}{2} (-\De \La P_\la) w +  P_\la \La P_\la w  }  \\
&\quad  + O \Bigg( \Big( \frac{b^2 \nu^{k-1}+ \nu^{2k-1}+ b^4 + a^4}{\la^2}  \Big)  \| \bs w \|_{\HH^2} +    \Big( \frac{b^2+ a^2 + \nu^{2k}}{\la^2}  \Big) \| \bs w \|_{\HH} \| \bs w \|_{\HH^2}\Bigg)   \\
& \quad +  O \Bigg(  \Big( \frac{b \la^{\frac{k}{2}} + a\nu + b^2 + \nu^k}{\la} \Big) \| \bs w \|_{\HH^2}^2    +  \frac{1}{\la}\| w \|_{H} \| \bs w \|_{\HH^2}^2   +  \frac{ \abs{a} + \abs{b}}{\lam}   \| \bs w \|_{\HH^2}^2 \| \bs w \|_{\HH}^2 \Bigg)   \\ 
  &\quad  + O \Bigg(    \Big( \frac{\abs{a} + \abs{b} + \nu^k}{\lam^2} \Big)  \| \bs w \|_{\HH}^2 \| \bs w \|_{\HH^2}  +  \| w \|_H \| \bs w \|_{\HH^2}^3\Bigg) 
} 
uniformly on $J$. Note that above,  $\Ks$ is as in~\eqref{eq:K-def} and $P$ is as in~\eqref{eq:P-def}. 
\end{lem} 

\begin{lem}\label{l:V2'} 
Let $J \subset \R$ be an interval on which $\bs u(t)$ satisfies the hypothesis of Lemma~\ref{l:mod2}.  Assume in addition that $\bs u(t) \in \HH \cap \HH^2 \cap \bs \Lam^{-1} \HH$ for all $t \in J$.  Let $a, b, \la, \mu$  and $\bs w(t) \in  \HH \cap \HH^2 \cap \bs \Lam^{-1} \HH$ be given by Lemma~\ref{l:mod2} and let $\nu:= \la/\mu$ as usual.  Then, there exists a uniform constant $C_1>0$ so that, 
\EQ{
(b \Vs_{2, \la})'(t)  &\le c_0 \frac{b}{\lam} \| \bs w \|_{\HH^2}^2  \\
&\quad -  2\frac{b}{\lam} \int_0^{R\lam} ( \LL_0 w)^2  \, r\, \ud r  - \frac{b}{\lam} \ang{ w \mid \Ks_\lam w}  - \frac{b}{\la} \int_0^{R \la}\Big( ( \p_r \dot w)^2 + \frac{k^2}{r^2} \dot w^2  \Big) \, \rdr \\
&\quad  - \frac{b}{\lam} \ang{\Lam_0 \dot w \mid P_\la \dot w}+ \frac{b}{\lam} \ang{w \mid P_\lam^2 w}   \\
&\quad + \frac{b}{\lam} \ang{w \mid \La P_\lam \LL_0 w -  \p_r (\La P_\lam) \p_r w + \frac{1}{2}(- \De \La P_\lam) w +  P_\lam \La P_\lam w}  \\
&\quad  +C_1 \Big( \frac{\abs{a}^5 + \abs{b}^5 + \abs{b} \nu^{2k-1} + \abs{b}^3 \nu^{k-1}}{\lam^2}\Big)\Big( \| \bs w \|_{\HH^2} +   \frac{\| \bs w \|_{\HH}}{\lam}  \Big) \\
& \quad +C_1  \Big( \abs{a} + \abs{b} +  \frac{b^2+a^2 + \nu^k}{\lam} \Big)\Big(  \| \bs w \|_{\HH^2}+ \frac{ \| \bs w \|_{\HH}}{\lam} + \frac{ \| \bs w \|_{\HH}^2}{\lam^2}\Big)  \| \bs w \|_{\HH^2}   \\
& \quad + C_1 \Big( \abs{a} + \abs{b} \Big)\Big( \| \bs w \|_{\HH^2} + \frac{ \| \bs w \|_{\HH}}{\lam} +  \frac{ \| \bs w \|_{\HH}^2}{\lam^2} \Big) \| \bs w \|_{\HH^2}^2 + C_1\frac{ a^2 + b^2}{\lam^2} \frac{\| \bs w \|_{\HH}^2}{\lam} \\
&\quad   +C_1 \frac{1}{\lam} \| \bs w \|_{\HH}^2 \| \bs w \|_{\HH^2}^2   + C_1\frac{1}{\lam^2} \| \bs w \|_{\HH}^3 \| \bs w\|_{\HH^2}  + C_1\frac{\abs{b}}{\lam^3} \| \bs w\|_{\HH}^3 
}
where $c_0>0$ is a constant that can be taken as small as we like by choosing $R>0$ large enough in Lemma~\ref{l:opA}.  Importantly, $c_0$ can be taken small independently of $\mu, \lam, a, b$ and $J$. 
\end{lem} 

%

We first prove  Proposition~\ref{p:H2'} assuming the conclusions of Lemma~\ref{l:E2'} and Lemma~\ref{l:V2'}. 

\begin{proof}[Proof of Proposition~\ref{p:H2'}] 
From Lemma~\ref{l:E2'} and Lemma~\ref{l:V2'} we have 
\EQ{
\Hs_2'(t) &\ge  2\frac{b}{\lam} \int_0^{R\lam} ( \LL_0 w)^2  \, r\, \ud r  + 2 \frac{b}{\lam}\ang{ w \mid \Ks_\lam w}  \\
&\quad + \frac{b}{\la} \int_0^{R \la}\Big( ( \p_r \dot w)^2 + \frac{k^2}{r^2} \dot w^2  \Big) \, \rdr + \frac{b}{\lam} \ang{ \dot w \mid P_\lam \dot w} \\
&\quad - c_0 \frac{b}{\lam} \| \bs w \|_{\HH^2}^2    -  C_1   \Big( \abs{a} + \abs{b} +  \frac{b^2+a^2 + \nu^k}{\lam} \Big) \| \bs w \|_{\HH^2}^2  - C_1\frac{ a^2 + b^2}{\lam} \frac{\| \bs w \|_{\HH}^2}{\lam^2}  \\ 
&\quad   - C_1  \Big( \frac{b^2 \nu^{k-1}+ \nu^{2k-1}+ b^4 + a^4}{\la^2}  \Big)  \| \bs w \|_{\HH^2}   \\
&\quad  -C_1 \Big( \frac{\abs{a}^5 + \abs{b}^5 + \abs{b} \nu^{2k-1} + \abs{b}^3 \nu^{k-1}}{\lam^2}\Big)    \frac{\| \bs w \|_{\HH}}{\lam}    \\
&\quad - C_1   \Big( \abs{a} + \abs{b} +  \frac{b^2+a^2 + \nu^k}{\lam} \Big) \frac{ \| \bs w \|_{\HH}}{\lam} \| \bs w \|_{\HH^2}   \\
&\quad   -C_1  \frac{1}{\la}\| \bs w \|_{\HH} \| \bs w \|_{\HH^2}^2      - C_1\frac{1}{\lam^2} \| \bs w \|_{\HH}^3 \| \bs w\|_{\HH^2}  - C_1\abs{b} \frac{ \| \bs w\|_{\HH}^3}{\lam^3}   \\ 
& \quad -C_1  \Big( \abs{a} + \abs{b} +  \frac{b^2+a^2 + \nu^k}{\lam} \Big) \frac{ \| \bs w \|_{\HH}^2}{\lam^2} \| \bs w \|_{\HH^2}   \\
& \quad - C_1 \Big( \abs{a} + \abs{b}  + \| \bs w \|_{\HH} \Big)   \| \bs w \|_{\HH^2}^3- C_1 \Big( \abs{a} + \abs{b} + \lam \Big) \frac{ \| \bs w \|_{\HH}^2}{\lam^2}  \| \bs w \|_{\HH^2}^2 \\
}
which completes the proof. 
\end{proof}

\begin{proof}[Proof of Lemma~\ref{l:E2'}]  
First we note the identity, 
\EQ{ \label{eq:E2'} 
\Es_2'(t) &=\ang{ \Psi_2 \mid \LL_{\Phi} \dot w } + \ang{ \Psi_1 \mid  \calL_\Phi^2 w}    +\frac{1}{2}\ang{ \dot w \mid [\p_t, \LL_\Phi] \dot w}  +\frac{1}{2} \ang{ w \mid [\p_t, \LL_\Phi^2] w}
}
To see this, we compute using~\eqref{eq:weq3}, 
\EQ{
\Es_2'(t) &= \ang{ \p_t  \dot w  \mid  \LL_\Phi  \dot w} + \ang{ \p_t w \mid \LL_\Phi^2 w}  + \frac{1}{2} \ang{ \dot w \mid [\p_t, \LL_\Phi] \dot w}  +  \frac{1}{2} \ang{  w \mid [\p_t, \LL_\Phi^2]  w}  \\
& = \ang{ -\calL_\Phi w \mid \LL_\Phi \dot w} + \ang{ \Psi_2 \mid \LL_{\Phi} \dot w} + \ang{ \dot w \mid \calL_\Phi^2 w} + \ang{ \Psi_1 \mid  \calL_\Phi^2 w} \\
& \quad+ \frac{1}{2} \ang{ \dot w \mid [\p_t, \LL_\Phi] \dot w}  +  \frac{1}{2} \ang{  w \mid [\p_t, \LL_\Phi^2]  w}  \\
& =\ang{ \Psi_2 \mid \LL_{\Phi} \dot w} + \ang{ \Psi_1 \mid  \calL_\Phi^2 w}     + \frac{1}{2} \ang{ \dot w \mid [\p_t, \LL_\Phi] \dot w}  +  \frac{1}{2} \ang{  w \mid [\p_t, \LL_\Phi^2]  w} 
}
which is precisely~\eqref{eq:E2'}. We estimate each of the terms on the right-hand side of~\eqref{eq:E2'}. 

\begin{claim}\label{c:Psi2Ldotw} 
The following estimate holds true. 
\EQ{ \label{eq:Psi2-Ldotw} 
\abs{\ang{ \Psi_2 \mid \LL_{\Phi} \dot w } } &\lesssim \| \dot w \|_{H} \Big( \frac{b^2 \nu^{k-1}}{\la^2}  +  \frac{\nu^{2k-1}}{\la^2} + \frac{b^4}{\la^2}+\frac{a^4}{\la^2} \Big)   \\ 
  &\quad  +   \| \dot w \|_{H}  \| \bs w \|_{\HH} \Big( \frac{b^2}{\la^2} + \frac{a^2}{\la^2} + \frac{\nu^{2k}}{\la^2} \Big)  +  \| \dot w \|_{H}  \| \bs w \|_{\HH}^2 \Big(   \frac{b^2}{\la^2} + \frac{a^2\nu}{\la^2}   + \frac{\nu^k}{\la^2} \Big)   \\
  & \quad + \frac{1}{\la} \|  w \|_{H^2} \| w \|_{H} \| \dot w \|_{H}     + \|  w \|_{H^2}^2 \| w \|_H  \| \dot w \|_{H} . 
}
\end{claim} 
\begin{proof} 
First, by the definition of ~$\Psi_2$ in~\eqref{eq:Psi2} we have 
\EQ{
\ang{ \Psi_2 \mid  \LL_{\Phi} \dot w } & = \ang{  -\p_t \dot \Phi + \De \Phi - \frac{1}{r^2} f( \Phi)  \mid  \LL_{\Phi}\dot w }  \\
&\quad  - \ang{  \frac{1}{r^2}  \Big ( f( \Phi +w) - f( \Phi) - f'( \Phi) w \Big) \mid  \LL_{\Phi} \dot w} 
}
The claim then follows from two estimates: 
 \EQ{ \label{eq:ptdotPhiest2} 
\Big| \langle  - \p_t \dot \Phi &+ \De \Phi - \frac{1}{r^2} f( \Phi)  \mid  \LL_{\Phi} \dot w \rangle \Big|   \lesssim \| \dot w \|_{H} \Big( \frac{b^2 \nu^{k-1}}{\la^2}  +  \frac{\nu^{2k-1}}{\la^2} + \frac{b^4}{\la^2}+\frac{a^4}{\la^2} \Big)   \\ 
  &\quad  +   \| \dot w \|_{H}  \| \bs w \|_{\HH} \Big( \frac{b^2}{\la^2} + \frac{a^2}{\la^2} + \frac{\nu^{2k}}{\la^2} \Big)  +  \| \dot w \|_{H}  \| \bs w \|_{\HH}^2 \Big(   \frac{b^2}{\la^2} + \frac{a^2\nu}{\la^2}   + \frac{\nu^k}{\la^2} \Big)  
 }
 and 
 \EQ{ \label{eq:w^2-est2} 
 \abs{ \ang{  \frac{1}{r^2}  \Big ( f( \Phi +w) - f( \Phi) - f'( \Phi) w \Big) \mid  \LL_{\Phi}\dot w} } &\lesssim  \frac{1}{\la} \| \p_r w \|_{H} \| w \|_{H} \| \dot w \|_{H}   \\
 & \quad  + \| \p_r w \|_{H}^2 \| w \|_H  \| \dot w \|_{H}  
}
Note that~\eqref{eq:w^2-est2} follows from  Lemma~\ref{l:w^2est}.  To see this recall that $\LL_{\Phi} = -\De  + \frac{k^2}{r^2} \cos 2 \Phi$ and thus integration by parts yields, 
\EQ{
&\Big\langle \frac{1}{r^2}  \Big ( f( \Phi +w) - f( \Phi) - f'( \Phi) w \Big) \mid \LL_{\Phi} \dot w \Big\rangle = \ang{ \frac{1}{r^2} \p_r   \Big ( f( \Phi +w) - f( \Phi) - f'( \Phi) w \Big) \mid  \p_r \dot w}  \\
& - \Big\langle  \frac{2}{r^3}   ( f( \Phi +w) - f( \Phi) - f'( \Phi) w ) \mid  \frac{\dot w}{r} \Big\rangle  + \Big\langle  \frac{1}{r^3}   ( f( \Phi +w) - f( \Phi) - f'( \Phi) w ) \mid k^2 \cos 2 \Phi \frac{\dot w}{r} \Big\rangle
} 
Now apply Cauchy Schwarz to all three terms on the right above and use~\eqref{eq:w^2L22} on the first term and~\eqref{eq:w^2L23} on the last two. 

It remains to prove~\eqref{eq:ptdotPhiest2}. The proof is similar to the proof of~\eqref{eq:ptdotPhiest}. Using~\eqref{eq:ptdotPhi} we have 
\EQ{
\langle & -\p_t \dot \Phi  + \De \Phi - \frac{1}{r^2} f( \Phi)  \mid \LL_{\Phi} \dot w \rangle = I + II  + III + IV 
}
where term $I$ is given by
\EQ{
I = & -(b'+ \gamma_k \frac{\nu^k}{\la}) \ang{ \La Q_{\U\la} \mid \LL_\Phi \dot w}  - (a' + \ti \gamma_k \frac{\nu^k}{\mu}) \ang{ \La Q_{\U \mu} \mid \LL_\Phi \dot w} \\
& + b \frac{(b+\la')}{\la} \ang{ \La_0 \La Q_{\U \la} \mid  \LL_{\Phi} \dot w}   +   a\frac{(\mu'- a)}{\mu} \ang{ \La_0 \La Q_{\U \mu}\mid \LL_{\Phi} \dot w}  ,
}
and terms $II, III, IV$ are identical to the analogous  terms $II, III, IV$ in the proof~\eqref{eq:ptdotPhiest} except now the pairing is with $\LL_\Phi \dot w$ instead of with $\dot w$. 
In contrast to~\eqref{eq:ptdotPhiest} the first two terms on the right-hand side of term $I$ above do not vanish identically, but they do vanish to top order. Indeed, since $\LL_\la \La Q_{\U \la} =0$  for the first term in $I$ we have 
\EQ{
 \ang{ \La Q_{\U\la} \mid \LL_\Phi \dot w}  &=  \ang{ \La Q_{\U\la} \mid \LL_\la \dot w}  +  \ang{ \La Q_{\U\la} \mid (\LL_\Phi- \LL_\la)  \dot w}  \\
 & =  \ang{ \La Q_{\U\la} \mid (\LL_\Phi- \LL_\la) \dot w}   =  \ang{ \La Q_{\U\la} \mid \frac{k^2}{r^2} ( \cos 2 \Phi - \cos 2Q_\la)  \dot w}
 }
 Using~\eqref{eq:cos2Phi1} we then have 
 \EQ{
 \Big|(b'+ \gamma_k \frac{\nu^k}{\la})& \ang{ \La Q_{\U\la} \mid \LL_\Phi \dot w} \Big|  \lesssim \frac{ \abs{b'+ \gamma_k \frac{\nu^k}{\la}}}{\la}  \| r^{-1} \dot w \|_{L^2}  \| r^{-1} \La Q_\la ( \cos 2\Phi - \cos 2 Q_\la) \|_{L^2} \\
 &\lesssim   \frac{ \abs{b'+ \gamma_k \frac{\nu^k}{\la}}}{\la}  \| \dot w \|_{H}  \Big( \| r^{-1} [\La Q_\la]^2 \La Q_\mu  \|_{L^2} + b^2  \| r^{-1} [\La Q]^2 A \|_{L^2}  \\
&\quad + \nu^k \| r^{-1} [\La Q]^2 B \|_{L^2}    + a^2  \| r^{-1} [\La Q_\la]^2 A_\mu \|_{L^2} + \nu^k \|r^{-1} [\La Q_\la]^2 \ti \B_\mu \|_{L^2}   \\
&\quad  + \| r^{-1} \La Q_\la \La Q_\mu^2 \|_{L^2}   + b^4 \| r^{-1} \La Q A^2 \|_{L^2} + \nu^{2k} \|r^{-1} \La Q B^2 \|_{L^2} \\
& \quad +  a^4 \|r^{-1} \La Q_\la [A_\mu]^2 \|_{L^2} +\nu^{2k} \| r^{-1} \La Q_\la [\ti B_\mu]^2 \|_{L^2} \Big)  \\
& \lesssim\frac{ \abs{b'+ \gamma_k \frac{\nu^k}{\la}}}{\la} \Big( \nu^k + b^2 + a^2 \nu^k\Big)  \| \dot w \|_{H}
}
Similarly, for the second term in $I$ we use~\eqref{eq:cos2Phi2} to obtain,  
\EQ{
 \abs{(a' + \ti \gamma_k \frac{\nu^k}{\mu}) \ang{ \La Q_{\U \mu} \mid \LL_\Phi \dot w} } &\lesssim \frac{\abs{a' + \ti \gamma_k \frac{\nu^k}{\mu}}}{\mu}  \| \dot w \|_H \| r^{-1} ( \cos 2 \Phi - \cos 2 Q_\mu)  \La Q_\mu \|_{L^2}  \\
 & \lesssim  \frac{\abs{a' + \ti \gamma_k \frac{\nu^k}{\mu}}}{\mu} \Big( \nu^k + b^2\nu^{k-2} + a^2\Big)\| \dot w \|_{H}
}
For the third term in $I$ we have 
\EQ{
\abs{b \frac{(b+\la')}{\la} \ang{ \La_0 \La Q_{\U \la} \mid  \LL_{\Phi} \dot w} } & \lesssim \frac{b}{\la^2} \abs{ b + \la'} \| \La_0 \La Q \|_H \| \dot w \|_H
}
The fourth term is handled similarly. 

The terms $II, III, IV$ are estimated in an identical fashion to the corresponding terms in the  proof of ~\eqref{eq:ptdotPhiest}.  We arrive at the estimate, 
\begin{align}  \label{eq:281} 
\Big|\langle & -\p_t \dot \Phi  + \De \Phi - \frac{1}{r^2} f( \Phi)  \mid \LL_{\Phi} \dot w \rangle  \Big| 
  \lesssim \| \dot w \|_{H} \Bigg( \frac{b^2 \nu^{k-1}}{\la^2}  +  \frac{\nu^{2k-1}}{\la^2} + \frac{b^4}{\la^2}+\frac{a^4}{\la^2}    \\ 
& \quad  +  \abs{b'+ \gamma_k \frac{\nu^k}{\la}} \Big(\frac{\nu^k}{\la} +  \frac{b^2}{\la} + \frac{a^2 \nu^k}{\la} + \frac{a^4 \nu^2}{\la^2}  \Big)  + \abs{a' + \ti \gamma_k \frac{\nu^k}{\mu}}   \Big(      \frac{\nu^{k+1}}{\la}  + \frac{b^2\nu^{k-1}}{\la}  + \frac{a^2\nu}{\la} \Big)  \\
&\quad +\abs{ b + \la'} \Big(  \frac{b}{\la^2}     \Big)  + \abs{a - \mu'} \Big(  \frac{a \nu^2}{\la^2}    \Big) \Bigg)
\end{align} 
where we note that the application of the estimates in Lemma~\ref{l:fest2} used to estimate terms from $IV$ is straightforward after noting that 
\EQ{
\ang{ r^{-2} h \mid  \LL_\Phi \dot w} =   -\ang{ \p_r h \mid \p_r \dot w} +  2\ang{ r^{-3} h \mid  \p_r \dot w}  + \ang{  r^{-3} h \mid k^2 \cos 2\Phi r^{-1} \dot w}
}
and thus, 
\EQ{
\abs{\ang{ r^{-2} h \mid  \LL_\Phi \dot w}} \lesssim \Big( \| r^{-3} h \|_{L^2} +  \| r^{-2} \p_r h \|_{L^2} \Big) \| \dot w \|_{H} 
}
Finally we apply Lemma~\ref{l:modc2} to the right-hand side of~\eqref{eq:281} to obtain, 
\EQ{
\Big|\langle & -\p_t \dot \Phi  + \De \Phi - \frac{1}{r^2} f( \Phi)  \mid \LL_{\Phi} \dot w \rangle  \Big|  
  \lesssim \| \dot w \|_{H} \Big( \frac{b^2 \nu^{k-1}}{\la^2}  +  \frac{\nu^{2k-1}}{\la^2} + \frac{b^4}{\la^2}+\frac{a^4}{\la^2} \Big)   \\ 
  &\quad  +   \| \dot w \|_{H}  \| \bs w \|_{\HH} \Big( \frac{b^2}{\la^2} + \frac{a^2}{\la^2} + \frac{\nu^{2k}}{\la^2} \Big)  +  \| \dot w \|_{H}  \| \bs w \|_{\HH}^2 \Big(   \frac{b^2}{\la^2} + \frac{a^2\nu}{\la^2}   + \frac{\nu^k}{\la^2} \Big)  
}
which proves~\eqref{eq:ptdotPhiest2} as claimed. 
\end{proof} 

\begin{claim}\label{c:Psi1L2w} The following estimate holds true. 
\EQ{ \label{eq:Psi1-L2w}
\abs{\ang{ \Psi_1 \mid \LL_{\Phi}^2  w } } &\lesssim  \Big(  \frac{a^5}{\lam^2}   + \frac{b^5}{ \lam^2} + \frac{(\abs{a}+ \abs{b}) \nu^{2k-1}}{\lam^2} + \frac{ \abs{a} b^2 \nu^{k}}{\lam^2}  \Big)  \|  w \|_{H^2} \\
& \quad + \Big(  \frac{a^2 + b^2 + \nu^{2k}}{\lam^2} \Big)  \| \bs w \|_{\HH}  \| w \|_{H^2}  + \Big( \frac{\abs{a} + \abs{b}}{\lam^2} \Big) \| \bs w \|_{\HH}^2 \|  w \|_{H^2}.
}
\end{claim} 
\begin{proof}[Proof of Claim~\ref{c:Psi1L2w}] 
The argument here is nearly identical to the proof of Claim~\ref{c:Psi1Lw}. Using~\eqref{eq:Psi1w} 
we have 
\EQ{ \label{eq:Psi1L2w} 
 \ang{ \Psi_1 \mid  \calL_\Phi^2 w} &=   (b+ \la') \ang{ \La Q_{\U \la} \mid \LL_\Phi^2 w}+ (a - \mu')\ang{ \La Q_{\U \mu} \mid  \LL_\Phi^2 w} \\
 & \quad  +  b^2( b + \la') \ang{\La A_{\U \la}\mid  \LL_\Phi^2 w}    - 2 b \la ( b' + \gamma_k \frac{\nu^{k}}{\la}) \ang{A_{\U \la}  \mid  \LL_\Phi^2 w}   \\
 & \quad +  \nu^k( b + \la') \ang{\La B_{\U \la}\mid  \LL_\Phi^2 w} - k \nu^k( b + \la')  \ang{B_{ \U\la} \mid  \LL_\Phi^2 w}  \\
 &\quad - k \nu^{k+1}( a - \mu') \ang{ B_{ \U\la} \mid \LL^2_{\Phi}} + a^2(  a - \mu')\ang{ \La A_{\U \mu}\mid  \LL_\Phi^2 w} \\ 
 &\quad   + 2a \mu ( a' +  \ga_k \frac{\nu^k}{\mu}) \ang{A_{\U\mu}  \mid  \LL_\Phi^2 w} + \nu^k(   a- \mu') \ang{\La \ti B_{\U \mu}\mid  \LL_\Phi^2 w} \\
 &\quad   + k \nu^{k-1} (  b + \la') \ang{\ti B_{\U \mu} \mid  \LL_\Phi^2 w}+ k \nu^{k} (  a- \mu') \ang{ \ti B_{\U \mu}  \mid \LL^2_\Phi }
}
To estimate the first two terms on the right-hand side above we exploit the fact that $\LL_\la \La Q_{\U \la} =0$ and $\LL_\mu \La Q_{\U \mu} = 0$. To handle the first term,  we write 
\EQ{
\LL_{\Phi} h  = (\LL_\Phi - \LL_\la) h+  \LL_\la h  = \frac{k^2}{r^2} ( \cos 2 \Phi - \cos 2 Q_\la) h + \LL_\la h 
}
Then we have, 
\EQ{ \label{eq:Psi1L2w-t1} 
(b+ \la') \ang{ \La Q_{\U \la} \mid \LL_\Phi^2 w}  &= (b+ \la') \ang{ \La Q_{\U \la} \mid \LL_\la \LL_\Phi w} + (b+ \la') \ang{ \La Q_{\U \la} \mid (\LL_{\Phi} - \LL_\la) \LL_\Phi w} \\
& = (b+ \la') \ang{ \La Q_{\U \la} \mid\frac{k^2}{r^2} ( \cos 2 \Phi  - \cos 2 Q_\la)  \LL_\Phi w}
}
We now estimate 
\EQ{
\abs{ \ang{ \La Q_{\U \la} \mid\frac{k^2}{r^2} ( \cos 2 \Phi  - \cos 2 Q_\la) \LL_\Phi w}} &\lesssim \frac{1}{\la} \| \LL_\Phi w\|_{L^2} \left( \int_0^\infty \La Q_{\la}^2 \big(\cos 2 \Phi  - \cos 2 Q_\la\big)^2 \, \frac{\ud r}{r^3} \right)^{\frac{1}{2}} \\
& \lesssim  \frac{ \nu^k + b^2 + a^2}{\la^2}  \|  w \|_{H^2} 
}
where the last line follows from an explicit computation using~\eqref{eq:cos2Phi1} from Lemma~\ref{l:cos2Phi} and also using Lemma~\ref{l:wHH2}. Plugging this into the previous line gives, 
\EQ{
\abs{(b+ \la') \ang{ \La Q_{\U \la} \mid \LL_\Phi^2 w}}  & \lesssim  \frac{ \nu^k + b^2 + a^2}{\la^2} \abs{b+ \la'}  \|  w \|_{H^2} 
}
Similarly, we have 
\EQ{
\abs{(a - \mu')\ang{ \La Q_{\U \mu} \mid  \LL_\Phi^2 w} }& \lesssim  \frac{ \nu^k + b^2 + a^2}{\la^2}  \nu^2 \abs{a- \mu'}  \|  w \|_{H^2} 
}
For the remaining terms we argue as follows. With $F = A, B, \ti B$, we estimate 
\EQ{
\abs{\ang{ F_{\U \la} \mid   \LL_\Phi^2 w}} &\lesssim \frac{1}{\la} \| \LL_\Phi F_{\la}\|_{L^2} \| \LL_\Phi w \|_{L^2}  \lesssim  \frac{1}{\la^2} \| \LL_{\Phi} F \|_{L^2}  \|  w \|_{H^2} \lesssim \frac{1}{\la^2}   \|  w \|_{H^2} 
}
and similarly, 
\EQ{
\abs{\ang{ F_{\U \mu} \mid   \LL_\Phi^2 w}} &\lesssim \frac{1}{\mu} \| \LL_\Phi F_{\mu}\|_{L^2} \| \LL_\Phi w \|_{L^2}   \lesssim  \frac{1}{\mu^2} \| \LL_{\Phi} F \|_{L^2}  \|  w \|_{H^2}  \lesssim \frac{\nu^2}{\la^2}   \|  w \|_{H^2} 
}
Using all of these estimates on the right-hand side of~\eqref{eq:Psi1L2w} yields 
\EQ{
\abs{ \ang{\Psi_1 \mid \LL^2_\Phi w} } &\lesssim  \frac{ \nu^k + b^2 + a^2}{\la^2}\Big( \abs{b+ \la'}  + \nu^2\abs{a - \mu'} \Big) \|  w \|_{H^2}  \\
& \quad + \frac{\abs{b}}{\la} \abs{ b+ \ga_k \frac{\nu^k}{\la}}  \|  w \|_{H^2}   + \frac{\abs{a} \nu}{\la} \abs{a' +  \ga_k \frac{\nu^k}{\mu}}  \|  w \|_{H^2} 
}
Applying Lemma~\ref{l:modc2} to the right above gives, 
\EQ{
\Big| \langle\Psi_1& \mid \LL^2_\Phi w\rangle \Big| \lesssim  \frac{ \nu^k + b^2 + a^2}{\la^2}\Big((\abs{a} + \abs{b} ) \| \bs w \|_{\HH} +( \abs{a} + \abs{b})\nu^{2k-1}  + ( \abs{a}^5 + \abs{b}^5) \nu^{k-2} \Big)\|  w \|_{H^2} \\
& \quad + \frac{\abs{b}}{\la} \Big(\frac{\abs{a}  + \abs{b}}{\la}   \| \dot w \|_{L^2}  +  \frac{1}{\la} \| \bs w \|_{\HH}^2 + \frac{\nu^{2k-o(1)}}{\lam}  +  \frac{b^4}{\lam} + \frac{a^4}{\lam} \Big) \|  w \|_{H^2} \\
&\quad + \frac{\abs{a} \nu}{\la}\Big(\frac{\abs{a}\nu  + \abs{b}\nu}{\la}   \| \dot w \|_{L^2}  +  \frac{\nu}{\la} \| \bs w \|_{\HH}^2 + \frac{\nu^{2k-1}}{\lam} + \frac{b^2 \nu^{k-1}}{\lam} +  \frac{b^4\nu}{\lam} + \frac{a^4\nu}{\lam} \Big)  \|  w \|_{H^2} \\
& \lesssim  \Big(  \frac{\abs{a}^5}{\lam^2}   + \frac{\abs{b}^5}{ \lam^2} + \frac{(\abs{a}+ \abs{b}) \nu^{2k-1}}{\lam^2} + \frac{ \abs{a} b^2 \nu^{k}}{\lam^2}  \Big)  \| w \|_{H^2} \\
& \quad + \Big(  \frac{a^2 + b^2 + \nu^{2k}}{\lam^2} \Big)  \| \bs w \|_{\HH}  \| w \|_{H^2}  + \Big( \frac{\abs{a} + \abs{b}}{\lam^2} \Big) \| \bs w \|_{\HH}^2 \| w \|_{H^2} 
}
as claimed. 
\end{proof} 
%

\begin{claim}\label{c:commwdot2} 
The following estimate holds true. 
\EQ{ \label{eq:commwdot2} 
\frac{1}{2} \ang{ \dot w \mid [\p_t, \LL_\Phi] \dot w} & = - \frac{b}{\la} \ang{ \La_0 \dot w \mid P_{\la} \dot w} + \frac{b}{\lam} \ang{ \dot w \mid P_\lam \dot w}   \\
&\quad + O \Bigg(\Big( \frac{b \la^{\frac{k}{2}} + a\nu + b^2 + \nu^k}{\la} \Big) \| \dot w \|_{H}^2  + \Big( \frac{ \abs{a} + \abs{b}}{\lam}  \Big)  \| \dot w \|_{H}^2 \| \bs w \|_{\HH}  \\
&\quad \qquad + \Big( \frac{ \abs{a} + \abs{b}}{\lam}  \Big) \| \dot w \|_{H}^2 \| \bs w \|_{\HH}^2 \Bigg) . 
}
\end{claim} 
\begin{proof} The proof is identical to the proof of Claim~\ref{c:vir}, after replacing $w$ with $\dot w$. 
%
\end{proof}

\begin{claim}\label{c:commw2} 
The following holds true. 
\EQ{
\frac{1}{2} \ang{  w \mid [\p_t, \LL_\Phi^2]  w}  &=  \frac{b}{\la}  \ang{w \mid \Ks_\la w}   +\frac{b}{\la} \ang{ w \mid P_\la^2 w}  \\
&  +  \frac{b}{\la} \ang{ w \mid  \La P_\la\LL_0 w  -  \p_r (\La P_\la) \p_r w  +  \frac{1}{2} (-\De \La P_\la) w +  P_\la \La P_\la w  }  \\
& + O \Bigg(  \Big( \frac{b \la^{\frac{k}{2}} + a\nu + b^2 + \nu^k}{\la} \Big) \| w \|_{H^2}^2  \\
&+ \Big( \frac{ \abs{a} + \abs{b}}{\lam}  \Big) \|  w \|_{H^2}^2 \| \bs w \|_{\HH}  + \Big( \frac{ \abs{a} + \abs{b}}{\lam}  \Big)  \| w \|_{H^2}^2 \| \bs w \|_{\HH}^2  \Bigg) . 
}
\end{claim} 
\begin{proof} 
First we claim the estimate 
\EQ{ \label{eq:commPhi-Qla} 
\Big| \frac{1}{2} \ang{  w \mid [\p_t, \LL_\Phi^2]  w} & -   \frac{1}{2} \ang{  w \mid [\p_t, \LL_\lam^2]  w} \Big| \lesssim \Big( \frac{b \la^{\frac{k}{2}} + a\nu + b^2 + \nu^k}{\la} \Big) \| w \|_{H^2}^2  \\
&+ \Big( \frac{ \abs{a} + \abs{b}}{\lam}  \Big) \|  w \|_{H^2}^2 \| \bs w \|_{\HH}  + \Big( \frac{ \abs{a} + \abs{b}}{\lam}  \Big)  \| w \|_{H^2}^2 \| \bs w \|_{\HH}^2 
}
To see this note that for a general function $\Psi(t)$, 
\EQ{
\LL_\Psi^2 w &= ( -\De + \frac{1}{r^2} f'(\Psi)) ( -\De w+ \frac{1}{r^2} f'(\Psi) w) \\
& = \De^2 w - \frac{2}{r^2} f'(\Psi) \De w - 2 \p_r( \frac{1}{r^2} f'(\Psi)) \p_r w + \Big( - \De( \frac{1}{r^2} f'(\Psi)) + \frac{1}{r^4} (f'(\Psi))^2\Big) w
}
and hence, 
\EQ{
[\p_t, \LL^2_\Phi] w -   [\p_t, \LL_\lam^2]  w&= - \frac{2}{r^2} ( f''(\Phi) \p_t \Phi    -  f''(Q_\la) \p_t Q_\la) \De w \\
& \quad  -  2 \p_r\Big( \frac{1}{r^2} ( f''(\Phi) \p_t \Phi  -  f''(Q_\la) \p_t Q_\la)\Big) \p_r w  \\
&\quad - \De\big( \frac{1}{r^2} (f''(\Phi) \p_t \Phi  -  f''(Q_\la) \p_t Q_\la)\big) w \\
&\quad  + 2\frac{1}{r^4} \Big(f'(\Phi) f''(\Phi) \p_t \Phi  -  f'(Q_\la) f''(Q_\la) \p_t Q_\la \Big) w 
}
The estimate now follows from a nearly identical argument as the one used to prove~\eqref{eq:comm-est1}. 

Next we compute $[  \p_t, \LL^2_\la] w$. First we express $\LL^2_\la$ as follows. To ease notation we recall the operator $P_\lam$ as in~\eqref{eq:Plam} and we recall that  the operator $ \Ks_{\lam}$ as in~\eqref{eq:K-def} is as follows. 
\EQ{
\LL_\la^2 w &=  \Big( - \De + \frac{k^2}{r^2} + P_\lam\Big)^2w \\
&  =  \Big( -\De + \frac{k^2}{r^2}\Big)^2w  + 2 P\big( -\De + \frac{k^2}{r^2}\big) w  - 2 \p_r P \p_r w  + (-\De P) w + P^2 w  \\
&=: \Big(-\De + \frac{k^2}{r^2}\Big)^2  w + \Ks_\la w 
}
Then, $
[ \p_t, \LL_\lam^2]w  = [\p_t, \Ks_\lam]w  
$ 
and a straightforward computation reveals that 
\EQ{
\frac{1}{2} \ang{ w \mid [\p_t, \LL_\la^2] w} &= \frac{-\la'}{\la}  \ang{w \mid \Ks_\la w}   +\frac{-\la'}{\la} \ang{ w \mid P_\la^2 w}  \\
&  \quad +  \frac{-\la'}{\la} \ang{ w \mid  \La P_\la\LL_0 w  -  \p_r (\La P_\la) \p_r w  +  \frac{1}{2} (-\De \La P_\la) w +  P_\la \La P_\la w  } \\
& = \frac{b}{\la}  \ang{w \mid \Ks_\la w}   +\frac{b}{\la} \ang{ w \mid P_\la^2 w}  \\
& \quad  +  \frac{b}{\la} \ang{ w \mid  \La P_\la\LL_0 w  -  \p_r (\La P_\la) \p_r w  +  \frac{1}{2} (-\De \La P_\la) w +  P_\la \La P_\la w  } \\
&\quad - \frac{ b+ \la'}{\la}  \ang{w \mid \Ks_\la w}   -\frac{b+\la'}{\la} \ang{ w \mid P_\la^2 w}  \\
&  \quad -  \frac{b+\la'}{\la} \ang{ w \mid  \La P_\la\LL_0 w  -  \p_r (\La P_\la) \p_r w  +  \frac{1}{2} (-\De \La P_\la) w +  P_\la \La P_\la w  }
}
To conclude we use Lemma~\ref{l:mod2} to estimate the last three lines above as follows, 
\EQ{
\abs{\frac{ b+ \la'}{\la}  \ang{w \mid \Ks_\la w}} & + \abs{\frac{ b+ \la'}{\la}  \ang{ w \mid P_\la^2 w}}   \\
& + \abs{ \frac{b+\la'}{\la} \ang{ w \mid  \La P_\la\LL_0 w  -  \p_r (\La P_\la) \p_r w  +  \frac{1}{2} (-\De \La P_\la) w +  P_\la \La P_\la w  }}\\
&\lesssim \frac{1}{\lam} \Big(\abs{a} + \abs{b} ) \| \bs w \|_{\HH} +( \abs{a} + \abs{b})\nu^{2k-1}  + ( \abs{a}^5 + \abs{b}^5) \nu^{k-2} \Big) \| w \|_{H^2}^2 .
}
This completes the proof. 
\end{proof}

To conclude, we combine  Claim~\ref{c:Psi2Ldotw}, Claim~\ref{c:Psi1L2w}, Claim~\ref{c:commwdot2}, and Claim~\ref{c:commw2} to deduce that 
\EQ{
\Es_2'(t) &= - \frac{b}{\la} \ang{ \La_0 \dot w \mid P_{\la} \dot w} + \frac{b}{\lam} \ang{ \dot w \mid P_\lam \dot w}  \\
& \quad + \frac{b}{\la}  \ang{w \mid \Ks_\la w}   +\frac{b}{\la} \ang{ w \mid P_\la^2 w}  \\
& \quad  +  \frac{b}{\la} \ang{ w \mid  \La P_\la\LL_0 w  -  \p_r (\La P_\la) \p_r w  +  \frac{1}{2} (-\De \La P_\la) w +  P_\la \La P_\la w  }  \\
&\quad  + O \Bigg( \Big( \frac{b^2 \nu^{k-1}}{\la^2}  +  \frac{\nu^{2k-1}}{\la^2} + \frac{b^4}{\la^2}+\frac{a^4}{\la^2} \Big)  \| \bs w \|_{\HH^2}\Bigg)  +  O \Bigg(  \Big( \frac{b \la^{\frac{k}{2}} + a\nu + b^2 + \nu^k}{\la} \Big) \| \bs w \|_{\HH^2}^2     \Bigg)   \\ 
  &\quad  + O \Bigg(  \| \bs w \|_{\HH^2}  \| \bs w \|_{\HH} \Big( \frac{b^2}{\la^2} + \frac{a^2}{\la^2} + \frac{\nu^{2k}}{\la^2} \Big)  +  \| \bs w \|_{\HH^2}  \| \bs w \|_{\HH}^2  \Big( \frac{\abs{a} + \abs{b} + \nu^k}{\lam^2} \Big) \Bigg) \\
  &\quad   +O \Bigg( \frac{1}{\la} \| \bs w \|_{\HH^2}^2 \| w \|_{H}     + \| \bs w \|_{\HH^2}^3 \| w \|_H    +  \frac{ \abs{a} + \abs{b}}{\lam}   \| \bs w \|_{\HH^2}^2 \| \bs w \|_{\HH}^2 \Bigg),
}
which completes the proof. 
\end{proof}

\begin{proof}[Proof of Lemma~\ref{l:V2'}]

Recall that the second order virial correction is defined as
\EQ{
b \Vs_{2, \la} = b\ang{ \A_0(\la) w \mid \LL_\la \dot w}  - 2 b \ang{\A_0(\la) \dot w \mid \LL_\la w}.
}
In the computations below we will make use of the following version of the equation for $\bs w$, 
\EQ{ \label{eq:w-eq5} 
\p_t w &=  \dot w + \Psi_1, \\
\p_t \dot w &= -\LL_\lam w +  \ti \Psi_2,
}
where $\Psi_1$ is defined as usual as in~\eqref{eq:Psi1} and  $\ti \Psi_2$ is defined as 
\EQ{ \label{eq:tiPsi2-def2} 
 \ti \Psi_2 = (-\p_t \dot \Phi + \De \Phi -\frac{1}{r^2} f( \Phi)) - \frac{1}{r^2}( f( \Phi + w) - f( \Phi) - f'(Q_\lam) w).
}
We compute, 
\EQ{
(b \Vs_{2, \la} )' &= b'\ang{ \A_0(\la) w \mid \LL_\la \dot w}  - 2 b' \ang{\A_0(\la) \dot w \mid \LL_\la w} \\
& \quad + \frac{b \la'}{\lam} \ang{( \lam \p_ \lam \A_0(\la)) w \mid \LL_\la \dot w} - 2 \frac{b \la'}{\la} \ang{ (\la\p_\la\A_0(\la)) \dot w \mid \LL_\la w}  \\
& \quad +  b\ang{ \A_0(\la) w \mid [\p_t, \LL_\la] \dot w}  - 2 b \ang{\A_0(\la) \dot w \mid [\p_t, \LL_\la] w} \\
& \quad + b\ang{ \A_0(\la) \p_t w \mid \LL_\la \dot w} - 2 b \ang{\A_0(\la) \p_t \dot w \mid \LL_\la w}    \\
& \quad + b\ang{ \A_0(\la)  w \mid  \LL_\la \p_t  \dot w} - 2 b \ang{\A_0(\la) \dot w \mid \LL_\la \p_t w}   .
}
Using the equation~\eqref{eq:w-eq5} and noting that $\ang{\A_0(\la) \LL_\lam w \mid \LL_\la w}=0$ the above can be expanded and rearranged as 
\EQ{ \label{eq:vir2-ident} 
(b \Vs_{2, \la} )' &= b'\ang{ \A_0(\la) w \mid \LL_\la \dot w}  - 2 b' \ang{\A_0(\la) \dot w \mid \LL_\la w} \\
& \quad + \frac{b \la'}{\lam} \ang{( \lam \p_ \lam \A_0(\la)) w \mid \LL_\la \dot w} - 2 \frac{b \la'}{\la} \ang{ (\la\p_\la\A_0(\la)) \dot w \mid \LL_\la w}  \\
& \quad +  b\ang{ \A_0(\la) w \mid [\p_t, \LL_\la] \dot w}  - 2 b \ang{\A_0(\la) \dot w \mid [\p_t, \LL_\la] w} \\
& \quad +   b\ang{ \A_0(\la) \Psi_1 \mid \LL_\la \dot w}- 2 b \ang{\A_0(\la) \dot w \mid \LL_\la \Psi_1}  \\
& \quad + b\ang{ \A_0(\la)  w \mid  \LL_\la \ti \Psi_2}  - 2 b \ang{\A_0(\la) \ti \Psi_2 \mid \LL_\la w} \\
& \quad - b\ang{\A_0(\la) \dot w \mid \LL_\la \dot w}     - b\ang{ \A_0(\la)  w \mid  \LL_\la^2 w}. 
}
We estimate the second to last term above as follows.  As usual we write.  
\EQ{
\LL_\la = (-\De + \frac{k^2}{r^2}) + \frac{1}{r^2}( f'(Q_\la) - k^2) =: \LL_0 + P_{\la} .
}
 Using~\eqref{eq:A-pohozaev} along with the estimate~\eqref{eq:vir-new}  we have 
\EQ{
 - b\ang{\A_0(\la) \dot w \mid \LL_\la \dot w}  &=  + b\ang{\A_0(\la) \dot w \mid -\LL_0 \dot w}  - b\ang{\A_0(\la) \dot w \mid P_\la \dot w} \\
 & \le c_0 \frac{b}{\la} \| \dot w \|_{H}^2  - \frac{b}{\la} \int_0^{R \la}\Big( ( \p_r \dot w)^2 + \frac{k^2}{r^2} \dot w^2  \Big) \, \rdr  - \frac{b}{\lam} \ang{\Lam_0 \dot w \mid P_\la \dot w}
}

Next, consider the last term in~\eqref{eq:vir2-ident}. Writing as usual $\LL_\la^2 = \LL_0^2 + \Ks_\lam$ we have  
\EQ{
 - b\ang{ \A_0(\la)  w \mid  \LL_\la^2 w} &=  - b\ang{ \A_0(\la)  w \mid  \LL_0^2 w}  - b\ang{ \A_0(\la)  w \mid  \Ks_\la w} \\
 & =  - b\ang{ \A_0(\la)  w \mid  \LL_0^2 w}  - b\ang{ (\A_0(\la) - \lam^{-1} \La_0)  w \mid  \Ks_\la w} - \frac{b}{\lam}\ang{  \Lam_0  w \mid  \Ks_\la w} 
}
Using the Pohozaev-type inequality~\eqref{eq:A-pohozaev2} along with the estimate~\eqref{eq:vir-new2} we have 
\EQ{ \label{eq:6-term-vir2} 
 - b\ang{ \A_0(\la)  w \mid  \LL_\la^2 w} & \le c_0 \frac{b}{\lam} \| w \|_{H^2}^2 -  2\frac{b}{\lam} \int_0^{R\lam} ( \LL_0 w)^2  \, r\, \ud r - \frac{b}{\lam}\ang{  \Lam_0  w \mid  \Ks_\la w}
}
Consider the last term on the right above. Since $\Ks_\lam$ is symmetric, it follow that 
\EQ{
-\ang{  \Lam  w \mid  \Ks_\la w}  = +\frac{1}{2} \ang{ w \mid [\Lam, \Ks_\lam] w}  +  \ang{ w \mid \Ks_\lam w}
}
and hence, 
\EQ{
- \frac{b}{\lam}\ang{  \Lam_0  w \mid  \Ks_\la w} &= - \frac{b}{\lam} \ang{  \Lam  w \mid  \Ks_\la w}  - \frac{b}{\lam} \ang{ w \mid \Ks w}  =  \frac{1}{2} \frac{b}{\lam} \ang{ w \mid [\Lam, \Ks_\lam] w}
}
A computation gives that 
\EQ{
\frac{1}{2} [ \La, \Ks] w &=   [ \La, P\LL_0] w  - [\La,  \p_r P \p_r] w  + \frac{1}{2} [\La, (-\De P)] w + \frac{1}{2}[\La,  P^2] w   \\
&  = \La P \LL_0 w  - \p_r (\La P) \p_r w   + \frac{1}{2}(- \De \La P) w +  P \La P w  - \Ks_\la w  + P^2 w 
}
And thus we arrive at the identity, 
\EQ{
- \frac{b}{\lam}\ang{  \Lam_0  w \mid  \Ks_\la w} &=\frac{1}{2} \frac{b}{\lam} \ang{ w \mid [\Lam, \Ks_\lam] w} \\
& = - \frac{b}{\lam} \ang{ w \mid \Ks_\lam w}  + \frac{b}{\lam} \ang{w \mid P_\lam^2 w}   \\
&\quad + \frac{b}{\lam} \ang{w \mid \La P_\lam \LL_0 w -  \p_r (\La P_\lam) \p_r w + \frac{1}{2}(- \De \La P_\lam) w +  P_\lam \La P_\lam w} 
}
Combing the above we arrive at the following estimate for the last term in~\eqref{eq:vir2-ident}, 
\EQ{ \label{eq:7-term-vir2}
 - b\ang{ \A_0(\la)  w \mid  \LL_\la^2 w} & \le c_0 \frac{b}{\lam} \| w \|_{H^2}^2 -  2\frac{b}{\lam} \int_0^{R\lam} ( \LL_0 w)^2  \, r\, \ud r \\
&\quad  - \frac{b}{\lam} \ang{ w \mid \Ks_\lam w}  + \frac{b}{\lam} \ang{w \mid P_\lam^2 w}   \\
&\quad + \frac{b}{\lam} \ang{w \mid \La P_\lam \LL_0 w -  \p_r (\La P_\lam) \p_r w + \frac{1}{2}(- \De \La P_\lam) w +  P_\lam \La P_\lam w}
}

Next, we treat the remaining terms in~\eqref{eq:vir2-ident}. We begin with terms on the first line. The first term can be rewritten as, 
\EQ{
b'\ang{ \A_0(\la) w \mid \LL_\la \dot w} = b'\ang{ \p_r \A_0(\lam) w \mid \p_r \dot w} + b'\ang{ r^{-1} \A_0(\lam) w \mid r^{-1} f'(Q_\lam) \dot w} 
}
Using the above along with the estimate~\eqref{eq:prA0} we deduce the estimate, 
\EQ{
\abs{b'\ang{ \A_0(\la) w \mid \LL_\la \dot w} } \lesssim  \frac{\nu^{k}}{\lam} ( \| w \|_{H^2} + \frac{1}{\lam} \|w \|_{H})  \|\dot w\|_{H} +  \abs{b' + \gamma_k \frac{\nu^{k}}{\lam}}  ( \| w \|_{H^2} + \frac{1}{\lam} \|w \|_{H})  \|\dot w\|_{H}
}
Applying the estimate~\eqref{eq:b'est} to the right above and arguing similarly for the second term on the right hand side of~\eqref{eq:vir2-ident} we conclude that 
\EQ{
\Big| b'\langle \A_0(\la) w &\mid \LL_\la \dot w \rangle - 2 b' \ang{\A_0(\la) \dot w \mid \LL_\la w} \Big| \lesssim  \frac{\nu^{k}}{\lam} ( \| w \|_{H^2} + \frac{1}{\lam} \|w \|_{H})  \|\dot w\|_{H} \\
&  \quad +\Big( \frac{\abs{a}  + \abs{b}}{\la}   \| \dot w \|_{L^2}  +  \frac{1}{\la} \| \bs w \|_{\HH}^2 + \frac{\nu^{2k-o(1)}}{\lam}  +  \frac{b^4}{\lam} + \frac{a^4}{\lam}\Big) ( \| w \|_{H^2} + \frac{1}{\lam} \|w \|_{H})  \|\dot w\|_{H} \\
& \lesssim  \frac{ \nu^k + a^4 +b^4}{\lam} \| \bs w \|_{\HH^2}^2   +  \frac{\abs{a}  + \abs{b}}{\la}   \| \dot w \|_{L^2}  \| \bs w \|_{\HH^2}^2+ \frac{1}{\lam} \| \bs w \|_{\HH}^2 \| \bs w \|_{\HH^2}^2 \\
& \quad +  \frac{\nu^k + a^4 + b^4}{\lam^2}  \| w \|_H \| \bs w \|_{\HH^2}  + \frac{\abs{a}  + \abs{b}}{\la^2}   \| \bs w \|_{\HH}^2 \| \bs w \|_{\HH^2}  + \frac{1}{\lam^2} \|\bs w \|_{\HH}^3 \| \bs w \|_{\HH^2}
}
We handle the second line in~\eqref{eq:vir2-ident} using an identical argument, based again on~\eqref{eq:prA0}, yields the estimates, 
\EQ{
\Big| \frac{b \la'}{\lam} \langle( \lam \p_ \lam \A_0(\la)) w &\mid \LL_\la \dot w\rangle  - 2 \frac{b \la'}{\la} \ang{ (\la\p_\la\A_0(\la)) \dot w \mid \LL_\la w}  \Big|  \\
& \lesssim \frac{b^2}{\lam} ( \| \bs w \|_{\HH^2} + \frac{1}{\lam} \|w \|_{H}) \| \bs w \|_{\HH^2}  + \frac{b\abs{ b + \la'} }{\lam} ( \| \bs w \|_{\HH^2} + \frac{1}{\lam} \|w \|_{H}) \| \bs w \|_{\HH^2}  \\
&  \lesssim \frac{b^2+a^2 + \nu^k}{\lam} \| \bs w \|_{\HH^2}^2   + \frac{b^2 + a^2 + \nu^k}{\lam} \| \bs w \|_{\HH} \| \bs w \|_{\HH^2}^2  \\
&\quad  + \frac{b^2 + a^2 + \nu^k}{\lam^2} \|w \|_{H} \| \bs w \|_{\HH^2}  +     \frac{b^2 + a^2 + \nu^k}{\lam^2} \| \bs w \|_{\HH}^2 \| \bs w \|_{\HH^2}  
}
where in the last line we used~\eqref{eq:b+la'}. The exact same argument again takes care of the third line in~\eqref{eq:vir2-ident}, using now that $[\p_t, \LL_\la] = -\la' \la^{-1} r^{-2} f''(Q_\lam) \Lam Q_\lam$ to give
\EQ{
\Big|  b\ang{ \A_0(\la) w \mid [\p_t, \LL_\la] \dot w}  &- 2 b \ang{\A_0(\la) \dot w \mid [\p_t, \LL_\la] w} \Big| \\
&\lesssim \frac{b^2+a^2 + \nu^k}{\lam} \| \bs w \|_{\HH^2}^2   + \frac{b^2 + a^2 + \nu^k}{\lam} \| \bs w \|_{\HH} \| \bs w \|_{\HH^2}^2  \\
&\quad  + \frac{b^2 + a^2 + \nu^k}{\lam^2} \|w \|_{H} \| \bs w \|_{\HH^2}  +     \frac{b^2 + a^2 + \nu^k}{\lam^2} \| \bs w \|_{\HH}^2 \| \bs w \|_{\HH^2}  
}
Next, we use the same argument together with~\eqref{eq:Psi1-H} and~\eqref{eq:Psi1-H2} to estimate the fourth line in~\eqref{eq:vir2-ident}, 
\EQ{
& \Big|  b \ang{ \A_0(\la) \Psi_1 \mid \LL_\la \dot w}- 2 b \ang{\A_0(\la) \dot w \mid \LL_\la \Psi_1}   \Big| \lesssim \abs{b} \Big( \| \Psi_1 \|_{H^2} + \frac{1}{\lam} \| \Psi_1 \|_{H} \Big) \| \dot w \|_{H}  \\
 & \lesssim \abs{b}   \bigg( \frac{(\abs{a} + \abs{b} ) }{\lam^2} \| \bs w \|_{\HH} +   \frac{(\abs{a} + \abs{b} ) }{\lam^2} \| \bs w \|_{\HH}^2   +(\abs{a} + \abs{b}) \Big( \frac{\nu^{2k-1}+ b^2 \nu^{k-1} +  b^4+ a^4}{\lam^2}   \Big) \bigg) \| \dot w \|_{H}  \\
& \lesssim  \frac{a^6 + b^6 + (\abs{a} + \abs{b})( \abs{b} \nu^{2k-1} + \abs{b}^3  \nu^{k-1})}{\lam^2}  \|\dot w \|_{H} + \frac{a^2 + b^2}{\lam^2} \Big( \| \bs w \|_{\HH} + \| \bs w \|_{\HH}^2\Big)  \| \dot w \|_{H}
 }
 Finally, we estimate the fifth line in~\eqref{eq:vir2-ident} as follows. First, we have 
 \EQ{
 \Big|  b\ang{ \A_0(\la)  w \mid  \LL_\la \ti \Psi_2}  - 2 b \ang{\A_0(\la) \ti \Psi_2 \mid \LL_\la w} \Big| \lesssim \abs{b} \| \ti \Psi_2 \|_H \Big( \| w \|_{H^2} + \frac{1}{\lam} \| w \|_H \Big) 
 }
 Next from the definition of ~$\ti \Psi_2$ in~\eqref{eq:tiPsi2-def2} we have, 
\EQ{ \label{eq:tiPsi2-2} 
\| \ti \Psi_2 \|_H \lesssim  \|-\p_t \dot \Phi + \De \Phi -\frac{1}{r^2} f( \Phi) \|_H  + \| r^{-2}( f( \Phi + w) - f( \Phi) - f'(Q_\lam) w) \|_H
}
By a similar argument used to arrive at the estimate~\eqref{eq:tiPsi2-1} we have 
\EQ{
\| r^{-2}( f( \Phi + w) - f( \Phi) - f'(Q_\lam) w) \|_H \lesssim \Big( 1 + \frac{a^2 + b^2 + \nu^k}{\lam} \Big) \|w \|_{H^2} + \Big( 1 + \| \bs w \|_{\HH}\Big) \| w \|_{H^2}^2 
}
Thus, using the above along with the estimate~\eqref{eq:ptdotPhi-H} to handle the first term on the right of~\eqref{eq:tiPsi2-2} we obtain, 
\EQ{
 \Big|  &b\ang{ \A_0(\la)  w \mid  \LL_\la \ti \Psi_2}  - 2 b \ang{\A_0(\la) \ti \Psi_2 \mid \LL_\la w} \Big| \\
& \lesssim   \Bigg(  \Big( 1 + \frac{a^2 + b^2 + \nu^k}{\lam} \Big) \|w \|_{H^2} + \| w \|_{H^2}^2 +  \| \bs w \|_{\HH} \| w \|_{H^2}^2  \\
&\quad + \frac{ \nu^{2k-1}}{\la^2} + \frac{b^2 \nu^{k-1}}{\la^2}  + \frac{b^4}{\la^2} + \frac{a^4}{\la^2}    +  \frac{\abs{a}  + \abs{b}}{\la^2}   \| \bs w \|_{\HH}     +  \frac{1}{\la^2} \| \bs w \|_{\HH}^2\Bigg)  \Big(\abs{b} \| w \|_{H^2} + \frac{\abs{b}}{\lam} \| w \|_H \Big)  \\
& \lesssim  \frac{\abs{a}^5 + \abs{b}^5 + \abs{b} \nu^{2k-1} + \abs{b}^3 \nu^{k-1}}{\lam^2} \| w \|_{H^2}  + \frac{\abs{a}^5 + \abs{b}^5 + \abs{b} \nu^{2k-1} + \abs{b}^3 \nu^{k-1}}{\lam^2} \frac{\| \bs w \|_{\HH}}{\lam}  \\
&\quad +  \Big( \abs{b} + \frac{ \abs{a}^3 + \abs{b}^3 +  \nu^{\frac{3}{2}k}}{\lam} \Big) \| w \|_{H^2}^2  ++ \Big( \abs{b} + \frac{ \abs{a}^3 + \abs{b}^3 +  \nu^{\frac{3}{2}k}}{\lam} \Big) \frac{ \| \bs w \|_{\HH}}{\lam} \| w \|_{H^2}  \\
&\quad + \abs{b}  \| w \|_{H^2}^3  +  \frac{\abs{b}}{\lam}  \| \bs w \|_{\HH}\| w \|_{H^2}^2 +   \abs{b}   \| \bs w \|_{\HH}\| w \|_{H^2}^3   +  \frac{\abs{b}}{\lam}  \| \bs w \|_{\HH} \| w \|_H\| w \|_{H^2}^2 \\
&\quad + \frac{a^2 + b^2}{\lam^2} \| \bs w \|_{\HH} \| w \|_{H^2} + \frac{ a^2 + b^2}{\lam^2}   \frac{  \| \bs w \|_{\HH}}{\lam} \| w \|_{H} + \frac{\abs{b}}{\lam^2} \| \bs w\|_{\HH}^2 \| w \|_{H^2}   + \frac{\abs{b}}{\lam^3} \| \bs w\|_{\HH}^2 \| w \|_H
}
Combining all of the estimates for the terms on the right-hand side of~\eqref{eq:vir2-ident} yields, 
\EQ{
(b \Vs_{2, \la})'(t)  &\le c_0 \frac{b}{\lam} \| \bs w \|_{\HH^2}^2  \\
&\quad -  2\frac{b}{\lam} \int_0^{R\lam} ( \LL_0 w)^2  \, r\, \ud r  - \frac{b}{\lam} \ang{ w \mid \Ks_\lam w}  - \frac{b}{\la} \int_0^{R \la}\Big( ( \p_r \dot w)^2 + \frac{k^2}{r^2} \dot w^2  \Big) \, \rdr \\
&\quad  - \frac{b}{\lam} \ang{\Lam_0 \dot w \mid P_\la \dot w}+ \frac{b}{\lam} \ang{w \mid P_\lam^2 w}   \\
&\quad + \frac{b}{\lam} \ang{w \mid \La P_\lam \LL_0 w -  \p_r (\La P_\lam) \p_r w + \frac{1}{2}(- \De \La P_\lam) w +  P_\lam \La P_\lam w}  \\
&\quad +C_1 \Big( \frac{\abs{a}^5 + \abs{b}^5 + \abs{b} \nu^{2k-1} + \abs{b}^3 \nu^{k-1}}{\lam^2}\Big)\Big( \| \bs w \|_{\HH^2} +   \frac{\| \bs w \|_{\HH}}{\lam}  \Big) \\
& \quad +C_1  \Big( \abs{a} + \abs{b} +  \frac{b^2+a^2 + \nu^k}{\lam} \Big)\Big(  \| \bs w \|_{\HH^2}+ \frac{ \| \bs w \|_{\HH}}{\lam} + \frac{ \| \bs w \|_{\HH}^2}{\lam^2}\Big)  \| \bs w \|_{\HH^2}   \\
& \quad + C_1 \Big( \abs{a} + \abs{b} \Big)\Big( \| \bs w \|_{\HH^2} + \frac{ \| \bs w \|_{\HH}}{\lam} +  \frac{ \| \bs w \|_{\HH}^2}{\lam^2} \Big) \| \bs w \|_{\HH^2}^2 + C_1\frac{ a^2 + b^2}{\lam^2} \frac{\| \bs w \|_{\HH}^2}{\lam} \\
&\quad   +C_1 \frac{1}{\lam} \| \bs w \|_{\HH}^2 \| \bs w \|_{\HH^2}^2   + C_1\frac{1}{\lam^2} \| \bs w \|_{\HH}^3 \| \bs w\|_{\HH^2}  + C_1\frac{\abs{b}}{\lam^3} \| \bs w\|_{\HH}^3 
}
as claimed. 
 \end{proof}

\subsection{Energy estimates for $\bs w(t)$ in $\bs \Lam^{-1} \HH$} 


Define the weighted energy functional 
\EQ{ \label{eq:E3} 
\E_3(t) := \frac{1}{2} \ang{ \La_0 \dot w \mid \La_0 \dot w} + \frac{1}{2} \ang{ \La w \mid \LL_{\Phi} \La w} 
}

\begin{lem} \label{l:w-E3} There exists a uniform constant $C>0$ with the following property.  Let $\bs w \in \HH \cap  \HH^2 \cap \bs \Lam^{-1} \HH$. 
\EQ{
\frac{1}{2}\| \bs w \|_{\bs \La^{-1} \HH}^2 - C \| w \|_H^2 \le  \E_3(t)  \le \| \bs w \|_{\bs \La^{-1} \HH}^2 + C \| w \|_H^2
}
\end{lem} 
\begin{proof} 
We expand, 
\EQ{
2 \E_3(t)  &= \| \La_0 \dot w \|_{L^2}^2  + \| \Lam w \|_H^2  + \ang{ (f'(\Phi) - k^2)r^{-2} \Lam w \mid \Lam w} 
}
Noting that 
\EQ{
\abs{\ang{ (f'(\Phi) - k^2)r^{-2} \Lam w \mid \Lam w} } \lesssim \| r^{-1} \Lam w \|_{L^2}^2  = \| \p_r w \|_{L^2}^2  \le \| w \|_H^2
}
completes the proof. 
\end{proof} 

The main result of this section is the following lemma. 

\begin{lem}  \label{l:E3} Let $J \subset \R$ be an interval on which $\bs u(t)$ satisfies the hypothesis of Lemma~\ref{l:mod2}.  Assume in addition that $\bs u(t) \in \HH \cap \HH^2 \cap \bs \Lam^{-1} \HH$ for all $t \in J$.  Let $a, b, \la, \mu$  and $\bs w(t) \in  \HH \cap \HH^2 \cap \bs \Lam^{-1} \HH$ be given by Lemma~\ref{l:mod2} and let $\nu:= \la/\mu$ as usual. Then, 
\EQ{
\abs{\E_3'(t) } &\lesssim  \| \La_0 \dot w \|_{L^2}  \| w \|_{H^2}  + \| w \|_{H^2} \| \dot w \|_{L^2} + \frac{\abs{a} + \abs{b}}{\lam} \| w \|_H^2  \\
& \quad + (\abs{a} + \abs{b}) \Big( \frac{\nu^{2k-1}}{\lam} + \frac{b^2 \nu^{k-1}}{\lam} +  \frac{b^4}{\lam} + \frac{a^4}{\lam}  \Big)\Bigg)\big( \| w \|_{H} + \| \Lam w \|_H \big) \\
& \quad +  \frac{ \abs{a} + \abs{b}}{\lam}  \| \bs w \|_{\HH}^2 + \frac{ \abs{a} + \abs{b}}{\lam}  \| \bs w \|_{\HH} \| \Lam w \|_{H}  \\
&\quad +  \Big(  \frac{ \nu^{2k-1}}{\la} + \frac{b^2 \nu^{k-1}}{\la}  + \frac{b^4}{\la} + \frac{a^4\nu}{\la} \Big) \| \La_0 \dot w \|_{L^2}  \\
&\quad +  \frac{\abs{a}  + \abs{b}}{\la}   \| \bs w \|_{\HH}\| \La_0  \dot w \|_{L^2} +   \frac{1}{\la} \| \bs w \|_{\HH}^2  \| \La_0 \dot  w \|_{L^2} +  \| w \|_{H^2} \|w \|_H   \| \La_0 \dot  w \|_{L^2}
}
uniformly on $J$. 
\end{lem} 

The proof of Lemma~\ref{l:E3} requires the following estimates, which we collect here. 
\begin{lem} The following estimates hold true. 
\EQ{ \label{eq:Psi1-LaH} 
\| \Lam(\dot \Phi - \p_t \Phi) \|_H &\lesssim   \frac{(\abs{a} + \abs{b} ) }{\lam} \| \bs w \|_{\HH} +   \frac{(\abs{a} + \abs{b} ) }{\lam} \| \bs w \|_{\HH}^2  \\
&\quad +(\abs{a} + \abs{b}) \Big( \frac{\nu^{2k-1}}{\lam} + \frac{b^2 \nu^{k-1}}{\lam} +  \frac{b^4}{\lam} + \frac{a^4}{\lam}  \Big)
 }
\EQ{ \label{eq:ptdotPhi-Lam0L2} 
\| \Lam_0( - \p_t \dot \Phi + \De \Phi  -\frac{1}{r^2} f(\Phi)  )\|_{L^2} &\lesssim   \frac{ \nu^{2k-1}}{\la} + \frac{b^2 \nu^{k-1}}{\la}  + \frac{b^4}{\la} + \frac{a^4\nu}{\la}   \\
&\quad  +  \frac{\abs{a}  + \abs{b}}{\la}   \| \bs w \|_{\HH}     +  \frac{1}{\la} \| \bs w \|_{\HH}^2 
}
\EQ{\label{eq:f-est-Lam0-L2} 
\| \Lam_0(  r^{-2}( f( \Phi + w) - f( \Phi) - f'( \Phi) w)) \|_{L^2}  & \lesssim  \frac{1}{\lam} \| w \|_{H}^2 +  \| w \|_{H^2} \|w \|_H   
}
\end{lem} 
\begin{proof} The estimates~\eqref{eq:Psi1-LaH} and~\eqref{eq:ptdotPhi-Lam0L2} follow from the same arguments used to prove Lemma~\ref{l:ptdotPhiL2}. The estimate~\eqref{eq:f-est-Lam0-L2} follows from the argument used to prove Lemma~\ref{l:w^2est}. 
\end{proof} 
We now prove Lemma~\ref{l:E3}. 
\begin{proof} 
Applying $\bs \La$ to~\eqref{eq:weq3} we arrive at the system of equations, 
\EQ{ \label{eq:Law} 
\p_t \La w &= \La_0 \dot w - \dot w + \La \Psi_1  \\
\p_t \La_0 \dot w &=   -\LL_\Phi  \La w - \LL_{\Phi} w + [\La , \LL_\Phi] w + \La_0 \Psi_2
}
Differentiating~\eqref{eq:E3} and using~\eqref{eq:Law} we arrive at the identity, 
\EQ{ \label{eq:E3'1} 
\E_3'(t) &= \ang{ \p_t \Lam_0 \dot w \mid \Lam_0  \dot w} + \ang{ \p_t \Lam w \mid \LL_{\Phi} \Lam w} +  \frac{1}{2} \ang{ \La w \mid [\p_t, \LL_\Phi] \La w} \\ 
& = \ang{ -\LL_\Phi  \La w\mid \Lam_0  \dot w} +  \ang{- \LL_{\Phi} w\mid \Lam_0  \dot w}+  \ang{  [\La , \LL_\Phi] w\mid \Lam_0  \dot w}+  \ang{ \La_0 \Psi_2 \mid \Lam_0  \dot w} \\
& \quad + \ang{ \La_0 \dot w \mid \LL_{\Phi} \Lam w} + \ang{ - \dot w  \mid \LL_{\Phi} \Lam w} + \ang{ \La \Psi_1 \mid \LL_{\Phi} \Lam w}  +  \frac{1}{2} \ang{ \La w \mid [\p_t, \LL_\Phi] \La w} \\ 
& =   \ang{  [\La , \LL_\Phi] w\mid \Lam_0  \dot w} - \ang{ [ \LL_{\Phi}, \Lam] w \mid \dot w} + \ang{ \dot w \mid \LL_{\Phi} w}  +  \frac{1}{2} \ang{ \La w \mid [\p_t, \LL_\Phi] \La w} \\
&\quad +  \ang{ \La \Psi_1 \mid \LL_{\Phi} \Lam w} +  \ang{ \La_0 \Psi_2 \mid \Lam_0  \dot w}  
}
For the first term on the right-hand side of~\eqref{eq:E3'1} we have 
\EQ{
\ang{ [  \LL_\Phi,  \La] w \mid \La_0 \dot w }  =  2 \ang{ \La_0 \dot w \mid  \LL_\Phi w}  +  2\ang{ \La_0 \dot w \mid \frac{ k^2 \sin 2 \Phi  \La \Phi}{r^2} w }
}
It follows that 
\EQ{
\abs{\ang{ \La_0 \dot w \mid [  \LL_\Phi,  \La] w} } \lesssim  \| \La_0 \dot w \|_{L^2}  \| w \|_{H^2} 
}
Similarly, 
\EQ{
\abs{\ang{ [ \LL_{\Phi}, \Lam] w \mid \dot w}} + \abs{\ang{ \dot w \mid \LL_{\Phi} w}} \lesssim \| w \|_{H^2} \| \dot w \|_{L^2} 
}
For the fourth term on the right-hand side of~\eqref{eq:E3'1} we have 
\EQ{
\frac{1}{2} \ang{ \La w \mid [\p_t, \LL_\Phi] \La w} = -\frac{1}{2} \ang{ \La w \mid  \frac{2k^2}{r^2} \sin2 \Phi \p_t \Phi  \La w}
}
Note the pointwise estimate, 
\EQ{
\abs{\p_t \Phi} &\lesssim \frac{\abs{b}}{\lam} + \frac{\abs{a} \nu}{\lam} + \frac{\abs{b+\la'}}{\lam} + \frac{ \nu \abs{a- \mu'}}{\lam} + \abs{b} \abs{ b' + \gamma_k \frac{\nu^k}{\lam}} + \abs{a} \abs{ a' + \ti \gamma_k \frac{\nu^{k}}{\mu}}   \lesssim \frac{\abs{a} + \abs{b}}{\lam} 
}
which yields, 
\EQ{
\abs{ \frac{1}{2} \ang{ \La w \mid [\p_t, \LL_\Phi] \La w} } \lesssim  \frac{\abs{a} + \abs{b}}{\lam} \abs{\ang{ \p_r w \mid \p_r w}} \lesssim  \frac{\abs{a} + \abs{b}}{\lam} \| w \|_H^2 
}
We estimate the fifth term as follows, 
\EQ{
\big|  \ang{ \La \Psi_1 \mid \LL_{\Phi} \Lam w} \big|& \lesssim \big( \| \Lam \Psi_1 \|_{H} + \| \Psi_1 \|_H \big)\big( \| w \|_{H} + \| \Lam w \|_H \big)  \\
& \lesssim (\abs{a} + \abs{b}) \Big( \frac{\nu^{2k-1}}{\lam} + \frac{b^2 \nu^{k-1}}{\lam} +  \frac{b^4}{\lam} + \frac{a^4}{\lam}  \Big)\Bigg)\big( \| w \|_{H} + \| \Lam w \|_H \big) \\
& \quad +  \frac{ \abs{a} + \abs{b}}{\lam}  \| \bs w \|_{\HH}^2 + \frac{ \abs{a} + \abs{b}}{\lam}  \| \bs w \|_{\HH} \| \Lam w \|_{H} 
}
where in the last line we used~\eqref{eq:Psi1-LaH} and~\eqref{eq:Psi1-H}. 

Finally, for the sixth term on the right-hand side of~\eqref{eq:E3'1} we have 
\EQ{
&\big|\ang{ \La_0 \Psi_2 \mid \Lam_0  \dot w}  \big| \lesssim \| \La_0 \Psi_2 \|_{L^2} \| \La_0 \dot w \|_{L^2} \\
& \lesssim \Big( \| \Lam_0( - \p_t \dot \Phi + \De \Phi  -\frac{1}{r^2} f(\Phi)  )\|_{L^2} + \| \Lam_0(  r^{-2}( f( \Phi + w) - f( \Phi) - f'( \Phi) w)) \|_{L^2}  \Big) \| \La_0 \dot  w \|_{L^2}  \\
& \lesssim  \Big(  \frac{ \nu^{2k-1}}{\la} + \frac{b^2 \nu^{k-1}}{\la}  + \frac{b^4}{\la} + \frac{a^4\nu}{\la} \Big) \| \La_0 \dot w \|_{L^2}  \\
&\quad +  \frac{\abs{a}  + \abs{b}}{\la}   \| \bs w \|_{\HH}\| \La_0  \dot w \|_{L^2} +   \frac{1}{\la} \| \bs w \|_{\HH}^2  \| \La_0 \dot  w \|_{L^2} +  \| w \|_{H^2} \|w \|_H   \| \La_0 \dot  w \|_{L^2} 
}
Combining these estimates for the terms on the right-hand side of~\eqref{eq:E3'1} we arrive at the estimate, 
\EQ{
\abs{ \Es_3'(t)} &\lesssim  \| \La_0 \dot w \|_{L^2}  \| w \|_{H^2}  + \| w \|_{H^2} \| \dot w \|_{L^2} + \frac{\abs{a} + \abs{b}}{\lam} \| w \|_H^2  \\
& \quad + (\abs{a} + \abs{b}) \Big( \frac{\nu^{2k-1}}{\lam} + \frac{b^2 \nu^{k-1}}{\lam} +  \frac{b^4}{\lam} + \frac{a^4}{\lam}  \Big)\Bigg)\big( \| w \|_{H} + \| \Lam w \|_H \big) \\
& \quad +  \frac{ \abs{a} + \abs{b}}{\lam}  \| \bs w \|_{\HH}^2 + \frac{ \abs{a} + \abs{b}}{\lam}  \| \bs w \|_{\HH} \| \Lam w \|_{H}  \\
&\quad +  \Big(  \frac{ \nu^{2k-1}}{\la} + \frac{b^2 \nu^{k-1}}{\la}  + \frac{b^4}{\la} + \frac{a^4\nu}{\la} \Big) \| \La_0 \dot w \|_{L^2}  \\
&\quad +  \frac{\abs{a}  + \abs{b}}{\la}   \| \bs w \|_{\HH}\| \La_0  \dot w \|_{L^2} +   \frac{1}{\la} \| \bs w \|_{\HH}^2  \| \La_0 \dot  w \|_{L^2} +  \| w \|_{H^2} \|w \|_H   \| \La_0 \dot  w \|_{L^2}
}
as claimed. 
\end{proof}

\subsection{Final construction of $\bs u_c$ and proof of Theorem~\ref{t:refined} } \label{s:construct}
In this section we use the estimate on the modulations parameters from Lemma~\ref{l:modc2} together with the energy estimates proved in the previous three subsections to construct the blow up solution described in Theorem~\ref{t:refined}. 

We define the formal modulation parameters as the unique solution to the nonlinear system, 
\EQ{ \label{eq:mod-sys} 
\mu_{\app}'(t) &= a_{\app}(t)\\
\lam_{\app}'(t) &= -b_{\app}(t) \\
a_{\app}'(t)& = - \gamma_k \lam_{\app}(t)^k \mu_{\app}(t)^{-k-1} \\
b_{\app}'(t) & = - \gamma_k \lam_{\app}(t)^{k-1} \mu_{\app}(t)^{-k} 
 }
 To solve this system we note that the explicit quadruplet $(\ti \mu_{\app}, \ti \lam_{\app}, \ti a_{\app}, \ti b_{\app})$ given by, 
\EQ{
\ti \lam_{\app}(t)&:= q_k t^{-\frac{2}{k-2}} \\
\ti b_{\app} (t) &:=  - \ti \la_{\app}'(t)  =  \rho_k  \ti \la_{\app}(t)^{\frac{k}{2}}  \\
\ti \mu_{\app} (t) &:= 1 -\frac{k}{2(k+2)} \ti \lam_{\app}(t)^2 \\
\ti a_{\app}(t) &:=  \ti \mu_{\app}'(t) = \frac{k}{k+2} \rho_k\ti  \lam_{\app}^{\frac{k}{2} + 1}
}
solves the approximate system, 
\EQ{
\ti \mu_{\app}'(t) &= \ti a_{\app}(t)\\
\ti \lam_{\app}'(t) &= - \ti b_{\app}(t) \\
\ti a_{\app}'(t)& = - \gamma_k \ti \lam_{\app}(t)^k \\
\ti b_{\app}'(t) & = - \gamma_k \ti  \lam_{\app}(t)^{k-1}
}
where we remind the reader of the explicit constants, 
 \EQ{
 q_k &:= \left( \frac{k-2}{2}\right)^{-\frac{2}{k-2}} \rho_k^{-\frac{2}{k-2}}, \quad 
 \rho_k = \Big( \frac{8k}{\pi}  \sin( \pi/k) \Big)^{\frac{1}{2}} , \quad \gamma_k = \frac{k}{2} \rho_k^2.
 }
 By standard ODE arguments, one may then construct a unique solution to the system~\eqref{eq:mod-sys} satisfying, 
 \EQ{ \label{eq:mod-form-asy} 
 \la_{\app}(t) &= \ti \lam_{\app}(t) ( 1+ O(  \ti \lam_{\app}(t)^{2-\eps}) )   \mas t \to \infty\\
 \mu_{\app}(t) &= \ti \mu_{\app}(t) + O(  \ti \lam_{\app}(t)^{3- \eps})   \mas t \to \infty\\
 a_{\app}(t) & = \ti a_{\app}(t)  + O(  \ti \lam_{\app}(t)^{\frac{k}{2} + 2-\eps})  \mas t \to \infty \\
 b_{\app}(t) & = \ti b_{\app}(t) + O(  \ti \lam_{\app}(t)^{\frac{k}{2} + 2-\eps})\mas t \to \infty
 }
 where $\eps>0$ is a constant that can be chosen as small as we like. 
We define the ansatz, 
\EQ{
\bs \Phi_{\app}(t,  \cdot) :=  \bs \Phi( \mu_{\app}(t), \lam_{\app}(t), a_{\app}(t), b_{\app}(t)) 
}
where $\bs\Phi( \mu, \lam, a, b)$ is defined in~\eqref{eq:Phidef}.  Note that for $t_0>0$ sufficiently large we have $ \E( \bs \Phi_{\app}(t_0))< \infty$ and in fact,  
\EQ{\label{eq:d0} 
 \bfd_+( \bs \Phi_{\app}(t_0))  \lesssim b_{\app}(t_0)^2  +  \la_{\app}(t_0)^{k}  + \abs{b_{\app}(t_0)} \lam_{\app}(t_0)^{\frac{k}{2} - 1} \lesssim t_0^{-\frac{2(k-1)}{k-2}} 
}
Let $\bs u_0(t)$ denote the solution to~\eqref{eq:wmk} with initial data at time $t_0 \gg1$ given 
\EQ{
\bs u_{0}(t_0)  = \bs \Phi_{\app}(t_0) \in \HH \cap \HH^2 \cap \bs \La^{-1} \HH
}
and denote by $J_{0, \max} \ni t_0$ the  maximal interval of existence of $\bs u_0(t)$.  It follows from~\eqref{eq:d0} that for each $t_0>0$ large enough, there exists an interval $J_0$ with $t_0 \in J_0 \subset J_{\max, 0}$ on which the conditions of Lemma~\ref{l:mod2} are satisfied. Thus on $J$ we may define modulation  parameters, $\mu(t), \lam(t), a(t), b(t)$ and $ \bs w \in \HH \cap \HH^2 \cap  \bs\Lam \HH$ so that 
\EQ{ \label{eq:w-def} 
\bs u_0(t) = \bs \Phi( \mu(t), \la(t), a(t), b(t)) + \bs w(t)
}
where 
\EQ{
0 = \ang{ \La Q_{\U \mu} \mid w}= \ang{ \La Q_{\U \lam} \mid w} = \ang{ \La Q_{\U \mu} \mid \dot w}= \ang{ \La Q_{\U \lam} \mid \dot w} 
}
 Note that at time $t_0$ we have 
 \EQ{
 \bs w(t_0) = 0  \mand (\mu(t_0), \lam(t_0), a(t_0), b(t_0) )  = (\mu_{\app} (t_0), \lam_{\app} (t_0), a_{\app} (t_0), b_{\app}(t_0) ) 
  }

 The main ingredient in the proof of Theorem~\ref{t:refined} is the following proposition, which will be proved by a bootstrap argument. 
 
 \begin{prop} \label{p:boot}  There exists a time $T_0 >0$ sufficiently large with the following property. For any $t_0> T_0$ denote by $\bs u_{0}(t)$ the unique solution to~\eqref{eq:wmk} with initial data at time $t_0> T_0$ given by $\bs u_0(t_0) = \bs  \Phi_{\app}(t_0)$. Let $J_0 \subset J_{\max, 0} $ denote the time interval on which the modulation parameters $\mu(t), \lam(t), a(t), b(t)$ are defined as above. Then $[T_0, t_0] \subset J_0 \subset J_{\max, 0}$ and we have the estimates, 
 \EQ{
 \abs{ \la(t) - \la_{\app}(t)} \lesssim \lam_{\app}(t)^{k-1}  &, \quad  \abs{ \mu(t) - \mu_{\app}(t) }  \lesssim  \lam_{\app}(t)^{k}    \\
 \abs{ a(t) - a_{\app}(t)} \lesssim   a_{\app}(t) \lam_{\app}(t)^{k-2} &,   \quad 
 \abs{ b(t) - b_{\app}(t)} \lesssim b_{\app}(t) \lam_{\app}(t)^{k-2} 
 }
 uniformly on $t \in [T_0, t_0]$.  Moreover, the estimates,  
 \EQ{
 \|  \bs u_{0}(t) -  \bs \Phi(\mu(t), \lam(t), a(t), b(t)) \|_{\HH} &\lesssim  \lam_{\app}(t)^{\frac{3k}{2} - 1}  \\
 \|  \bs u_{0}(t) -  \bs\Phi(\mu(t), \lam(t), a(t), b(t)) \|_{\HH^2} &\lesssim  \lam_{\app}(t)^{\frac{3k}{2} - 2}  \\
 \| \bs \Lam \bs u_{0}(t) - \bs \Lam \Phi(\mu(t), \lam(t), a(t), b(t)) \|_{\HH} &\lesssim  \lam_{\app}(t)^{k-1}
 }
 hold uniformly on $t \in [T_0, t_0]$, i.e, with a constant independent of $t, t_0$. 
 \end{prop}   
 
 \begin{cor} \label{c:boot} 
 Let $T_0$ be as in Proposition~\ref{p:boot}.  Let $t_0>T_0$ and denote by $\bs u_{0}(t)$ the unique solution to~\eqref{eq:wmk} with initial data at time $t_0> T_0$ given by $\bs u_0(t_0) = \bs  \Phi_{\app}(t_0)$. Then, the estimates, 
 \EQ{
 \|  \bs u_{0}(t) -  \bs \Phi_{\app}(t) \|_{\HH} &\lesssim  t^{-1}   \\
 \|  \bs u_{0}(t) - \bs \Phi_{\app}(t)  \|_{\HH^2} &\lesssim  t^{-1}   \\
 \| \bs \Lam \bs u_{0}(t) - \bs \Lam \bs \Phi_{\app}(t)  \|_{\HH} &\lesssim  t^{-1} 
 }
 hold uniformly on $t \in [T_0, t_0]$, i.e, with a constant independent of $t, t_0$.
 \end{cor}

\begin{proof}[Proof of Proposition~\ref{p:boot}]
The proof will proceed via a bootstrap argument. Let $T_0>0$ be large enough so that for all $t_0>T_0$ the conditions of Lemma~\ref{l:mod2} hold for~$\bs u(t_0)$ at $t_0$, and thus on some (possibly) small interval containing $t_0$. Let $C_0>0$ be a constant to be fixed below, and let $T_1<t_0$ be the smallest time such that $T_1 \ge T_0$ and such that 
\EQ{ \label{eq:mod-boot} 
\abs{\lam(t) -\lam_{\app}(t)}  \le 2 C_0  \lam_{\app}(t)^{k-1} &, \quad 
\abs{ \mu(t)- \mu_{\app}(t)} \le 2 C_0 \lam_{\app}(t)^{k}\\
\abs{ a(t) - a_{\app}(t)} \le   2 C_0  k \abs{a_{\app}(t)} \lam_{\app}(t)^{ k -2}   &, \quad 
 \abs{ b(t) - b_{\app}(t)} \le   2 C_0 \frac{k}{2}\abs{b_{\app}(t)} \lam_{\app}(t)^{ k -2}   
 } 
 and 
 \EQ{ \label{eq:w-boot} 
 \| \bs w(t) \|_{\HH}^2& \le 2C_0^2  \lam_{\app}(t)^{3k -2} , \\
  \| \bs w(t) \|_{\HH^2}^2 &\le 2C_0^2  \lam_{\app}(t)^{3k -4} , \\
  \| \bs w(t) \|_{\bs \Lam \HH}^2 &\le 2C_0^2  \lam_{\app}(t)^{2k-2}
}
A standard argument shows that the solution along with the modulation parameters are well defined on $[T_1, t_0] \subset J_0$. The goal is to obtain a strict improvement to these estimates, which proves that $T_1 = T_0$. 

First, we obtain improvements to the estimates~\eqref{eq:w-boot}. 
 Applying the bootstrap assumptions to the right-hand side of the estimate for $\Hs_1'(t)$ in Proposition~\ref{p:Hs1'} we arrive at the estimate,  
\EQ{
\Hs_1'(t) &\ge   \frac{b(t)}{\lambda(t)}\int_0^{R\lambda(t)}\Big((\partial_r w)^2 + \frac{k^2}{r^2}w^2\Big)(t) \udr + \frac{b(t)}{\lam(t)} \int_0^\infty (f'(Q_{\lam(t)}) - k^2) \frac{w(t)^2}{r^2} \, \udr \\
&\quad   -c_0 \frac{\abs{b_{\app}(t)}}{\la_{\app}(t)} \| w(t) \|_H^2   - C_1 \lam_{\app}(t)^{2k-2} \| \bs w (t) \|_{\HH}  \\
& \ge - c_0 C_0^2 \lam_{\app}(t)^{\frac{7}{2} k - 3} - C_1 C_0 \lam_{\app}(t)^{\frac{7}{2} k - 3} \ge - c_0 C_0^2 \lam_{\app}(t)^{\frac{7}{2} k - 3}
}
where the second-to-last estimate follows from the localized coercivity estimate~\eqref{eq:L-loc1} after rescaling it by $\lam(t)$ and observing that the bootstrap assumptions~\eqref{eq:mod-boot} imply that $b(t)>0$ on $[T_1, t_0]$. The last estimate is achieved by taking $C_0$ large enough relative to $C_1$. Using Lemma~\ref{l:w-E1-H1}, the above estimate for $\Hs_1'(t)$ and the fact that $\Hs_1(t_0) = 0$ we have, for any $t \in [T_1, t_0]$, 
\EQ{
\| \bs w(t) \|_{\HH}^2 & \le C \Hs_1(t)  =  - C\int_{t}^{t_0}  \Hs_1'(s) \, \ud s \\
& \le c_0 C C_0^2 \int_{t}^{t_0}  \lam_{\app}(s)^{\frac{7}{2} k - 3}  \le c_0 2 C_0^2   \lam_{\app}(t)^{3k-2} 
}
where we have allowed the value of $c_0$ to change from line to line, but still noting it can be taken as small as we like in Proposition~\ref{p:Hs1'} independently of the relevant constants in this proof.  We have improved the first line in~\eqref{eq:w-boot}. The improvement in the second line of~\eqref{eq:w-boot} is obtained in a nearly identical manner, using this time the estimate in Proposition~\ref{p:H2'} along with~\eqref{eq:E2H2} and  the localized coercivity estimates~\eqref{eq:L-loc1} and~\eqref{eq:L2-loc1} on the right-hand side of the estimate~\eqref{eq:H2'}. Finally, the last estimate in~\eqref{eq:w-boot} follows from Lemma~\ref{l:w-E3}, Lemma~\ref{l:E3} and the bootstrap assumption for the modulation parameters,~\eqref{eq:mod-boot}. 

Next, we improve the esimates~\eqref{eq:mod-boot}. Under the bootstrap assumptions we obtain the following estimates on the interval $[T_1, t_0]$, 
\EQ{ \label{eq:mod-boot2} 
\abs{\la'(t) + b(t)} &\lesssim \lam_{\app}(t)^{2k-1}\\ 
\abs{ a(t) - \mu'(t)} &\lesssim \lam_{\app}(t)^{2k-1} \\ 
\abs{ b'(t) + \gamma_k \frac{\nu(t)^k}{\lam(t)}} &\lesssim \lam_{\app}(t)^{2k-2} \\
\abs{ a'(t) +  \gamma_k \frac{\nu(t)^k}{\mu(t)}} &\lesssim \lam_{\app}(t)^{2k-2}
}

First, using the bootstrap assumptions~\eqref{eq:mod-boot} along with~\eqref{eq:mod-boot2} we have, 
\EQ{
\abs{ \lam(t) - \lam_{\app}(t)}& \le \int_t^{t_0} \abs{\la'(\tau) + b_{\app}(\tau)} \, \ud \tau \\
& \le\int_t^{t_0} \abs{b(\tau) - b_{\app}(\tau)} \, \ud \tau + \int_t^{t_0} \abs{\la'(\tau) + b(\tau)} \, \ud \tau  \\
& \le 2 C_0 (1+o(1))\int_t^{t_0}  \frac{k}{2} \rho_k \lam_{\app}(\tau)^{\frac{3}{2} k -2} \, \ud \tau  \\
& \le \frac{k}{k-1}(1+o(1)) C_0  \lam_{\app}(t)^{k-1}  
}
where $o(1)$ means a constant that can me made as small as we like by taking $T_0$ large.  This is a strict improvement to the first inequality in~\eqref{eq:mod-boot} as long as $T_0$ is chosen large enough so that $\frac{k}{k-1}(1+o(1))<2$. 
Similarly, 
\EQ{
\abs{\mu(t) - \mu_{\app}(t)} & \le \int_t^{t_0} \abs{a(\tau) - a_{\app}(\tau)} \ud \tau + \int_{t}^{t_0} \abs{ a(\tau) - \mu'(\tau)} \, \ud \tau  \\
& \le 2 C_0( 1+ o(1)) \int_{t_0}^t \frac{k}{k+2}k   \rho_k \lam_{\app}(\tau)^{\frac{3}{2} k - 1} \, \ud \tau \\ 
& \le  \frac{k}{k+2} ( 1+ o(1)) 2 C_0 \lam_{\app}(t)^k 
}
which strictly improves the second inequality in~\eqref{eq:mod-boot}. Next, 
\EQ{
\abs{b(t) - b_{\app}(t)}  &\le \int_t^{t_0} \abs{b'(s) - b_{\app}'(s)} \, \ud s  \\
& \le \int_t^{t_0}\gamma_k \Big|\frac{\lam(s)^{k-1}}{\mu(s)^k} - \frac{\lam_{\app}(s)^{k-1}}{\mu_{\app}(s)^{k}}\Big| \, \ud s \,  \ud \tau + \int_t^{t_0}\ \Big|b'(s) + \gamma_k\frac{\nu(s)^k}{\lam(s)}\Big| \, \ud s   \\
& \le \int_t^{t_0}\gamma_k \mu(s)^{-k}\Big|\lam(s)^{k-1} - \lam_{\app}(s)^{k-1}\Big| \, \ud s  \\
&\quad + \int_t^{t_0}\gamma_k \lam_{\app}(s)^{k-1}\Big|\frac{1}{\mu(s)^k} -\frac{1}{\mu_{\app}(s)^k}\Big| \, \ud s  
 + \int_t^{t_0} \Big|b'(s) + \gamma_k\frac{\nu(s)^k}{\lam(s)}\Big| \, \ud s 
}
Note that for $\ell \ge 1$. 
\EQ{
 \abs{\lam^{ \ell} -\lam_{\app}^{\ell}}  &\le  \ell(1+ o(1))  \abs{\lam - \lam_{\app}}\lam_{\app}^{\ell-1} \le 2 C_0 \ell (1+ o(1))\lam_{\app}^{k+ \ell -2}  \\
 \abs{\frac{1}{\mu^{\ell}} - \frac{1}{\mu_{\app}^{\ell}}} &\le \ell (1 + o(1)) \abs{\mu- \mu_{\app}}  \le 2 C_0 \ell (1+o(1)) \lam_{\app}^k 
}
Using these inequalities and the estimate~\eqref{eq:mod-boot2} above along with the fact that $\gamma_k = k/2 \rho_k^2$ we have, 
\EQ{
\abs{b(t) - b_{\app}(t)}  &\le  2C_0(1+ o(1)) \frac{k(k-1)}{2} \rho_k^2 \int_{t}^{t_0} \lam_{\app}(s)^{2k-3}  \, \ud s  \\
&\quad + 2C_0(1+ o(1)) \frac{k^2}{2} \rho_k^2 \int_{t}^{t_0} \lam_{\app}(s)^{2k-1}  \, \ud s  + C\int_{t_0}^t   \lam_{\app}(s)^{2k-2} \, \ud s \\
& \le 2C_0(1+ o(1)) \frac{k(k-1)}{2} \rho_k^2 \int_{t}^{t_0} \lam_{\app}(s)^{2k-3}  \, \ud s \\
& \le 2 C_0( 1+o(1)) \frac{k-1}{\frac{3}{2} k - 2} \frac{k}{2}    b_{\app}(t) \lam_{\app}(t)^{k-2}
}
which is a strict improvement over the bootstrap assumption for $b(t)$ as long as $T_0$ is chosen large enough. Finally, we estimate, 
\EQ{
\abs{a(t) - a_{\app}(t)} &\le \int_{t_0}^t  \abs{a'( s) - a_{\app}'(s)} \, \ud s  \\
& \le\int_{t_0}^t  \gamma_k \abs{\frac{\lam(s)^k}{\mu(s)^{k+1}} - \frac{\lam_{\app}(s)^k}{\mu_{\app}(s)^{k+1}}} \, \ud s + \int_{t_0}^t  \abs{a'( s) + \gamma_k  \frac{\nu(s)^k}{\mu(s)}} \, \ud s \\
& \le 2 C_0(1+ o(1)) \frac{k^2}{2} \rho_k^2 \int_{t_0}^t \lam_{\app}(s)^{2k-1} \, \ud s  \\
& \le \frac{k+2}{3k-2} (1+ o(1)) 2 C_0 k a_{\app}(t) \lam_{\app}(t)^{ k -2} 
}
which improves the bootstrap assumption for $a(t)$. We conclude that in fact, $T_1 = T_0$. 
\end{proof} 

\begin{proof}[Proof of Corollary~\ref{c:boot}]
For the first estimate, we simply apply the conclusions of Proposition~\ref{p:boot} along with the estimate, 
\EQ{
\| \bs  \Phi(\mu(t), \lam(t), a(t), b(t)) - \bs \Phi_{\app}(t) \|_{\HH} & \lesssim  \abs{ \frac{\lam(t)}{\lam_{\app}(t)} - 1}^{\frac{1}{2}} +  \abs{ \frac{\mu(t)}{\mu_{\app}(t)} - 1}^{\frac{1}{2}}  \\
&\quad + \abs{ a(t) - a_{\app}(t)}^{\frac{1}{2}} + \abs{b(t) - b_{\app}(t)}^{\frac{1}{2}} 
}
The corresponding estimates in $\HH^2$ and $\bs \Lam^{-1} \HH$ are identical. 
\end{proof} 

\begin{proof}[Proof of Theorem~\ref{t:refined}] 
We begin by constructing the solution using a version of the classical weak convergence/compactness argument introduced by Martel~\cite{Martel05} and Merle~\cite{Merle90}.  Let $t_n \to \infty$ be the sequence defined by $t_0 = T_0$, $t_n := 2^{2n} T_0$. Let $\bs u_n(t)$ denote the solution to~\eqref{eq:wmk} with initial data at time $t_n$ given by 
$
\bs u(t_n) := \bs \Phi_{\app}(t_n)
$. 
By Corollary~\ref{c:boot} we have, for all $m \ge 1$, the estimate
\EQ{ \label{eq:mt12} 
\sup_{t \in [t_{m-1}, t_{m}]} t^{\frac{1}{2}}  \| \bs u_m(t) - \bs \Phi_{\app}(t) \|_{\HH \cap \HH^2 \cap \bs \Lam^{-1} \HH} \le 2^{-(m-1)} T_0^{-\frac{1}{2}}
}
and in particular for all $n \ge 1$, 
\EQ{
\sup_{t \in [T_0, t_{n}]} t^{\frac{1}{2}}  \| \bs u_n(t) - \bs \Phi_{\app}(t) \|_{\HH \cap \HH^2 \cap \bs \Lam^{-1} \HH} \le  2 T_0^{-\frac{1}{2}}
}
By taking $T_0>0$ large enough we can (after possibly extracting a subsequences) find a solution $\bs u_c(t) \in \HH \cap \HH^2 \cap \bs \Lam^{-1} \HH$ to~\eqref{eq:wmk} on the time interval $[T_0, \infty)$  so that $\bs u_n(t) \rightharpoonup \bs u_c(t)$ in $\HH \cap \HH^2 \cap \bs \Lam^{-1} \HH$  for all $t \in [T_0, \infty)$; see~\cite[Corollary A6]{JJ-AJM} for details. It follows from weak convergence and the estimate~\eqref{eq:mt12} that the solution $\bs u_c(t)$ satisfies, 
\EQ{ \label{eq:t12} 
 \| \bs u_c(t) - \bs \Phi_{\app}(t) \|_{\HH \cap \HH^2 \cap \bs \Lam^{-1} \HH}  \lesssim t^{-\frac{1}{2}}  
}
for all $t \in [T_0, \infty)$. 

The function $\bs u_c(t)$ will be our desired 2-bubble solution, but we must improve the estimate~\eqref{eq:t12}. First, we apply the modulation lemma, i.e. Lemma~\ref{l:mod2} to the solution $\bs u_c(t)$ on any interval $[T, \infty)$ with $T \ge T_0$. Using the estimate~\eqref{eq:t12} and the estimate~\eqref{eq:mod-bound} from the proof of Lemma~\ref{l:mod2} we obtain modulation parameters $(\mu_c(t), \lam_c(t), a_c(t), b_c(t))$ satisfying, 
\EQ{ \label{eq:mod-uc}
\abs{ \frac{\lam_c(t)}{\lam_{\app}(t)} - 1}  &= o(1)  \mas t \to \infty \\
\abs{ \frac{\mu_c(t)}{\mu_{\app}(t)} - 1}  &= o(1)  \mas t \to \infty \\
\abs{a_c(t) - a_{\app}(t) } &\lesssim o(1)  \lam_{\app}(t)^{\frac{k}{2}}  \mas t  \to \infty\\
\abs{ b_c(t) - b_{\app}(t) } &\lesssim o(1) \lam_{\app}(t)^{\frac{k}{2}} \mas t \to \infty
}
Note that although we have $a_{\app}(t) \simeq \lam_{\app}(t)^{\frac{k}{2} +1}$ the bounds above only yield the preliminary estimate $\abs{a_c(t)} =o(1) \lam_{\app}(t)^{\frac{k}{2}}$, but this will be sufficient for our purposes. 
Moreover, for $$\bs w_c(t) = \bs u_c(t) - \bs \Phi( \mu_c(t), \lam_c(t), a_c(t), b_c(t))$$ we have the preliminary  \emph{quantitative} estimate, 
\EQ{ \label{eq:wc-bound1} 
 \| w_c(t)  \|_{H} + \nu_c(t)^{-\frac{k}{2}} \| \dot w_c \|_{L^2} \lesssim  \bfd_+( \bs u_c (t)) \lesssim t^{-\frac{1}{2}} 
}
given by Lemma~\ref{l:mod2} and the estimate~\eqref{eq:t12} -- here $\nu_c(t) = \lam_c(t)/ \mu_c(t)$ as usual. Now  that we have estimates for the sizes $\mu_c(t) \simeq 1$, $\lam_c(t) \simeq t^{-\frac{2}{k-2}}$, $\abs{a_c(t)} \lesssim t^{-\frac{k}{k-2}}$ and $b_c(t) \simeq t^{-\frac{k}{k-2}}$, and the bound~\eqref{eq:wc-bound1}, we use Proposition~\ref{p:Hs1'} to obtain an improvement. Indeed, by Lemma~\ref{l:w-E1-H1} we have, 
\EQ{ \label{eq:wc-bound2} 
\| \bs w_c(t) \|_{\HH}^2 \lesssim -\int_t^{\infty}  \Hs_1'(s) \, \ud s
}
Next, using that $b_c(t)>0$ on $[T_0, \infty)$ we can use the localized coercivity estimate~\eqref{eq:L-loc1} on the right-hand side of~\eqref{eq:Hs1'} in Proposition~\ref{p:Hs1'} to obtain the estimate, 
\EQ{
-\Hs_1'(s)  \le \frac{c_0}{s}  \| \bs w \|_{\HH}^2 + C s^{-\frac{2}{k-2}(\frac{7}{2} k - 3)} 
}
where $c_0>0$ is a constant that can be made as small as we like. Note that above we have used crucially the preliminary quantitative estimate $\| \dot w_c(t) \|_{L^2} \lesssim \nu_c(t)^{\frac{k}{2}} \lesssim t^{-\frac{k}{k-2}}$ from~\eqref{eq:wc-bound1} and~\eqref{eq:mod-uc} to control several of the terms on the right-hand side of~\eqref{eq:Hs1'}. Inserting the above into~\eqref{eq:wc-bound2} we obtain, 
\EQ{
\| \bs w_c(t) \|_{\HH}^2 \lesssim t^{-\frac{2}{k-2}( 3k-2) } + c_0 \int_t^\infty s^{-1} \| \bs w_c(s) \|_{\HH}^2 \, \ud s 
}
Multiplying through by $t^{\frac{1}{2}}$ we have 
\EQ{
t^{\frac{1}{2}}\| \bs w_c(t) \|_{\HH}^2 \lesssim t^{-\frac{2}{k-2}( 3k-2) + \frac{1}{2} } + c_0 t^{\frac{1}{2}}\int_t^\infty s^{-1} s^{-\frac{1}{2}} s^{\frac{1}{2}} \| \bs w_c(s) \|_{\HH}^2 \, \ud s 
}
Defining $f(t) := \sup_{\tau \in [t, \infty)}\tau^{\frac{1}{2}}\| \bs w_c(\tau) \|_{\HH}^2$, we know by the preliminary quantitative estimate~\eqref{eq:wc-bound1} that $f(t) \lesssim 1$.  The above then yields the inequality, 
\EQ{
f(t) \lesssim t^{-\frac{2}{k-2}( 3k-2) + \frac{1}{2} }  + c_0 f(t) 
}
which implies that $f(t) \lesssim t^{-\frac{2}{k-2}( 3k-2) + \frac{1}{2} }$ by taking $c_0>0$ small enough.  In other words, we have proved the bound 
\EQ{ \label{eq:wc-bound3} 
\| \bs w_c(t) \|_{\HH} \lesssim \lam_{\app}(t)^{3k-2}
}
as desired. In the same fashion we obtain the claimed bounds 
\begin{align}
\| \bs w_c(t) \|_{\HH^2}^2 &\lesssim \lam_{\app}(t)^{3k-4} \label{eq:wc-bound-H2}    \\
  \| \bs w_c(t) \|_{\bs \Lam^{-1} \HH}^2 &\lesssim \lam_{\app}(t)^{2k-2} \label{eq:wc-bound-LaH} 
\end{align} 
using now the established estimate~\eqref{eq:wc-bound3} as input into Proposition~\ref{p:H2'} to prove~\eqref{eq:wc-bound-H2} and then~\eqref{eq:wc-bound3} and~\eqref{eq:wc-bound-H2} as input into Lemma~\ref{l:w-E3} and Lemma~\ref{l:E3} to prove~\eqref{eq:wc-bound-LaH}. 

Finally, we improve~\eqref{eq:mod-uc} establish the refined bounds~\eqref{eq:mod-est} and~\eqref{eq:mod'}. Inserting the estimates ~\eqref{eq:wc-bound3} along with~\eqref{eq:mod-uc} into Lemma~\ref{l:modc2} we see that $(\mu_c(t), \lam_c(t), a_c(t), b_c(t))$ satisfy the differential inequalities, 
\EQ{
\abs{\la_c'(t) + b_c(t)} + \abs{ a_c(t) - \mu_c'(t)}  &\lesssim \lam_{\app}(t)^{2k-1},  \\
\abs{ b_c'(t) + \gamma_k \frac{\nu_c(t)^k}{\lam_c(t)}} + \abs{ a_c'(t) + \gamma_k \frac{\nu_c(t)^k}{\mu_c(t)}} &\lesssim \lam_{\app}(t)^{2k-2} 
}
An ODE argument together with the estimates~\eqref{eq:mod-form-asy} for $(\mu_{\app}(t), \lam_{\app}(t), a_{\app}(t), b_{\app}(t))$ lead to~\eqref{eq:mod-est}. 
This completes the proof of Theorem~\ref{t:refined}. 
\end{proof}

\bibliographystyle{plain}
\bibliography{researchbib}

\bigskip
\centerline{\scshape Jacek Jendrej}
\smallskip
{\footnotesize
 \centerline{CNRS and LAGA, Université Sorbonne Paris Nord}
\centerline{99 av Jean-Baptiste Cl\'ement, 93430 Villetaneuse, France}
\centerline{\email{jendrej@math.univ-paris13.fr}}
} 
\medskip 
\centerline{\scshape Andrew Lawrie}
\smallskip
{\footnotesize
 \centerline{Department of Mathematics, Massachusetts Institute of Technology}
\centerline{77 Massachusetts Ave, 2-267, Cambridge, MA 02139, U.S.A.}
\centerline{\email{alawrie@mit.edu}}
}

\end{document}